\documentclass{amsart}
\usepackage{amssymb,enumerate,amsmath,amsthm,mathtools,tikz-cd,footmisc,footnote}
\usepackage[hide links]{hyperref}

\makesavenoteenv{tabular}

\newtheorem{theorem}{Theorem}
\newtheorem{question}{Problem}
\newtheorem{conjecture}[theorem]{Conjecture}

\newtheorem{challenge}[question]{Problem}
\newtheorem{lemma}[theorem]{Lemma}
\newtheorem{proposition}[theorem]{Proposition}
\newtheorem{corollary}[theorem]{Corollary}

\theoremstyle{definition}
\newtheorem{definition}[theorem]{Definition}

\newcommand{\e}{\epsilon}

\newcommand{\Z}{\mathcal Z}
\newcommand{\Cu}{\mathrm{Cu}}

\makeatletter
\AtEndDocument{\write\@auxout{\protect\total@problems{\Roman{question}}}}
\def\total@problems#1{\global\def\totalproblems{#1}}
\total@problems{0}
\makeatother

\usepackage{perpage}
\newcounter{mparcnt}
\MakePerPage{mparcnt}

\begin{document}

\author[C. Schafhauser]{Christopher Schafhauser}
\address{\hspace{.5ex}Christopher Schafhauser, Department of  Mathematics, University of Nebraska-\linebreak Lincoln, Lincoln, NE 68588, United States}
\email{cschafhauser2@nebraska.edu}

\author[A. Tikuisis]{Aaron Tikuisis}
\address{\hspace{.5ex}Aaron Tikuisis, Department of Mathematics and Statistics, University of
  Ottawa, 585 King Edward, Ottawa, ON, K1N 6N5, Canada}
\email{aaron.tikuisis@uottawa.ca}

\author[S. White]{Stuart White}
\address{\hspace{.5ex}Stuart White, Mathematical Institute, University of Oxford,
  Oxford, OX2 6GG, United Kingdom}
\email{stuart.white@maths.ox.ac.uk}

\thanks{This work was supported by: NSF grant DMS-2400178 (Schafhauser); NSERC Discovery Grant (Tikuisis); a Visiting Research Fellowship at All Souls College (Tikuisis); Engineering and Physical Sciences Research Council [EP/X026647/1] (White). For the purpose of open access, the authors have applied a CC-BY license to any author accepted manuscript arising from this submission.}

\title{Nuclear {$C^*$}-algebras: 99 problems}
\dedicatory{In memory of Eberhard Kirchberg}
\date{\today}

\begin{abstract}
We present a collection of questions related to the structure and classification of nuclear {$C^*$}-algebras.
\end{abstract}

\maketitle

The Z\"urich International Congress of Mathematics in 1994 featured two talks on the structure and classification of simple nuclear {$C^*$}-algebras. In one, George Elliott set out his ambitious classification conjecture for simple separable nuclear {$C^*$}-algebras (\cite{Elliott:ICM}). In the other, Eberhard Kirchberg described a solution to half of Elliott's conjecture (\cite{Kirchberg:ICM}) -- this would later be published in \cite{Phillips:DM} as the renowned Kirchberg--Phillips theorem -- together with his celebrated `Geneva theorems’ on tensorial absorption. Over the next thirty years, and allowing for a certain amount of modification, the other half of Elliott’s grand vision has been realised through classification and structure theorems for simple nuclear {$C^*$}-algebras (Theorems \ref{Thm:UnitalClassification} and \ref{Thm:Structure} below), which parallel Connes' work on the structure and consequent uniqueness of the injective II$_1$ factor (\cite{Connes:Ann}).  Yet despite this success, there are many major challenges that remain; our purpose in this article is to collate a series of open questions stemming from the structure and classification of nuclear {$C^*$}-algebras. The continued fast pace of progress made for an extra challenge we were not expecting: some of our initial questions were (partially) answered during the time it took us to assemble this collection. In particular, \cite{AGKP:preprint} answers Problem~\ref{q:C*F2comparison} in full and Problem~\ref{q:Selfless} partially and instigated a whole slew of further developments, the very recent preprint \cite{Szabo:preprint} solves Problem \ref{KKUniqueness} and has impacts on others, and \cite{HP:arXiv} makes very interesting progress solving a case of Problem \ref{FreeIsoQn}. We have left these problems, and our commentary on them, essentially unchanged from our original version but added addenda to the end of Sections \ref{sec:K1injectivity}, \ref{sec:NonNuc} and \ref{Sect:Gen} setting out some of these developments.

Kirchberg’s contribution to and influence on the classification and structure theory for nuclear {$C^*$}-algebras has been immense.  A central theme of the paper is how results -- typically proved or inspired by Kirchberg -- for purely infinite {$C^*$}-algebras lay down a road to follow for stably finite {$C^*$}-algebras. It is no surprise that some essence of Kirchberg and his work can be found in nearly every topic within this paper.

\smallskip
\noindent\textbf{Acknowledgements. } Our work has benefited immensely from our participation in the AIM SQuaRE \emph{von Neumann algebraic techniques in the classification of {$C^*$}-algebras}, and we thank the other members of our SQuaRE -- Jos\'e Carri\'on, Kristin Courtney, and Jamie Gabe -- for their insights. We also thank Jakub Curda, Jamie Gabe, Julian Kranz, Robert Neagu, G\'abor Szab\'o, Hannes Thiel, Andrea Vaccaro, and the referee for the very helpful feedback and corrections they provided on earlier drafts of this manuscript.

\tableofcontents

\section{Classification and structure theorems for simple nuclear {$C^*$}-algebras}

Before turning to the questions, we start with a short summary of the current state of structure and classification theorems for {$C^*$}-algebras.
Projections in operator algebras are either infinite or finite, according to whether they are Murray--von Neumann equivalent to a proper subprojection of themselves. This provides the foundation of the theory of von Neumann algebras in Murray and von Neumann’s type decomposition, which ensures that a factor $\mathcal M$ is either finite (all projections are finite), the tensor product of a finite factor with $\mathcal B(\mathcal H)$, or type III -- purely infinite -- where all non-zero projections are infinite (and mutually equivalent in the countably decomposable setting).  R\o{}rdam’s famous example from \cite{Rordam:Acta} shows that a corresponding result does not hold for {$C^*$}-algebras. Nevertheless, there is a profound dichotomy -- due, perhaps inevitably, to Kirchberg\footnote{This result appears in R\o{}rdam’s book \cite{Rordam:Book} as Theorem 4.1.10 and in \cite{BlanchardKirchberg:JFA}, and we are grateful to Mikael R\o{}rdam for pointing out that Kirchberg communicated his dichotomy theorem in the mid-1990s. The form of the dichotomy for $\Z$-stable {$C^*$}-algebras in Theorem \ref{Thm:KDichotomoy} was independently noted in \cite{GJS:CMB} soon after the introduction of $\Z$.} -- which splits the {$C^*$}-algebras within the scope of classification into those that are stably finite (all projections in all matrix amplifications are finite) and those that are purely infinite (analogous to the type III factors; we defer the precise definition a few pages).  

The Jiang--Su algebra $\mathcal Z$ from \cite{JiangSu:AJM} appearing in the final part of the dichotomy below is, in a precise sense, the most natural {$C^*$}-analogue of the hyperfinite II$_1$ factor $\mathcal R$,\footnote{More precisely, $\Z$ is the minimal strongly self-absorbing {$C^*$}-algebra (\cite{Winter:JNCG}). Strongly self-absorbing {$C^*$}-algebras are discussed further in Section~\ref{sec:SSA}.}  and so $\Z$-stability for a {$C^*$}-algebra $A$ (i.e.\ $A\cong A\otimes\mathcal Z$) is best regarded as a {$C^*$}-analogue of the McDuff property for a II$_1$ factor $\mathcal M$ (i.e.\ $\mathcal M\cong\mathcal M\,\overline{\otimes}\,\mathcal R$). The algebra $\mathcal Z$ also shares many properties with $\mathbb C$ and is $KK$-equivalent to it, so that it can also be viewed as a non-commutative version of $\mathbb C$.

\begin{theorem}[{Kirchberg's dichotomy; \cite[Corollary~3.11]{BlanchardKirchberg:JFA}}]\label{Thm:KDichotomoy}
    Let $A$ and $B$ be simple non-elementary {$C^*$}-algebras.  If one of $A$ or $B$ is not stably finite, then $A\otimes B$ is purely infinite.
    In particular, every simple $\mathcal Z$-stable {$C^*$}-algebra is either stably finite or purely infinite.
\end{theorem}

Structure and classification results for operator algebras go hand in glove.  Murray and von Neumann’s uniqueness theorem for the (separably acting) hyperfinite II$_1$ factor is a classification result, but it is Connes’ structural counterpart -- the abstract characterisation of hyperfiniteness in terms of amenability and injectivity -- which both makes the uniqueness theorem vastly applicable in examples (particularly those coming from groups and dynamics) and underpins much of modern von Neumann algebra theory.

Turning to {$C^*$}-algebras, the \emph{unital classification theorem} is a centrepiece result established through the combined work of many researchers, including \cite{Kirchberg:ICM,Kirchberg:Book,Phillips:DM} for the purely infinite case -- the Kirchberg--Phillips theorem -- and \cite{Winter:Crelle,Winter:AJM,GLN:CR1,GLN:CR2,EGLN:JNCG,TWW:Ann} which handle the finite case, including conceptual breakthroughs by Winter (\cite{Winter:Crelle,Winter:AJM}) and Gong, Lin, and Niu's opus \cite{GLN:CR1,GLN:CR2}, which is the culmination of the long-term project of classifying by tracial approximation (\cite{Lin:Duke,Lin:IM,LinNiu:AIM}, for example).
The preprint \cite{CGSTW} gives an alternative, more abstract, proof of the finite case.  We follow the notation of \cite[Definition~2.3]{CGSTW} and write $KT_u(\,\cdot\,)$ for the classification invariant, consisting of $K$-theory together with the position of the unit, traces, and their pairing.\footnote{This formally differs from the Elliott invariant $\mathrm{Ell}(\,\cdot\,)$ in that the order structure on $K_0$ is not included (though on the class of algebras covered by the unital classification theorem, this order is recovered from tracial data, so the two invariants carry the same information).  See the discussion after \cite[Definition~2.3]{CGSTW} about the precise relation between these invariants.}

\begin{theorem}[Unital classification theorem]\label{Thm:UnitalClassification} Let $A$ and $B$ be unital simple separable nuclear $\mathcal Z$-stable {$C^*$}-algebras which satisfy Rosenberg and Schochet’s universal coefficient theorem (UCT).  Then $A\cong B$ if and only if $KT_u(A)\cong KT_u(B)$. \end{theorem}

A non-unital classification -- for {$C^*$}-algebras satisfying the same hypotheses except unitality -- was announced independently by Gong and Lin and by the authors with Carri\'on and Gabe around 2021.
In the former case, this consists of the papers \cite{EGLN:JGP,GongLin:JGP,GongLin:AKT,GongLin:AKT2} (some of which are joint with Elliott and Niu). See the survey article \cite{GLN23:Survey} for an account of the Gong--Lin--Niu tracial approximation approach to classification, covering both the unital and non-unital cases. The latter approach continues to be in preparation (\cite{CGSTW:draft}).
We call simple separable nuclear {$C^*$}-algebras which are $\Z$-stable and satisfy the UCT \emph{classifiable}.

The classification theorems, and their proofs, split according to Kirchberg's dichotomy: the purely infinite case was settled independently by Kirchberg and Phillips in the '90s (a detailed comparative overview of the two proofs is given in \cite{Rordam:Book}). Together with Carri\'on and Gabe, we wrote a history and survey of Theorem~\ref{Thm:UnitalClassification} in the introduction to \cite{CGSTW} (see also \cite{White:ICM}), so we will be relatively brief here and note that the classification of {$C^*$}-algebras is obtained from a classification of $^*$-embeddings, and that the hypotheses involved split into two components.  The first batch of hypotheses -- working with (unital) simple separable nuclear {$C^*$}-algebras -- correspond to working with injective factors in the von Neumann situation. The two other hypotheses -- being $\Z$-stable and satisfying the UCT -- are more subtle. We will have lots more to say about both, but for now, we note that one key difference is that there is an abundance of known non-$\Z$-stable simple separable nuclear {$C^*$}-algebras out there (pioneered by Villadsen in \cite{Villadsen:JFA} and refined to great effect in \cite{Rordam:DM,Rordam:Acta,Villadsen:JAMS,Toms:Ann}, for example), whereas it remains a major challenge whether the UCT is automatic for (simple) separable nuclear {$C^*$}-algebras (Problem~\ref{q:UCT}).

The existence of non-$\mathcal Z$-stable {$C^*$}-algebras led to \emph{regularity} becoming a major theme in the structure theory of simple nuclear {$C^*$}-algebras. Initiated by Toms and Winter's analysis in \cite{TomsWinter:JFA}, this centres around distinguishing (simple nuclear) {$C^*$}-algebras which are poorly-behaved (for instance, due to unusual ordered $K$-theory) from those which are well-behaved.  With the unital classification theorem in place, the primary goal of the `regularity programme' is to characterise $\Z$-stability within the class of (simple) separable nuclear {$C^*$}-algebras.\footnote{It is worth emphasising that `regularity' type results, like Theorem \ref{Thm:Structure}, are expected to be independent of the UCT.} The following structure theorem sums up a major achievement of this endeavour to date.

\begin{theorem}[{Structure theorem; \cite{KirchbergPhillips:Crelle,Winter:IM10,Winter:IM12,Tikuisis:MA,MatuiSato:Duke, BBSTWW:MAMS,CETWW:IM,CE:APDE}}]\label{Thm:Structure}
    Let $A$ be a simple separable nuclear non-elementary {$C^*$}-algebra.
    The following are equivalent:
    \begin{enumerate}
        \item[(i)] 
        $A$ has finite nuclear dimension;
        \item[(i$'$)]
        $A$ has nuclear dimension at most one;
        \item[(ii)]
        $A$ is $\mathcal Z$-stable.
    \end{enumerate}
    If $A$ is stably finite and all its traces are quasidiagonal,\footnote{In the non-unital case, `all its traces are quasidiagonal' must be interpreted to include traces on hereditary subalgebras; see \cite[Theorem~7.2]{CE:APDE}.} then these are also equivalent to
    \begin{enumerate}
        \item[(i$''$)]
        $A$ has decomposition rank at most one.
    \end{enumerate}
\end{theorem}

Here,  nuclear dimension and decomposition rank are the non-commutative covering dimensions from \cite{WinterZacharias:AIM} and \cite{KirchbergWinter:IJM}, respectively.\footnote{We have more to say about quasidiagonality further below, but as the quasidiagonality of all traces condition in (i$''$) is necessary, having finite decomposition rank is a stronger condition than finite nuclear dimension and is only applicable on the stably finite side of Kirchberg’s dichotomy.} The unital classification theorem was first proved with finite nuclear dimension in place of $\mathcal Z$-stability (and before it was shown that they are equivalent). Approximately finite-dimensional (AF) {$C^*$}-algebras are precisely those of nuclear dimension zero. From this viewpoint, the bound of one in Theorem~\ref{Thm:Structure}(i$'$) allows one to view condition \eqref{it:TW.1} as a one-dimensional -- or  $2$-coloured -- {$C^*$}-version of hyperfiniteness for von Neumann algebras. 

There are strong parallels between the structure theorem and Connes' theorem which persist at the technical level of proofs.  The last part of Connes' proof is to show that an injective II$_1$ factor $\mathcal M$ is hyperfinite after having established that it is McDuff. This uses the approximately inner flip on $\mathcal R$ and an embedding $\mathcal M\hookrightarrow \mathcal R^\omega$ (obtained from injectivity and providing the source of the finite-dimensional approximations needed for hyperfiniteness). In the unique trace case, the proof of Theorem \ref{Thm:Structure}(ii)$\implies$(i) follows a coloured (or higher dimensional)  version of this argument; in their paper \cite{MatuiSato:Duke}, Matui and Sato show how to use quasidiagonality (which, for nuclear {$C^*$}-algebras, can be thought of as an analogue of Connes’ embedding $\mathcal M\hookrightarrow \mathcal R^\omega$) together with a two-coloured version of the approximately inner flip (a concept made explicit in \cite{SWW:IM}) to obtain finite decomposition rank from $\Z$-stability (in hindsight, their proof also views $\Z$-stability as a $2$-coloured version of UHF-stability).  Later, a $2$-coloured version of quasidiagonality was given in \cite{SWW:IM} to prove \eqref{it:TW.2}$\implies$\eqref{it:TW.1} in the unique trace case, with increasingly more general trace spaces handled in \cite{BBSTWW:MAMS,CETWW:IM}.

While the equivalence in Theorem \ref{Thm:Structure} gives two very different ways of accessing the regularity needed for classification in examples,\footnote{Depending on the nature of an example of interest, one of these routes can be much more tractable than the other. See the discussion in  \cite[Section~1.1.4]{CGSTW}.} both the conditions of finite nuclear dimension and $\Z$-stability are somewhat technical. Ideally, the structure theorem would be further extended to characterise $\Z$-stability of simple separable nuclear {$C^*$}-algebras in even more basic terms.  For the infinite part of the structure theorem, one of Kirchberg’s famous Geneva theorems shows the way.

\begin{theorem}[Kirchberg]\label{Thm:Oinftystable}
Let $A$ be a simple separable nuclear {$C^*$}-algebra. Then $A$ is $\mathcal O_\infty$-stable if and only if $A$ is purely infinite.  
\end{theorem}

Note that as a consequence of Kirchberg’s Theorems \ref{Thm:KDichotomoy} and \ref{Thm:Oinftystable} and the existence of traces on exact stably finite {$C^*$}-algebras (see Theorem \ref{thm:HaagerupQT} below), any  simple separable nuclear $\Z$-stable {$C^*$}-algebra without a non-zero densely defined lower semicontinuous trace is necessarily $\mathcal O_\infty$-stable.  

We now turn to the definition of simple purely infinite {$C^*$}-algebras, and emphasise the viewpoint that this is a straightforward condition on the order structure of positive elements.  Whereas projections play a key role in the structure of von Neumann algebras, {$C^*$}-algebras need not have many or even any projections, and instead one must work with positive elements. The theory of \emph{Cuntz comparison} is the positive element version of Murray and von Neumann's comparison theory for projections in von Neumann algebras. For positive elements $a$ and $b$ in a {$C^*$}-algebra $A$, $a$ is \emph{Cuntz below} $b$, written $a\precsim b$, if $a=\lim_{n\to\infty} x_n^*bx_n$ for some sequence $(x_n)_{n=1}^\infty \subseteq A$; they are called \emph{Cuntz equivalent} when $a\precsim b$ and $b\precsim a$.  The \emph{Cuntz semigroup} $\Cu(A)$ is built from positive elements in the stabilisation $A\otimes\mathcal K$ modulo Cuntz equivalence. (The Cuntz semigroup has been well-studied both as a tool for {$C^*$}-algebraic advances and, to some extent, as an object of interest in its own right. The survey articles \cite{AraPereraToms} and \cite{GardellaPerera} are excellent places to get started.)
Then a simple {$C^*$}-algebra $A\neq \mathbb C$ is \emph{purely infinite} if any two non-zero positive elements are Cuntz equivalent, or equivalently $\Cu(A)\cong\{0,\infty\}$. This should be compared with the fact that a countably decomposable von Neumann factor which is not $\mathbb C$ is type III if and only if all its non-zero projections are equivalent.
Although it is not immediate from this definition that simple purely infinite {$C^*$}-algebras have any projections, it turns out that they have real rank zero (\cite[Proposition~3.9]{BrownPedersen:JFA} for the unital case; the non-unital case can be seen by combining this result with \cite[Theorem~1.2]{BlackadarCuntz:AMJ} and \cite[Theorem~3.8]{BrownPedersen:JFA}).

It is very natural to ask for a stably finite analogue of Theorem~\ref{Thm:Oinftystable}, i.e.\ to characterise $\Z$-stability for  simple separable nuclear {$C^*$}-algebras in terms of positive elements.  We will ask this in Problems \ref{q:TW} and \ref{Q7} below,
but we are certainly not the first to pose this question. Indeed, our exposition has proceeded ahistorically -- all of Theorem \ref{Thm:Structure} and more was predicted in the highly prophetic \emph{Toms--Winter conjecture} from around 2008.  

We need one more concept to describe the Toms--Winter conjecture. Functionals on the Cuntz semigroup arise from quasitraces (as discussed in \cite{MilhojRordam}, for example).\footnote{In the context of the Toms--Winter conjecture, one can work with traces as, by Haagerup's Theorem \ref{thm:HaagerupQT}, any quasitrace on an exact {$C^*$}-algebra is a trace.} For a unital {$C^*$}-algebra $A$ (where this is easier to describe), let $QT(A)$ denote the set of normalised quasitraces. Then any $\tau\in QT(A)$ induces an invariant of Cuntz equivalence by
\begin{equation}\label{dtau}
d_\tau(a)\coloneqq \lim_{n\to\infty}\tau(a^{1/n}),\quad a\in (A\otimes\mathcal K)_+
\end{equation}
(extending $\tau$ canonically to a densely defined lower semicontinuous quasitrace on $A\otimes\mathcal K$). We say that a unital simple {$C^*$}-algebra $A$ has \emph{strict comparison}\footnote{The reader is warned that the literature contains many variants of `strict comparison' (often under the same name, and especially outside of the unital simple case covered here); one must be careful about this when applying results.} if for all non-zero $a,b \in (A\otimes\mathcal K)_+$, 
\begin{equation}\label{eqn:StrictComp}
d_\tau(a)<d_\tau(b)\text{ for all }\tau \in QT(A)\Longrightarrow a\precsim b.
\end{equation}
This is the same as $\Cu(A)$ being almost unperforated (see \cite[Proposition~3.2]{Rordam:IJM}, which is generalised to an appropriate  version of strict comparison in the non-simple case).
Notice that if a unital simple {$C^*$}-algebra has no  quasitraces, then the left-hand side of \eqref{eqn:StrictComp} is vacuous and strict comparison is equivalent to $A\otimes\mathcal K$ being purely infinite (which is equivalent to pure infiniteness of $A$).  
The notion of strict comparison goes back at least to Blackadar's `fundamental comparability question' (version 2) in \cite{Blackadar88}.\footnote{This is a version of \eqref{eqn:StrictComp} for projections, which was ultimately answered by Villadsen \cite{Villadsen:JFA}.}
For stably finite {$C^*$}-algebras, we view strict comparison as the appropriate {$C^*$}-algebra version of the fact that the order on projections in a II$_1$ factor is determined by the trace.   With this setup, the Toms--Winter conjecture (and now mostly theorem) can be stated as follows.

\begin{conjecture}[Toms--Winter]\label{TWConjecture}
Let $A$ be a simple separable nuclear non-elementary {$C^*$}-algebra.  The following are equivalent:
\begin{enumerate}[(i)]
\item \label{it:TW.1}
$A$ has finite nuclear dimension;
\item \label{it:TW.2}
$A$ is $\mathcal Z$-stable;
\item \label{it:TW.3}
$A$ has strict comparison.
\end{enumerate}
\end{conjecture}

 Conjecture~\ref{TWConjecture} holds in the absence of (densely defined lower semicontinuous) traces. We record this separately for later reference, though it essentially is contained in the results described above.\footnote{We noted above that Theorems \ref{Thm:KDichotomoy} and \ref{Thm:Oinftystable} give (ii)$\implies$(iii) and (iii$'$), and the reverse direction follows from Theorem \ref{Thm:Oinftystable} as $\mathcal O_\infty$ is $\Z$-stable.  (i)$\implies$(ii) is best done using a version of the dichotomy theorem for simple {$C^*$}-algebras with finite nuclear dimension (\cite[Theorem 5.4]{WinterZacharias:AIM}, heavily using Kirchberg's work from \cite{Kirchberg:Abel}), and the reverse implication was first established by Matui and Sato in \cite{MatuiSato:Duke} with the optimal bound later obtained in \cite{BBSTWW:MAMS}.}

 \begin{theorem}\label{Thm:Tracelessstructure}
 Let $A$ be a simple separable nuclear (non-elementary) {$C^*$}-algebra with no non-zero densely defined lower semicontinuous traces.  The following are equivalent:
\begin{enumerate}[(i)]
\item[(i)] 
$A$ has finite nuclear dimension;
\item[(i$'$)] $A$ has nuclear dimension one;
\item[(ii)]
$A$ is $\mathcal Z$-stable;
\item[(ii$'$)] $A$ is $\mathcal O_\infty$-stable;
\item[(iii)] 
$A$ has strict comparison;
\item[(iii$'$)] $A$ is purely infinite.
\end{enumerate}
 \end{theorem}

 Toms and Winter made their conjecture\footnote{Unfortunately, Toms and Winter failed to write down their conjecture in a joint paper; the best early references are \cite[Remarks 3.5]{TomsWinter:JFA}, which is only for the finite case and uses decomposition rank instead of nuclear dimension, and \cite[Conjecture 9.3]{WinterZacharias:AIM}.} based on their analysis of Villadsen’s construction in \cite{TomsWinter:JFA} before any of the stably finite parts of Theorem \ref{Thm:Structure} were known in generality.  Sticking to our non-chronological approach, even before this, R\o{}rdam 
proved \eqref{it:TW.2}$\implies$\eqref{it:TW.3} in Conjecture \ref{TWConjecture} (see \cite[Corollary~4.6]{Rordam:IJM} for the unital case -- the non-unital case follows from \cite[Corollary~4.7]{Rordam:IJM}; unlike all the other implications in the conjecture, this holds without nuclearity). The remaining challenge in Conjecture \ref{TWConjecture} (Problem~\ref{q:TW} below) is \eqref{it:TW.3}$\implies$\eqref{it:TW.2}. 
 Matui and Sato made a major breakthrough with this implication in 2012, establishing it under the additional hypothesis that $A$ has a unique trace, or more generally, when $A$ has finitely many extremal traces. 
\begin{theorem}[Matui--Sato; \cite{MatuiSato:Acta}]\label{Thm:MS}
Let $A$ be a unital simple separable nuclear stably finite {$C^*$}-algebra with strict comparison and finitely many extremal traces.  Then $A$ is $\Z$-stable.
\end{theorem}

This covers many natural examples (e.g.\ crossed products arising from uniquely ergodic free minimal actions of discrete amenable groups), and the breakthrough methods they introduced (see Section \ref{sec:TW}) sparked the modern use of von Neumann algebraic techniques to tackle structural problems in simple nuclear {$C^*$}-algebras. 

\section{The quasitrace problem}
\label{sec:QT}

Quasitraces are a technical generalisation of traces arising from dimension functions. From a modern prospective, they correspond one-to-one with functionals on the Cuntz  semigroup, in that every such functional arises in the form \eqref{dtau} for some quasitrace $\tau$.\footnote{As with traces, the term `quasitrace' (by which we mean a $2$-quasitrace) is sometimes used to denote normalised functions, whereas other times, it may indicate a lower semicontinuous extended function.  In this paper, we are at times intentionally ambiguous to encapsulate different variations on themes and questions in the literature.} The essential difference between quasitraces and traces is that the former need only be additive on commuting elements. Quasitraces have a long history, dating right back to Murray and von Neumann's first paper \cite{MurrayVonNeumann1} (though the terminology and their abstract study came later in \cite{BlackadarHandelman:JFA}). Murray and von Neumann originally showed that a II$_1$ factor has a unique quasitrace (\cite[Theorem XIII]{MurrayVonNeumann1}) and later showed the additivity in \cite{MurrayVonNeumann2}, so that II$_1$ factors have a unique trace.

A long-standing open question about quasitraces, which in spirit goes back to \cite{Kaplansky1,Kaplansky2}, asks whether the ``quasi'' is necessary.  This is equivalent to Kaplansky's question of whether type II$_1$ $AW^*$-factors are von Neumann algebras.

\begin{question}\label{q:QT}
    If $\tau$ is a bounded quasitrace on a {$C^*$}-algebra, must $\tau$ be a trace?
\end{question}

The following is a tremendous result by Haagerup (\cite{Haagerup:CRMASSRC})\footnote{The preprint of this paper was completed in 1991 and was circulated among the community from that point.} in the unital case and then extended to lower semicontinuous unbounded quasitraces on non-unital {$C^*$}-algebras by Kirchberg (\cite{Kirchberg97}; see also \cite[Remark~2.29(i)]{BlanchardKirchberg:JFA}):

\begin{theorem}[Haagerup, Kirchberg]\label{thm:HaagerupQT}
    If $A$ is an exact {$C^*$}-algebra then all lower semicontinuous quasitraces on $A$ are traces.
\end{theorem}

For this reason, the open part of the quasitrace problem lies outside the realm of nuclear {$C^*$}-algebras.

The following corollary is an important consequence of Theorem \ref{thm:HaagerupQT} together with Blackadar and Handelman's work \cite{BlackadarHandelman:JFA}.\footnote{One can also use these ideas to obtain non-zero densely defined traces on more general exact stably finite algebras, for example in the unital case, or when the primitive ideal space of $A$ is compact (see \cite[Theorem 2.15]{Rordam:TAMS}, a reference we thank the referee for bringing to our attention). But algebras such as $C_0((0,1])\otimes\mathcal O_2$ are stably finite as their matrix amplifications have no projections, but equally have no non-zero traces.}

\begin{corollary}\label{cor:BH}
    If $A$ is a unital simple exact stably finite {$C^*$}-algebra then $A$ has a non-zero lower semicontinuous densely defined trace.
\end{corollary}

There is much more to be said about the quasitrace problem and its equivalence to other problems (such as whether the minimal tensor product of stably finite {$C^*$}-algebras is again stably finite). For this and further recent developments, we refer the reader to Milh\o{}j and R\o{}rdam's article \cite{MilhojRordam} 
devoted to this problem.

\section{The UCT}\label{sec:UCT}

The \emph{universal coefficient theorem (UCT)} of Rosenberg and Schochet (\cite{Rosenberg-Schochet87}) provides a powerful method for computing $KK$-groups in terms of the operator $K$-theory groups $K_0$ and $K_1$.  The name comes from the analogous theorem in algebraic topology, which, for suitable spaces $X$, computes the cohomology groups $H^*(X ; G)$ with coefficients in an abelian group $G$ in terms of $G$ and the homology groups $H_*(X)$. 
For {$C^*$}-algebras, one of the cleanest characterisations is that a separable {$C^*$}-algebra $A$ satisfies the UCT if for every ($\sigma$-unital)\footnote{With a suitable interpretation of the group $KK(A, B)$, this also holds for non-$\sigma$-unital {$C^*$}-algebras $B$ -- see \cite[Appendix~B]{CGSTW}, for example.  Reducing the UCT to the $\sigma$-unital case can be done using \cite[Proposition~1.11]{Schafhauser:Ann}, for example.} {$C^*$}-algebra $B$ such that $K_*(B)$ is divisible, the natural map
\begin{equation}\label{div-UCT}
\begin{tikzcd}
    KK(A, B) \arrow{r}{\alpha}  & \mathrm{Hom}\big(K_*(A), K_*(B)\big)
\end{tikzcd}
\end{equation}
is an isomorphism.  The fairly restrictive condition that $K_*(B)$ is divisible should be regarded as a normalisation condition.  If $A$ satisfies the UCT in the sense above, then for all {$C^*$}-algebras $B$, there is a natural short exact sequence
\begin{equation}\label{UCT}
\begin{tikzcd}
    \mathrm{Ext}\big(K_*(A), K_{*+1}(B)\big) \arrow[tail]{r}{\gamma} & KK(A, B) \arrow[two heads]{r}{\alpha} & \mathrm{Hom}\big(K_*(A), K_*(B)\big),
\end{tikzcd}
\end{equation}
which splits unnaturally.\footnote{Other normalisations are possible as well.  For example, a separable {$C^*$}-algebra $A$ satisfies the UCT if and only if for all {$C^*$}-algebras $B$, $K_*(B) = 0$ implies $KK(A, B) = 0$.  Informally, this implies that if there is any method for computing $KK(A, \,\cdot\,)$ in terms of $K_*(A)$, then there is an exact sequence as in \eqref{UCT}. The argument for this is a mapping cone construction similar to the second part of the proof of \cite[Corollary~8.4.6(ii)$\implies$(iii)]{Rordam:Book}.}

Although, on the surface, the UCT has an algebraic nature, it is usually verified through analytic methods.  For a {$C^*$}-algebra $B$ with $K_*(B)$ divisible, the functors $KK(\,\cdot\,, B)$ and $\mathrm{Hom}(K_*(\,\cdot\,), K_*(B))$ are homotopy invariant, Morita invariant, half-exact on semisplit extensions,\footnote{The failure of half-exactness of $KK$-theory in general is one of the impeti for the development of $E$-theory (\cite{Higson:JPAA, ConnesHigson:CRAcadSci}).} and $\sigma$-additive -- in the case of $KK(\,\cdot\,, B)$, these are deep analytic results of Kasparov (\cite{Kasparov79}).
Using the naturality of the map $\alpha$ in \eqref{div-UCT} and various six-term exact sequences arising from varying $A$ in both sides of \eqref{div-UCT}, Rosenberg and Schochet establish several permanence properties for the class of {$C^*$}-algebras satisfying the UCT:\footnote{Rosenberg and Schochet only consider permanence properties for nuclear {$C^*$}-algebras, but their arguments apply more generally as presented here.} it is closed under semisplit extensions, suspensions, direct limits of nuclear {$C^*$}-algebras, and $KK$-equivalence (in particular, homotopy equivalence and Morita equivalence).  Using the Pimsner--Voiculescu sequence and Connes' Thom isomorphism, the class is also closed under crossed products by $\mathbb Z$ and $\mathbb R$.  It is easy to see that $\mathbb C$ satisfies the UCT, and hence the UCT holds on a bootstrap class of separable {$C^*$}-algebras that can be built from $\mathbb C$ by iteratively applying these constructions.
To show all commutative {$C^*$}-algebras satisfy the UCT, Rosenberg and Schochet used their permanence properties to reduce it first to the case of CW complexes, and then to $C_0(\mathbb R^n)$, which is handled via Bott periodicity.
Pushing this further, they proved that all type I {$C^*$}-algebras satisfy the UCT, together with their inductive limits, thus including all approximately subhomogeneous (ASH) algebras -- the key observation being that all stable type I {$C^*$}-algebras admit a (typically transfinite) composition series with factors of the form $C_0(X) \otimes \mathcal K$.\footnote{One of many equivalent characterisations of type I {$C^*$}-algebras is that every quotient contains an abelian element (see \cite[Chapter 6 and Definition 6.1.1 in particular]{Pedersen:Book} or \cite[Chapter 4]{Dixmier:Book}).
Since the ideal generated by an abelian element is stably isomorphic to $C_0(X)\otimes \mathcal K$, repeatedly using such ideals leads to the required composition series.}

Beyond the ASH case, the most powerful method in practice for verifying the UCT comes from the theorem of Tu below on amenable groupoids.  The proof uses the same techniques needed to prove the Baum--Connes conjecture for such groupoids and is modelled on the result of Higson and Kasparov showing amenable groups satisfy the Baum--Connes conjecture (\cite{HigsonKasparov:IM}).

\begin{theorem}[Tu; \cite{Tu99}]\label{thm:Tu}
    If $\mathcal G$ is a second countable locally compact Hausdorff amen\-able groupoid, then $C^*(\mathcal G)$ satisfies the UCT.
\end{theorem}

Tu's theorem also follows from analytic methods obtaining a decomposition theorem for $C^*(\mathcal G)$.  In fact, he shows $C^*(\mathcal G)$ is $KK$-equivalent to a direct limit of type~I {$C^*$}-algebras.  More recently, Barlak and Li (\cite{Barlak-Li17}) extended Tu's theorem to twisted groupoid {$C^*$}-algebras $C^*(\mathcal G, \Sigma)$ under the additional hypothesis that $\mathcal G$ is \'etale by using the Packer--Raeburn stabilisation trick (\cite{Packer-Raeburn1,Packer-Raeburn2}) to untwist the groupoid at a suitable place in Tu's proof.  Combining this with Renault's theorem (\cite{RenaultIMSB}), every separable nuclear {$C^*$}-algebra with a Cartan subalgebra satisfies the UCT.

There is a rather striking permanence property due to Dadarlat (\cite{Dadarlat01}) that highlights the analytic nature of the UCT.  If $A$ is a separable nuclear {$C^*$}-algebra and every finite subset of $A$ is approximately contained in a (not necessarily nuclear) UCT subalgebra $B \subseteq A$, then $A$ satisfies the UCT.  When the subalgebras $B$ form an increasing union, this follows from a Milnor $\lim^1$-sequence, but there is no reason to believe a priori that such local approximations can be arranged to be nested (cf.\ \cite{DadarlatEilers:Crelle}).
Dadarlat's proof involves a series of reductions, ending with the case that $A$ and each approximating subalgebra is simple and tracially AF, and then classification techniques are used to force the subalgebras $B$ to be nested.

In spite of the impressive collection of examples and permanence properties obtained over several decades, it still remains unclear how large the class of {$C^*$}-algebras satisfying the UCT is.  In the non-nuclear case, a result of Skandalis (\cite{Skandalis88}) shows that if $G$ is a biexact group with property (T), then $C^*_r(G)$ fails the UCT.  In the nuclear case, the following UCT problem is wide open and is perhaps the most important open problem about nuclear {$C^*$}-algebras.

\begin{question}\label{q:UCT}
Do all separable nuclear {$C^*$}-algebras satisfy Rosenberg and Schochet's universal coefficient theorem?
\end{question}

There are several known reductions.  An easy one is that if the UCT is preserved under taking quotients of nuclear {$C^*$}-algebras, then all nuclear {$C^*$}-algebras satisfy the UCT (use that any {$C^*$}-algebra $A$ is a quotient of its cone $C_0((0,1])\otimes A$, which is contractible, so will satisfy the UCT).

A much harder reduction, due to Kirchberg, is that if crossed products of separable nuclear {$C^*$}-algebras by $\mathbb T$ satisfy the UCT, then all nuclear {$C^*$}-algebras satisfy the UCT (see \cite[Exercise~23.15.12]{Blackadar-kbook}).  Using Kirchberg's ideas, it has also been shown that if crossed products by both $\mathbb Z/p$ and $\mathbb Z/q$ for relatively prime natural numbers $p, q \geq 2$ preserve the UCT for separable nuclear {$C^*$}-algebras, then all separable nuclear {$C^*$}-algebras satisfy the UCT (\cite[Theorem~4.17]{BarlakSzabo:TAMS}).

Kirchberg has also shown how to reduce the UCT problem to the simple purely infinite setting (\cite[Proposition~8.4.5]{Rordam:Book}).  So, if all Kirchberg algebras satisfy the UCT, then all separable nuclear {$C^*$}-algebras satisfy the UCT.  In fact, via a mapping cone construction and the Kirchberg--Phillips theorem, the UCT problem is equivalent to the statement that every unital Kirchberg algebra $A$ with $K_*(A) = 0$ is isomorphic to the Cuntz algebra $\mathcal O_2$ (see \cite[Corollary 8.4.6(ii)]{Rordam:Book}).

Almost nothing is known about the class of {$C^*$}-algebras $B$ for which the UCT exact sequence \eqref{UCT} holds for all separable nuclear {$C^*$}-algebras $A$.  Indeed, to our knowledge, the only $B$ that are known to have this property are those with $KK(\cdot,B)=0$ (so including $B=\mathcal O_2$, as well as all purely infinite von Neumann algebras, and multipliers of stable separable {$C^*$}-algebras). Indeed, when $B \coloneqq \mathbb C$, $KK(A, \mathbb C)$ is the $K$-homology group $K^0(A)$, and when $B \coloneqq C_0(\mathbb R)$, $KK(A, C_0(\mathbb R))$ is the Brown--Douglas--Fillmore extension group $K^1(A) = \mathrm{Ext}^{-1}(A)$.
Even in these cases it is unknown if \eqref{UCT} holds in any more generality than when $A$ satisfies the UCT.

A special case that has interested our AIM SQuaRE is when $B \coloneqq \mathcal R^\omega$, the tracial ultrapower of the separably acting hyperfinite II$_1$ factor $\mathcal R$ -- we have been (perhaps misleadingly)  calling this the `$\mathcal R^\omega$-UCT problem'.  As with all II$_1$ factors, $K_1(\mathcal R^\omega) = 0$ and there is an isomorphism $K_0(\mathcal R^\omega) \cong \mathbb R$ given by sending a projection in a matrix algebra over $\mathcal R^\omega$ to its trace.  If $A$ is a separable nuclear {$C^*$}-algebra, then by Connes' theorem, two $^*$-homomorphisms $\phi, \psi \colon A \rightarrow \mathcal R^\omega$ are unitarily equivalent (and hence agree in $KK$) if and only if $\mathrm{tr} \circ \phi = \mathrm{tr} \circ \psi$.  The injectivity of the map in the following problem essentially asks whether this generalises from $^*$-homomorphisms to `$KK$-maps' $A \rightarrow \mathcal R^\omega$ -- does agreement on traces imply agreement in $KK$?

\begin{question}[$\mathcal R^\omega$-UCT problem]\label{q:RUCT}
    If $A$ is a separable nuclear {$C^*$}-algebra, is the natural map 
    \begin{equation}\label{eq:RUCT}
    KK(A, \mathcal R^\omega) \longrightarrow \mathrm{Hom}(K_0(A), \mathbb R)
    \end{equation}
    an isomorphism?
\end{question}

Surprisingly, the $\mathcal R^\omega$-UCT does hold when $A\coloneqq \mathcal D$ is a finite strongly self-absorbing {$C^*$}-algebra (we discuss strongly self-absorbing algebras in more detail in  Section~\ref{sec:SSA}), which gives some evidence that it may hold generally.
This is because $\mathcal R^\omega$ is a quotient of $\mathcal D_\omega$, and it is therefore separably $\mathcal D$-stable.\footnote{Here $\mathcal D_\omega$ denotes the {$C^*$}-algebra ultrapower, and by separable $\mathcal D$-stability of $\mathcal R^\omega$, we mean that any norm-separable $B_0\subset \mathcal R^\omega$ is contained in a norm-separable $\mathcal D$-stable subalgebra of $\mathcal R^\omega$.}
Then one combines the general result that $KK(\mathcal D, A)\cong K_0(A)$ whenever $A$ is $\mathcal D$-stable (\cite[Theorem~3.4]{DadarlatWinter:MS}) with a standard separabilisation argument.  This plays a crucial role in the first-named author's recently announced $KK$-classification of finite strongly self-absorbing algebras (\cite[Theorem~C]{Schafhauser24}).

In view of the slow rate of progress on the UCT problem, it becomes natural to try to remove UCT assumptions from various results, including structural and classification results and general properties of $KK$.
We discuss a couple questions along these lines here, and this theme recurs frequently later -- for example, in the discussion after Theorem~\ref{thm:AFembedding}, in Section~\ref{sec:SSA}, and in Problem~\ref{q:KKclassification}.  The first such problem concerns the computation of $K$-theory of (minimal) tensor products. For {$C^*$}-algebras $A$ and $B$, there is a natural $\mathbb Z/2$-graded map
\begin{equation}\label{Kunneth}
    \alpha:K_*(A)\otimes_{\mathbb Z} K_*(B)\to K_*(A\otimes B), 
\end{equation}
given in $K_0$ by tensoring projections.  In the unital case, the product of a class $[u]_1\in K_1(A)$ and $[p]_0\in K_0(B)$, coming from a unitary $u\in M_m(A)$ and projection $p\in M_n(B)$, respectively, is given by the class of $u\otimes p+1_A\otimes (1_{M_n(B)}-p)$, and the product of elements in $K_0(A)$ with elements of $K_1(B)$ is defined similarly.  Bott periodicity is used to describe the map $K_1(A)\otimes_{\mathbb Z} K_1(B) \rightarrow K_0(A \otimes B)$ in terms of classes of unitaries; see \cite[Appendix 2]{Connes:Adv}, for example. In \cite{Schochet82}, Schochet showed that this is an isomorphism whenever $K_*(B)$ is torsion-free, and $A$ lives in a bootstrap class of separable nuclear {$C^*$}-algebras (containing inductive limits of type $I$ algebras and closed under stable isomorphism, under a `two-out-of-three' property for short exact sequences, and under crossed products by $\mathbb Z$ or $\mathbb R$).  Following Rosenberg and Schochet's work, one gets the same conclusion whenever $A$ satisfies the UCT and $K_*(B)$ is torsion-free.\footnote{For non-nuclear {$C^*$}-algebras, one gets this by showing that the class of separable {$C^*$}-algebras $A$, for which \eqref{Kunneth} is an isomorphism for all torsion-free $K_*(B)$, is closed under $KK$-equivalence.}

 The requirement that $K_*(B)$ is torsion-free is a normalisation condition. As set out in \cite[Section 4]{Schochet82}, if $A$ is a {$C^*$}-algebra such that \eqref{Kunneth} is an isomorphism whenever $K_*(B)$ is torsion-free, then for any {$C^*$}-algebra $B$ one has a natural short exact sequence
\begin{equation}\label{Kunneth2}
    0\to K_*(A)\otimes K_*(B)\stackrel{\alpha}{\longrightarrow} K_*(A\otimes B)\stackrel{\beta}{\longrightarrow}\mathrm{Tor}_1^{\mathbb Z}(K_*(A),K_*(B))\to 0
\end{equation}
 (where $\beta$ has degree $1$) known as the K\"unneth formula for tensor products. For separable nuclear {$C^*$}-algebras, the bootstrap classes used to obtain the K\"unneth formula and the UCT are the same, but a priori,  the K\"unneth formula for $K$-theory is weaker than the full force of the UCT.

 In fact, as pointed out to us by the referee, there are examples of (non-nuclear) separable $C^*$-algebras which fail the UCT but satisfy the K\"unneth formula for minimal tensor products.  For example, if $G$ is a biexact group with property (T), then $C^*_r(G)$ fails the UCT by Skandalis's result recalled just before Problem~\ref{q:UCT} (\cite{Skandalis88}).  On the other hand, $G$ satisfies the Baum--Connes conjecture with coefficients (\cite[Th\'eor\`eme~1.2] {Lafforgue:JNCG}) and hence $C^*_r(G)$ satisfies the K\"unneth formula for minimal tensor products (\cite[Corollary~0.2]{Chabert-Echterhoff-OyonoOyono:GAFA}).

\begin{question}[K\"unneth formula for tensor products]
Let $A$ be a separable nuclear {$C^*$}-algebra and $B$ a separable {$C^*$}-algebra such that $K_*(B)$ is torsion-free.  Is the natural map $\alpha$ in \eqref{Kunneth} an isomorphism? 
\end{question}

In another direction, there is a natural topology on $KK(A, B)$ (see \cite{Dadarlat:JFA} for a definition and a discussion of the history), and the Hausdorffised groups $KL(A, B) \coloneqq KK(A, B) / \overline{\{0\}}$ play an important role in the classification theorem, dating back to R{\o}rdam's introduction of $KL$-groups in \cite{Rordam:JFA} (defined there only when $A$ satisfies the UCT!).

Under the UCT, the `universal multi-coefficient theorem' of Dadarlat and Loring (\cite{Dadarlat-Loring:Duke}) readily implies that $KL(\,\cdot\,, B)$ preserves limits on the class of nuclear {$C^*$}-algebras in the UCT class.  This was recently shown to hold without a UCT condition by Carri\'on and the first-named author in \cite{Carrion-Schafhauser}.  In a similar direction, it is easy to show that $KK(A, B)$ is Hausdorff whenever $A$ satisfies the UCT and $K_*(B)$ is divisible.
Without the UCT, divisibility of $K_*(B)$ alone should not put many restrictions on $KK(A, B)$, but divisibility conditions on the {$C^*$}-algebra $B$ itself might.
A positive answer to the following, in the case when $B \coloneqq \mathcal  R^\omega$, for example, would follow from a positive answer to Problem~\ref{q:RUCT}.  In the following, $\mathcal Q$ denotes the universal UHF algebra (so $K_0(\mathcal Q) \cong \mathbb Q)$.

\begin{question}
    If $A$ is a separable nuclear {$C^*$}-algebra and $B$ is a (separably) $\mathcal Q$-stable {$C^*$}-algebra, is $KK(A, B)$ Hausdorff?
\end{question}



On the structural side, the UCT has been used to produce various approximation properties for nuclear regular {$C^*$}-algebras of both an internal and external nature.  For external approximations, we defer the discussion to the next section on quasidiagonality and AF-embeddability.  Internally, the UCT has been used to produce tracial approximations for certain {$C^*$}-algebras in the work of Elliott, Gong, Lin, and Niu in \cite{EGLN}, which was one of the final steps in the Gong--Lin--Niu approach to the stably finite side of Theorem~\ref{Thm:UnitalClassification}.

\begin{question}\label{Q5}
    If $A$ is a unital simple separable nuclear $\mathcal Q$-stable stably finite {$C^*$}-algebra such that every trace on $A$ is quasidiagonal, does $A$ have generalised tracial rank at most one?
\end{question}

In the unique trace case, Problem~\ref{Q5} has a positive answer. Matui and Sato's \cite[Theorem 6.1]{MatuiSato:Duke} shows that unital simple separable  nuclear quasidiagonal {$C^*$}-algebras which absorb a UHF algebra of infinite type are tracially AF.  Matui and Sato's direct argument to access tracial approximations from quasidiagonality in this context is very much in the spirit of the final steps of Connes' proof that injectivity implies hyperfiniteness (as per the discussion following Theorem \ref{Thm:Structure}).  Problem~\ref{Q5} asks whether we can obtain tracial approximations in the spirit of a regularity type result, without the UCT.

\section{Quasidiagonality and AF-embeddability}

A {$C^*$}-algebra is residually finite-dimensional if it has approximately isometric finite-dimensional representations. Of course, no infinite-dimensional simple {$C^*$}-algebra can be residually finite-dimensional, but upon weakening genuine finite-dimensional representations to c.p.c.\ approximately multiplicative maps, suddenly a great many {$C^*$}-algebras have such approximations (including all the stably finite classifiable {$C^*$}-algebras).  This is the notion of quasidiagonality -- a property that was first  introduced by Halmos in the context of operator theory (\cite{Halmos:BAMS}) and then characterized by Voiculescu in terms of the aforementioned approximations (\cite{Voiculescu:Duke}). Putting Voiculescu's result another way, a separable {$C^*$}-algebra $A$ is quasidiagonal if and only if it has an embedding into $\mathcal Q_\omega$ 
with a c.p.c.\ lift to $\ell_\infty(\mathcal Q)$.
When $A$ is nuclear, the Choi--Effros lifting theorem (\cite{ChoiEffros:Ann}) tells us that it is enough that $A$ is \emph{matricial field} (MF), i.e.\ embeds into $\mathcal Q_\omega$; in this formulation, one sees the links to Connes' embedding problem (which originated in Connes' observation that separably acting injective II$_1$ factors embed into the von Neumann ultraproduct of the hyperfinite II$_1$ factor). The condition of being MF is equivalent to being the fibre at $\infty$ of some continuous field of {$C^*$}-algebras over $\mathbb N \cup \{\infty\}$, whose fibre at each $n\in\mathbb N$ is finite-dimensional (\cite[Theorem~3.2.2]{BlackadarKirchberg:MA}).

In their seminal papers \cite{BlackadarKirchberg:MA,BlackadarKirchberg:PJM}, Blackadar and Kirchberg examined the combination of nuclearity and quasidiagonality, calling such algebras \emph{NF algebras} (for `nuclear finite' in the hope of a positive answer to the next problem).
They demonstrated that this property is equivalent to the existence of a certain generalised inductive limit presentation:
a sequence of finite-dimensional {$C^*$}-algebras $(F_n)_{n=1}^\infty$ with c.p.c.\ connecting maps $\phi_n\colon F_n \to F_{n+1}$ which are asymptotically multiplicative.\footnote{There is a similar characterisation of MF algebras, not requiring the $\phi_n$ to be c.p.c.}
They ask the following question.
\begin{question}[{Blackadar--Kirchberg; \cite[Question 7.3.2]{BlackadarKirchberg:MA}}] \label{q:QD}
    Is every separable nuclear stably finite {$C^*$}-algebra quasidiagonal?
\end{question}

The more general question of whether every separable stably finite (not necessarily nuclear) {$C^*$}-algebra is MF is resolved by the negative solution to the Connes embedding problem (\cite{JNVWY}): if $\mathcal M$ is a separably acting II$_1$ factor which is not $\mathcal R^\omega$-embeddable then take a strongly dense separable {$C^*$}-subalgebra of $\mathcal M$ with unique trace -- this cannot be $\mathcal Q_\omega$-embeddable.

A related set of questions concerns AF-embeddability.
A {$C^*$}-algebra $A$ is \emph{AF-embeddable} if it can be embedded in some AF algebra. This property clearly implies quasidiagonality, and there are no other known obstructions for a separable exact {$C^*$}-algebra to be AF-embeddable.
Disentangling this question from the quasidiagonality question, the pertinent open question asks whether every  separable exact quasidiagonal {$C^*$}-algebra is AF-embeddable.
Blackadar and Kirchberg ask a weaker version of this question in \cite[Question 7.3.3]{BlackadarKirchberg:MA}, replacing exactness by nuclearity.

\begin{question}
\label{q:AFembedding}
    Is every separable exact quasidiagonal {$C^*$}-algebra AF-embeddable?
\end{question}

The quasidiagonality and AF-embeddability questions have received significant attention, especially for reduced group {$C^*$}-algebras of discrete amenable groups,\footnote{Quasidiagonality for such {$C^*$}-algebras is commonly referred to as Rosenberg's conjecture (see \cite{CarrionDadarlatEckhard:JFA,OzawaRordamSato:GAFA}, for example) and is now a theorem (\cite[Theorem~C]{TWW:Ann}).} but also for other cases such as crossed products (\cite{Pimsner:ETDS,Brown:JFA,Brown:HMJ, Lin:IMRP}) and contractible {$C^*$}-algebras (\cite{Voiculescu:Duke,Ozawa:GAFA}).
The state-of-the-art is that both problems have a positive solution, for separable exact {$C^*$}-algebras satisfying the UCT which also have a faithful amenable trace (a significant strengthening of stable finiteness in the non-simple case).  Using traces as a framework for studying amenability is one of the major ideas in Connes' celebrated work \cite{Connes:Ann}; this was further developed (outside the II$_1$ factor setting) by Kirchberg in \cite{Kirchberg:MathAnn} who called these traces `liftable'. 
An influential and comprehensive study of amenable and quasidiagonal traces (see below) was later undertaken by Nate Brown in \cite{Brown:MAMS}.

\begin{theorem}[\cite{TWW:Ann,Gabe:JFA,Schafhauser:Crelle,Schafhauser:Ann}]
\label{thm:AFembedding}
    Let $A$ be a separable exact {$C^*$}-algebra which satisfies the UCT.
    Then $A$ embeds into a unital simple  AF algebra if and only if $A$ has a faithful amenable trace.
\end{theorem}

In light of this, and in the spirit of `removing the UCT' from known structural results (as in Section~\ref{sec:UCT}), one is led to natural special cases of Problems~\ref{q:QD} and \ref{q:AFembedding} which are major challenges in their own right. 
\begin{question}\label{q:QDTrace}
\begin{enumerate}[(1)]
\item
    Let $A$ be a separable nuclear {$C^*$}-algebra with a faithful trace.  Is $A$ quasidiagonal?\label{q:QDTrace.1}
    \item   Let $A$ be a separable exact {$C^*$}-algebra with a faithful quasidiagonal trace.  Is $A$ AF-embeddable?\label{q:QDTrace.2}
    \end{enumerate}
\end{question}

Nate Brown developed the concept of quasidiagonal traces appearing in Problem \ref{q:QDTrace}(\ref{q:QDTrace.2}) above by converting the Hilbert--Schmidt norm approximations used to define amenability of traces to a stronger operator norm condition (\cite{Brown:MAMS}). His prediction that quasidiagonal  traces would prove important objects in the classification programme proved very accurate, particularly through the combination of \cite{EGLN,TWW:Ann}. In Brown's framework, Rosenberg's conjecture naturally extends to traces. If Problem~\ref{q:tracesqd}(\ref{q:tracesqd.1}) below (cf.\ \cite[Discussion before Proposition 3.5.1]{Brown:MAMS}) has a positive answer, it would imply a positive answer to Problem \ref{q:QDTrace}(\ref{q:QDTrace.1}).  It seems likely to us that for not necessarily exact {$C^*$}-algebras, the heart of matter is whether the trace on the hyperfinite II$_1$ factor is quasidiagonal (it is certainly an amenable trace).  Using that $\mathcal R$ has a unique trace and the trace is faithful, this is equivalent to Problem~\ref{q:tracesqd}(\ref{q:tracesqd.2}).

\begin{question}\label{q:tracesqd}
\begin{enumerate}[(1)]
\item \label{q:tracesqd.1}
    Are amenable traces on {$C^*$}-algebras necessarily quasidiagonal? 
    \item  Is the hyperfinite II$_1$ factor quasidiagonal?\label{q:tracesqd.2}
    \end{enumerate}
\end{question}

Note that Problem~\ref{q:tracesqd}(\ref{q:tracesqd.1}) easily reduces to the case of separable {$C^*$}-algebras.  Further, for exact {$C^*$}-algebras, Problem~\ref{q:tracesqd}(\ref{q:tracesqd.1}) can be reduced to the case of faithful traces.  Indeed, every trace $\tau$ on a {$C^*$}-algebra $A$ induces a faithful trace on $\pi_\tau(A)$, 
where $\pi_\tau$ denotes the GNS representation.  Further, when $A$ is exact, amenability of $\tau$ is equivalent to injectivity of $\pi_\tau(A)''$,\footnote{That this holds follows from \cite[Corollary~4.3.4]{Brown:MAMS}, which is a consequence of Connes' theorem, and Kirchberg's theorem that exact {$C^*$}-algebras are locally reflexive (see \cite[Corollary 9.4.1]{BrownOzawa:book}).}
and so $\tau$ is amenable on $A$ if and only if the induced trace on $\pi_\tau(A)$ is amenable.  It follows that if Theorem~\ref{thm:AFembedding} holds without the UCT, then Problem~\ref{q:tracesqd} holds for exact {$C^*$}-algebras.\footnote{This reduction actually requires something a bit stronger than Theorem~\ref{thm:AFembedding}: namely, that the AF-embedding can be arranged so that the given faithful amenable trace extends to a trace on the AF algebra. This stronger statement is what is proven under the UCT in \cite{Schafhauser:Ann}.}

In addition to those traces covered by the quasidiagonality theorem (\cite{TWW:Ann,Gabe:JFA,Schafhauser:Crelle}), there is a positive answer for all amenable traces on cones.  This is essentially a result of Gabe (\cite{Gabe:JFA}) but is recorded as \cite[Propositon 3.2]{BCW:Abel}.  But outside these frameworks, little is known. For example, the amenable traces always form a face in the tracial states; are the quasidiagonal traces also a face?

The discussion above has focused mostly on the case of {$C^*$}-algebras with a faithful trace since most recent progress has been in this setting.  The situation where there are no faithful traces is also very much of interest. Indeed at the other extreme, the quasidiagonality and AF-embeddability problems have been solved (without a UCT requirement) in the complete absence of traces.  The appropriate traceless condition is made precise through the following definition.

\begin{definition}\label{Def:TracialWeight}
    A \emph{tracial weight} on a {$C^*$}-algebra $A$ is an additive and $\mathbb R_+$-ho\-mo\-ge\-neous function $\tau\colon A_+\to[0,\infty]$ with $\tau(xx^*)=\tau(x^*x)$ for all $x\in A_+$.\footnote{These are sometimes called \emph{extended traces}, for example by R\o{}rdam in \cite{Rordam:IJM} and by Gabe in \cite{Gabe:GAFA}.}  A {$C^*$}-algebra is called \emph{traceless} if all lower semicontinuous tracial weights only take the values $\{0,\infty\}$.\footnote{As discussed in Section \ref{sec:nonsimpleclass}, the $\{0,\infty\}$-valued lower semicontinuous tracial weights correspond to ideals; being traceless means that there are no other lower semicontinuous tracial weights. }
\end{definition}

While simple exact traceless {$C^*$}-algebras cannot be stably finite by Corollary~\ref{cor:BH}, and so cannot be quasidiagonal, outside the simple case, stably finite algebras can be traceless.  Indeed, an example is a cone $C_0((0,1])\otimes A$ over a simple purely infinite {$C^*$}-algebra $A$, which is stably finite by virtue of being projectionless.  All cones are quasidiagonal by Voiculescu's homotopy-invariance of quasidiagonality (\cite{Voiculescu:Duke}). Gabe's result goes much further and characterises quasidiagonality and AF-embeddability for traceless separable exact algebras in terms of the primitive ideal space.

\begin{theorem}[Gabe; {\cite[Theorem~C]{Gabe:GAFA}}]
    For a separable exact traceless {$C^*$}-algebra $A$, the following are equivalent:
    \begin{enumerate}
        \item $A$ is AF-embeddable;
        \item $A$ is quasidiagonal;
        \item $A$ is stably finite;
        \item the primitive ideal space $\mathrm{Prim}(A)$ contains no non-empty open compact subsets.
    \end{enumerate}
\end{theorem}

The origin of the AF-embedding problem is Pimsner and Voiculescu's famous AF-embedding of the irrational rotation algebras (\cite{PV:JOT}). This led to crossed products becoming a natural focus for these problems, and a result of Pimsner solves both the quasidiagonality and AF-embeddability problems for crossed products $C(X) \rtimes \mathbb Z$.
(In the absence of a faithful invariant probability measure on $X$, such a crossed product does not have a faithful trace.)

\begin{theorem}[Pimsner; {\cite{Pimsner:ETDS}}]\label{thm:pimsner:AFE}
    For a homeomorphism $\alpha$ of a compact metrisable space $X$, the following are equivalent:
    \begin{enumerate}
        \item\label{pim1} $C(X) \rtimes_\alpha \mathbb Z$ is AF-embeddable;
        \item\label{pim2} $C(X) \rtimes_\alpha \mathbb Z$ is quasidiagonal;
        \item\label{pim3} $C(X) \rtimes_\alpha \mathbb Z$ is (stably) finite;
        \item\label{pim4} there is no open set $U \subseteq X$ such that $\overline{\alpha(U)}$ is a proper subset of $U$.
    \end{enumerate}
\end{theorem}

It is natural to try to extend Theorem~\ref{thm:pimsner:AFE} to more general amenable groups.  An invariant probability measure on the space gives rise to a trace on the crossed product, which will be faithful when the measure is faithful (i.e.\ non-empty open sets have non-zero measure). In the presence of such a faithful invariant measure, Theorem~\ref{thm:AFembedding} applies (using Tu's Theorem~\ref{thm:Tu} to verify the UCT).
However, Pimsner's theorem shows that one should expect quasidiagonal (even AF-embeddable) crossed products well beyond this case.

\begin{question}
    Let $X$ be a compact metrisable space, let $d \geq 2$ be an integer, and let
    $\mathbb Z^d \curvearrowright X$ be an action.  
    If $C(X) \rtimes \mathbb Z^d$ is stably finite, must it be quasidiagonal?  In this case, must it also be AF-embeddable?
\end{question}

To the best of our knowledge, it is also not known when such a crossed product is even stably finite.  While the question is of interest for more general groups,  $\mathbb Z^d$ (or even just $\mathbb Z^2$) is a natural starting point.

Varying Pimsner's Theorem~\ref{thm:pimsner:AFE} in a different direction, Nate Brown solved the AF-embedding problem for crossed products $A \rtimes \mathbb Z$ for an AF algebra $A$ (\cite{Brown:JFA}).  For such crossed products, a direct analogue of Theorem~\ref{thm:pimsner:AFE} holds with Condition~\ref{pim4} replaced with a similar `incompressibility' condition on the induced action of $\mathbb Z$ on $K_0(A)$.  Brown's theorem was extended to automorphisms of real rank zero approximately homogeneous (AH) algebras in \cite{RainoneSchafhauser:Adv} but with the weaker conclusion of quasidiagonality in place of AF-embeddability. This appears to be one of the few classes where we currently know the answer to the quasidiagonality problem but not the AF-embeddability problem, prompting the following question.

\begin{question}
    Let $A$ be an AH algebra of real rank zero.  For an action $\alpha \colon \mathbb Z \curvearrowright A$ such that $A \rtimes \mathbb Z$ is quasidiagonal, is $A \rtimes \mathbb Z$ AF-embeddable?
\end{question}

\section{Strongly self-absorbing {$C^*$}-algebras}\label{sec:SSA}

Strongly self-absorbing algebras play an important role in modern structure and classification theory, from the use of $\mathcal Z$-stability as a unifying regularity property to the role of $\mathcal O_\infty$- and $\mathcal O_2$-stability in purely infinite algebras (whose importance largely emerged in Kirchberg's Geneva theorems, \cite{Kirchberg:ICM,KirchbergPhillips:Crelle}, including Theorem~\ref{Thm:Oinftystable}).
The role of tensorially absorbing a strongly self-absorbing algebra mirrors the use of $\mathcal R$-stability (a.k.a.\ the McDuff property) in II$_1$ factors. Indeed, being $\mathcal D$-stable for strongly self-absorbing $\mathcal D$ has a `McDuff-type' central sequence characterisation (\cite[Theorem~2.1]{TomsWinter:TAMS} or \cite[Proposition~4.11]{Kirchberg:Abel}).

A \emph{strongly self-absorbing} {$C^*$}-algebra is a unital separable {$C^*$}-algebra $\mathcal D\neq \mathbb C$, such that $\mathcal D \cong \mathcal D \otimes \mathcal D$ in a strong way.\footnote{Precisely, the first-factor embedding $\mathcal D \to \mathcal D \otimes \mathcal D$ is approximately unitarily equivalent to an isomorphism.}
These {$C^*$}-algebras, introduced independently by Kirchberg in \cite[Section~4]{Kirchberg:Abel} and by Toms and Winter in \cite{TomsWinter:TAMS},  have a number of convenient structural properties: simplicity, nuclearity, $\mathcal Z$-stability (\cite{Winter:JNCG}), and at most one trace.
Many of these structural properties arise from much earlier work of Effros and Rosenberg on {$C^*$}-algebras with an approximately inner flip (\cite{EffrosRosenberg:PJM}).
When $\mathcal D$ is strongly self-absorbing, the McDuff-type characterisation of $\mathcal D$-stability implies that a pair of strongly self-absorbing {$C^*$}-algebras are isomorphic if and only if they are (approximately) bi-embeddable. There is only a small list of known strongly self-absorbing {$C^*$}-algebras: the Cuntz algebras $\mathcal O_2$ and $\mathcal O_\infty$, the Jiang--Su algebra $\mathcal Z$, UHF algebras of infinite type, and tensor products of $\mathcal O_\infty$ with a UHF algebra of infinite type.
Through classification, it is known that these are the only strongly self-absorbing {$C^*$}-algebras satisfying the UCT.

In talks and in \cite{Winter:Abel}, Winter has highlighted the class of strongly self-absorbing {$C^*$}-algebras as a `microcosm' of general nuclear {$C^*$}-algebras, suggesting the following as a line of attack on the general UCT problem.

\begin{question}[{\cite[6.3]{Winter:Abel}}] \label{q:UCTssa}
    If $\mathcal D$ is a strongly self-absorbing {$C^*$}-algebra, does $\mathcal D$ satisfy the UCT?
    Equivalently, is $\mathcal D$ one of the known strongly self-absorbing {$C^*$}-algebras (listed above)?
\end{question}

Picking apart consequences of the UCT (such as the K\"unneth formula and quasidiagonality in the stably finite case) naturally breaks Problem~\ref{q:UCTssa} down into three subquestions.
Here is the first, which has both a finite version (which, by Kirchberg's dichotomy theorem and nuclearity of strongly self-absorbing algebras, is a special case of Problem \ref{q:QD}) and a general version.

\begin{question}\label{q:SSAQDQ}
    Let $\mathcal D$ be a strongly self-absorbing {$C^*$}-algebra.
    \begin{enumerate}[(1)]
        \item \label{it:SSAqd1}
        $($\cite[3.9]{Winter:Abel}$)$ If $\mathcal D$ is finite, must it be quasidiagonal? Equivalently, do we have $\mathcal Q \cong \mathcal Q \otimes \mathcal D$?\footnote{If $\mathcal D$ is quasidiagonal, this gives an embedding $\mathcal D\to \mathcal Q_\omega$, which can be upgraded to an embedding into $\mathcal Q_\omega \cap \mathcal Q'$ by the strong self-absorption of $\mathcal Q$. This is equivalent to $\mathcal Q$ being $\mathcal D$-stable by the McDuff-type characterisation.}
        \item \label{it:SSAqd2}
        $($\cite[Conjecture~4.7]{Kirchberg:Abel}, \cite[3.10]{Winter:Abel}$)$ In general, if $\mathcal D \not\cong \mathcal O_2$, does $\mathcal D$ embed into $(\mathcal Q \otimes \mathcal O_\infty)_\omega$? Equivalently, do we have $\mathcal Q\otimes \mathcal O_\infty \cong \mathcal Q \otimes \mathcal O_\infty \otimes \mathcal D$?
    \end{enumerate}
\end{question}

In a recent preprint, the first-named author has shown that a positive answer to Problem\ref{q:SSAQDQ}(\ref{it:SSAqd2}) gives a positive answer to Problem\ref{q:SSAQDQ}(\ref{it:SSAqd1}) (\cite[Corollary 4.6]{Schafhauser24}).

The K\"unneth formula for $K$-theory in \eqref{Kunneth2} holds when one of $A$ or $B$ satisfies the UCT, and so computes $K_*(A\otimes B)$ in terms of $K_*(A)$ and $K_*(B)$. When $A$ is strongly self-absorbing and satisfies the UCT, this imposes strong constraints on $K_*(A)$: $K_1(A)=0$, and if $K_0(A)$ is non-zero, it must be a subring of $\mathbb Q$.  The second subproblem below essentially asks whether this K\"unneth formula result applies generally to strongly self-absorbing {$C^*$}-algebras, modulo the previous question.

\begin{question}
    Let $\mathcal D$ be a strongly self-absorbing {$C^*$}-algebra which is $(\mathcal Q\otimes \mathcal O_\infty)_\omega$-embeddable.
    Must $\mathcal D$ have the same $K$-theory as a known strongly self-absorbing {$C^*$}-algebra?
\end{question}

The final subquestion asks for a classification of strongly self-absorbing algebras, modulo the previous questions.

\begin{question}
    Let $\mathcal D$ be a strongly self-absorbing {$C^*$}-algebra which is $(\mathcal Q\otimes \mathcal O_\infty)_\omega$-embeddable (or even quasidiagonal) and which has the same ordered $K$-theory as a known strongly self-absorbing {$C^*$}-algebra $\mathcal E$.
    Does it follow that $\mathcal D \cong \mathcal E$?
\end{question}

To illustrate the point of the last question, suppose $\mathcal D$ is a quasidiagonal strongly self-absorbing {$C^*$}-algebra with $K_0(\mathcal D)=\mathbb Z[\frac12]$ and $K_1(\mathcal D)=0$ (the same $K$-theory as the CAR algebra $M_{2^\infty}$); we think of $\mathcal D$ as a `putative CAR algebra'. Certainly the $K$-theory assumption ensures $M_{2^\infty}$ embeds into $\mathcal D$ (for this, it is enough that $\mathcal D$ has cancellation because it is $\mathcal Z$-stable -- by \cite{Winter:JNCG} -- and stably finite).   Quasidiagonality provides a unital embedding of $\mathcal D$ into $\mathcal Q_\omega$, but it is unclear whether the unital embedding into $\mathcal Q_\omega$ uses the right matrix approximations, i.e.\ lands in the subalgebra $(M_{2^\infty})_\omega$.
(Alternatively, quasidiagonality provides an embedding into $(M_{2^\infty})_\omega$, but this embedding is not a priori unital.)

The known strongly self-absorbing {$C^*$}-algebras almost group into pairs, where each such stably finite strongly self-absorbing {$C^*$}-algebra $\mathcal D$, has a purely infinite counterpart $\mathcal D\otimes\mathcal O_\infty$. The exception is $\mathcal O_2$: no stably finite unital {$C^*$}-algebra can have zero $K$-theory, so that a stably finite counterpart of $\mathcal O_2$ would need to be non-unital.  However, a theory of non-unital strongly self-absorbing {$C^*$}-algebras has not been satisfactorily developed; one issue is that when $A$ is non-unital, one doesn't have a canonical first (or second) factor embedding of $A$ into $A\otimes A$.  Potential examples of non-unital strongly self-absorbing algebras may include $\mathcal D\otimes\mathcal K$ when $\mathcal D$ is unital and strongly self-absorbing; these algebras appear as fibres in Dadarlat and Pennig's generalised Dixmier--Douady theory (\cite{DadarlatPennig:Crelle,DadarlatPennig:JNCG,DadarlatPennig:AGT}) but, as pointed out by some of the readers of the previous version of this paper, it is unclear whether they should be thought of as strongly self-absorbing.\footnote{One issue is that absorbing $\mathcal D\otimes\mathcal K$ will not necessarily pass to hereditary subalgebras.}  

Another important candidate is the Razak--Jacelon algebra $\mathcal W$, studied early on in \cite{Razak:CJM,Dean:CJM,Jacelon:JLMS}. This is $KK$-contractible and has a unique tracial state, and through classification, it is now known to be self-absorbing (\cite{EGLN:JGP} or \cite{Nawata:APDE}). To some extent, $\mathcal W$ is therefore a stably finite analogue of $\mathcal O_2$.  Although $\mathcal W$-stability passes to ideals and quotients,\footnote{As pointed out to us by the referee, this is a consequence of $\mathcal W$ being simple and exact rather than a deeper property of $\mathcal W$. Indeed, any ideal $I\lhd A\otimes \mathcal W$ will take the form $J\otimes\mathcal W$ with quotient $(A\otimes \mathcal W)/I\cong (A/J)\otimes\mathcal W$, so that ideals and quotients of $\mathcal W$-stable $C^*$-algebras are again $\mathcal W$-stable.  To see this, set $J=\{x\in A:\exists w\in\mathcal W,\ w\neq 0,\ x\otimes w\in I\}=\{x\in A:x\otimes \mathcal W\subseteq I\}$, the latter equality coming from simplicity of $\mathcal W$, so that $J\otimes\mathcal W\subseteq I\subseteq A\otimes \mathcal W$ and by exactness $(A\otimes\mathcal W)/(J\otimes\mathcal W)\cong (A/J)\otimes\mathcal W$. If $I$ strictly contains $J\otimes\mathcal W$, then by Kirchberg's slice lemma the image of $I$ in $(A/J)\otimes\mathcal W$ would contain a non-zero elementary tensor, leading to a contradiction. We do not know whether hereditary subalgebras of $\mathcal W$-stable algebras are $\mathcal W$-stable.} other permanence properties as in the first part of the problem below are unknown. This together with the second part addresses to what extent $\mathcal W$-stability behaves like $\mathcal D$-stability for a strongly self-absorbing $\mathcal D$, while the third asks for a concrete proof of self-absorption, which has been sought since \cite{Jacelon:JLMS}.

\begin{question}
\begin{enumerate}[(1)]
    \item Is $\mathcal W$-stability preserved by taking extensions or passing to hereditary subalgebras?
\item Is there a characterisation of $\mathcal W$-stability of a separable {$C^*$}-algebra $A$ in terms of central sequences?\footnote{One test for such a characterisation, is that it could be used to answer part (1) of this problem in the spirit of \cite[Section 4]{TomsWinter:TAMS}.  That said, since $\mathcal R^\omega$ (which is certainly not separably $\mathcal W$-stable as it is not unital), is a quotient of $\mathcal W_\omega$, there are reasons to believe that there is no such central sequence characterisation for $\mathcal W$-stability.}
\item Find a direct proof that $\mathcal W$ is self-absorbing -- not relying on major classification results --   either in the spirit of Jiang and Su's original argument for $\Z\cong\Z\otimes\Z$ or of the new proofs of this fundamental fact in \cite{Schemaitat:JFA,Ghasemi:BLMS}.
    \end{enumerate}
\end{question}

One challenge is that as $\mathcal W$ is non-unital, there is no natural analogue of the first factor embeddings $\mathrm{id}_{\mathcal D}\otimes 1_{\mathcal D}\colon \mathcal D\to\mathcal D\otimes\mathcal D$ and $\mathrm{id}_{A}\otimes 1_{\mathcal D}\colon A\to A\otimes\mathcal D$ which have been crucial to working with strongly self-absorbing $\mathcal D$.




\section{From strict comparison to $\Z$-stability}
\label{sec:TW}

Despite all the progress on the fine structure of simple nuclear {$C^*$}-algebras over the last decade, the final open implication in the Toms--Winter conjecture remains a major challenge.\footnote{We wish to emphasise that regularity should be viewed independently from the UCT.  The structure theorem (Theorem \ref{Thm:Structure}) and the Toms--Winter conjecture (Conjecture~\ref{TWConjecture}) do not have a UCT hypothesis, nor do its partial confirmations (Theorem~\ref{Thm:MS} and its generalisations).}

\begin{question}\label{q:TW}
Let $A$ be a simple separable nuclear non-elementary {$C^*$}-algebra with strict comparison. Must $A$ be $\mathcal Z$-stable?
\end{question}

In concrete settings, there are a number of families for which the Toms--Winter conjecture is known in full generality.  These include AH algebras (\cite{Toms:CRMASSRC}, heavily using \cite{Winter:IM12}; we will return to this later below) and crossed products $C(X)\rtimes \mathbb Z^d$ associated to free minimal actions on compact Hausdorff spaces (\cite[Corollary~7.14]{LiNiu:StableRank}, a preprint making use of \cite{Niu:CJM}; the $d=1$ case was obtained earlier in \cite{AL:MJM}). 
Note that both of these classes contain examples of both $\Z$-stable and non-$\Z$-stable {$C^*$}-algebras.  We return to crossed products in Section \ref{sec:crossed-product} below.

Matui and Sato's Theorem \ref{Thm:MS} played a major role in initiating the modern use of von Neumann techniques in {$C^*$}-algebras. Their strategy is to lift the corresponding von Neumann algebra result to the {$C^*$}-level through the central sequence trace-kernel extension. Since this underpins most known abstract results giving conditions when \eqref{it:TW.3}$\implies$\eqref{it:TW.2} holds in Conjecture \ref{TWConjecture}, let us describe it in more detail.

Let $A$ be a unital simple separable nuclear non-elementary stably finite {$C^*$}-algebra with a unique trace $\tau$, and write $\mathcal M\coloneqq\pi_\tau(A)''$, where $\pi_\tau$ is the GNS-representation associated to $\tau$ (by Connes' theorem, $\mathcal M$ is necessarily the hyperfinite II$_1$ factor, but we stick to $\mathcal M$ to make the point that the full force of Connes' theorem is not needed).  For a free ultrafilter $\omega\in\beta\mathbb N\setminus\mathbb N$, we can form the {$C^*$}-ultrapower $A_\omega$, and the tracial ultrapower $\mathcal M^\omega$,
and it is a consequence of Kaplansky's density theorem that $\pi_\tau$ induces a surjection $A_\omega\to\mathcal M^\omega$.  The kernel of this map, $J_A$, is called the \emph{trace-kernel ideal}.   More subtle is that $\pi_\tau$ still induces a surjection at the level of central sequence algebras\footnote{This is proved in \cite[Theorem~3.3]{KirchbergRordam:Crelle} using Kirchberg's $\sigma$-ideal techniques introduced in \cite{Kirchberg:Abel}, building on the earlier observations for nuclear {$C^*$}-algebras (using different methods) chaining back to \cite{Sato11}.} -- a fact we call \emph{central surjectivity} -- resulting in the short exact sequence
\begin{equation}\label{centralsurjectivity}
\begin{tikzcd}
    0 \arrow{r} &  J_A\cap A' \arrow{r} & A_\omega\cap A' \arrow{r} & \mathcal M^\omega\cap \mathcal M'\arrow{r} & 0.
\end{tikzcd}
\end{equation}
Since $\mathcal M$ is an injective von Neumann algebra, it is McDuff (the implication \eqref{it:TW.3}$\implies$\eqref{it:TW.2} in the Toms--Winter conjecture is analogous to this part of Connes' theorem), so there exists a unital embedding $M_2\hookrightarrow \mathcal M^\omega\cap \mathcal M'$. By a projectivity theorem of Loring (\cite{Loring:MathScand}), this embedding lifts to a c.p.c.\ order zero map $\phi$ as shown below:
\begin{equation}\label{ozliftequation}
\begin{tikzcd}
    0\arrow[r]&J_A\cap A'\arrow[r]&A_\omega\cap A'\arrow[r]&{\mathcal M^\omega}\cap \mathcal M'\arrow[r]&0.\\&&M_2\arrow[u,"\phi"',"\text{order zero}",dashed]\arrow[ur,hook]&&
\end{tikzcd}
\end{equation}
This is almost enough to show $\mathcal Z$-stability of $A$; a McDuff-type characterisation of $\Z$-stability requires that the defect $1_{A_\omega}-\phi(1_{M_2})$ of $\phi$ (measuring how much $\phi$ fails to be unital) is Cuntz below $\phi(e_{11})$ in $A_\omega \cap A'$ (\cite{Kirchberg:Abel,TomsWinter:TAMS,RordamWinter:Crelle}).\footnote{The characterisation given by these references appears formally stronger, requiring the existence of $v\in A_\omega\cap A'$ with $v^*v=1_{A_\omega}-\phi(1_{M_2})$ and $\phi(e_{11})v=v$.
Upon changing $\phi$ via functional calculus and using standard Cuntz semigroup techniques, this can be verified if we have $1_{A_\omega}-\phi(1_{M_2})\preceq \phi(e_{11})$.}
Here, the defect $1_{A_\omega}-\phi(1_{M_2})$ lies in the trace-kernel ideal $J_A$, so it is infinitesimal in trace, and strict comparison of $A$ ensures that it is Cuntz below $\phi(e_{11})$ in $A_\omega$.  Matui and Sato's breakthrough showed how to use nuclearity to upgrade strict comparison to a comparison condition for the central sequence algebra (known as \emph{property (SI)}.\footnote{Property (SI) is really a `small-to-large' comparison condition on the central sequence algebra, but more general comparison results follow from property (SI).}  Property (SI) is exactly what is needed to ensure the Cuntz subequivalence $1_{A_\omega}-\phi(1_{M_2}) \precsim \phi(e_{11})$ can be taken in $A_\omega\cap A'$, proving Theorem \ref{Thm:MS} in the unique trace case.

\begin{theorem}[{Matui--Sato; \cite{MatuiSato:Acta} -- cf.\ \cite[Corollary 5.11]{KirchbergRordam:Crelle}}]\label{Thm:SI}
Let $A$ be a unital simple separable nuclear non-elementary stably finite {$C^*$}-algebra with strict comparison.  Then $A$ has property (SI);
i.e.\ for positive contractions $e\in J_A\cap A'$ and $f=(f_n)_{n=1}^\infty\in A_\omega\cap A'$ with
\begin{equation}\lim_{k\to\infty}\lim_{n\to\omega}\inf_{\tau\in T(A)}\tau(f^k_n)>0,\footnote{This `largeness' condition on $f$ should be viewed as saying that the spectral projection of $f$ at $1$ is bounded below under all limit traces, or more precisely, $\inf_{k\in\mathbb N} \inf_{\tau \in T_\omega(A)} \tau(f^k) > 0$ (where $T_\omega(A)$ denotes the `limit traces' on $A_\omega$).}
\end{equation}
there exists $s\in A_\omega\cap A'$ with $s^*s=e$ and $fs=s$.
\end{theorem}

In addition to converting \eqref{ozliftequation} to $\Z$-stability, Theorem \ref{Thm:SI} leads to an important characterisation of $\Z$-stability for unital simple separable nuclear stably finite {$C^*$}-algebras.
This property is called \emph{tracial $\Z$-stability}, which asks for c.p.c.\ order zero maps $\phi\colon M_k\to A$ with approximately central range, whose defect $1_A-\phi(1_{M_k})$ can be taken to be arbitrarily small in the Cuntz semigroup (see \cite[Definition 2.1]{HO:JFA}).\footnote{The adjective `tracial' is in the spirit of Lin's tracial approximations.}  This was worked out by Hirshberg and Orovitz in \cite{HO:JFA}, and, via Kerr's $\Z$-stability theorem (Theorem~\ref{Thm:Kerr}, discussed in Section~\ref{sec:crossed-product}), has proved particularly important in applications to $\Z$-stability of crossed products.

To generalise the Matui--Sato approach to obtaining $\Z$-stability from strict comparison to the setting where $A$ has multiple traces, the naive thing to try is to take $\mathcal M$ to be the finite part of the bidual, $A^{**}_{\mathrm{fin}}$, which will be McDuff; the central surjectivity and order zero lifting part of Matui and Sato's argument work in this context, but unless $A$ has only finitely many extremal traces (as in Matui and Sato's Theorem \ref{Thm:MS}), one cannot expect the defect of the resulting map $\phi$ to be Cuntz below $\phi(e_{11})$ in $A_\omega$, let alone in the central sequence algebra.  Instead, one must work with an object capturing tracial behaviour uniformly rather than pointwise. This is achieved by taking $\mathcal M$ to be Ozawa's uniform tracial completion, $\overline{A}^{T(A)}$, of $A$.\footnote{The uniform trace (semi)norm is given by $\|a\|_{2,T(A)}\coloneqq \sup_{\tau \in T(A)} \tau(a^*a)^{1/2}$, and then the tracial completion $\overline{A}^{T(A)}$ is the {$C^*$}-algebra of norm-bounded $\|\cdot\|_{2,T(A)}$-Cauchy sequences, modulo the norm-bounded $\|\cdot\|_{2,T(A)}$-null sequences.}  Matui and Sato's argument proves the implication \eqref{it:TW.3}$\implies$\eqref{it:TW.2} in the Toms--Winter conjecture whenever the uniform tracial completion $\mathcal M\coloneqq \overline{A}^{T(A)}$ has the \emph{McDuff property} in the sense that there is an embedding of $M_2$ into the central sequence algebra $\mathcal M^\omega\cap \mathcal M’$.\footnote{The ultrapower here is defined using $\|\cdot\|_{2,T(A)}$.} 

In fact, one does not need to obtain the full strength of the McDuff property. Using the tools introduced in \cite{CETWW:IM} to prove the implication \eqref{it:TW.2}$\implies$\eqref{it:TW.1} of the Toms--Winter conjecture, it suffices to be able to produce approximately central projections which divide elements in trace.  The precise condition is given below; it is a property for the tracial completion analogous to Murray and von Neumann's property $\Gamma$ for II$_1$ factors. In the case when $T(A)$ is a Bauer simplex, it suffices to take $x\coloneqq 1_{\mathcal M}$ in Definition \ref{def:Gamma} (\cite[Proposition~5.27]{CCEGSTW}).

\begin{definition}[{\cite[Definition 5.19 and Proposition 5.23]{CCEGSTW}}]\label{def:Gamma}
Let $A$ be a {$C^*$}-algebra whose tracial state space $T(A)$ is non-empty and compact.  We say that the tracial completion $\mathcal M\coloneqq \overline{A}^{T(A)}$ has \emph{property $\Gamma$} when there is a projection $p\in \mathcal M^\omega\cap \mathcal M'$ with $\tau(px)=\tau(x)/2$ for all $x\in\mathcal M$ and all limit traces $\tau$ on $\mathcal M^\omega$.\footnote{Limit traces are those traces $\tau$ defined using a sequence $(\tau_n)_{n=1}^\infty$ from $T(A)$ by $\tau(x)\coloneqq \lim_{n\to\omega}\tau_n(x_n)$, where $(x_n)_{n=1}^\infty$ is a representative sequence of $x\in \mathcal M^\omega$.}
\end{definition}

Using density, property $\Gamma$ for $\overline{A}^{T(A)}$ can be phrased at the level of $A$ using \mbox{$\|\cdot\|_{2,T(A)}$}-approximately central approximate projections which approximately divide the trace of elements of $A$.  We say that $A$ has \emph{uniform property $\Gamma$} when its tracial completion has property $\Gamma$.  See the introductions to \cite{CETWW:IM} and \cite{CCEGSTW} for a description of the passage from uniform property $\Gamma$ to McDuffness of the tracial completion of a unital simple separable nuclear {$C^*$}-algebra.\footnote{This goes by way of a condition known as \emph{complemented partitions of unity}, which enables one to glue local behaviour at fibres $\pi_\tau(A)''$ over traces $\tau \in T(A)$ to (approximate) global behaviour.}  Combining Matui and Sato's technique with these ideas shows that property $\Gamma$ is the missing link in the Toms--Winter conjecture.

\begin{theorem}[{\cite[Theorem 5.6]{CETW:IMRN}}]\label{Thm:ZStableEquivGammaSC}
Let $A$ be a unital simple separable nuclear stably finite {$C^*$}-algebra. Then $A$ is $\Z$-stable if and only if $A$ has strict comparison and uniform property $\Gamma$.
\end{theorem}

In Connes’ theorem, the very first step is to obtain property $\Gamma$ for a II$_1$ factor from semidiscreteness (\cite[Corollary~2.2]{Connes:Ann}) by means of a spectral gap argument (see \cite{Marrakchi18,Marrakchi19} for new approaches to this implication), but it remains a mystery whether uniform property $\Gamma$ can be obtained from nuclearity.\footnote{We state the question for simple {$C^*$}-algebras, but it is also open for {$C^*$}-algebras with no finite-dimensional representations.}  By Theorem \ref{Thm:ZStableEquivGammaSC}, a positive answer to Problem~\ref{q:Gamma} would imply a positive answer to Problem~\ref{q:TW}.

\begin{question}[{\cite[Question C]{CETW:IMRN}}]\label{q:Gamma}
Does every unital simple separable nuclear non-el\-e\-men\-tary stably finite  {$C^*$}-algebra have uniform property $\Gamma$?
\end{question}

Uniform property $\Gamma$ is known to hold for many unital simple separable nuclear stably finite {$C^*$}-algebras. For example, in the setting of dynamics, property $\Gamma$ holds for crossed products $C(X)\rtimes G$ associated to free minimal actions $G\curvearrowright X$ with the small boundary property of amenable groups $G$ on compact Hausdorff spaces $X$ (\cite[Theorem 9.4]{KerrSzabo:CMP}, stated as Theorem~\ref{Thm:KerrSzabo} below).
With hindsight, uniform property $\Gamma$ underpins the the extension of Matui and Sato's Theorem \ref{Thm:MS} to prove the implication \eqref{it:TW.3}$\implies$\eqref{it:TW.2} of the Toms--Winter conjecture whenever the extremal boundary $\partial_eT(A)$ of $T(A)$ is compact and has finite covering dimension (\cite{KirchbergRordam:Crelle,TWW:IMRN,Sato12}).\footnote{Ozawa formalised the uniform tracial completion after these papers. In the case when $T(A)$ is a Bauer simplex with boundary $K$, then the uniform tracial completion of $A$ is a $W^*$-bundle over $K$. When $A$ is nuclear and has no finite-dimensional representations, the fibres of this bundle are the hyperfinite II$_1$ factor.  Ozawa then showed that any $W^*$-bundle with fibre $\mathcal R$ over a finite-dimensional compact space will have property $\Gamma$.}

\begin{theorem}\label{FiniteDimBoundaryImpliesGamma}
Let $A$ be a unital simple separable nuclear non-el\-e\-men\-tary stably finite {$C^*$}-algebra whose tracial state space has a compact extremal boundary of finite covering dimension.  Then $A$ has uniform property $\Gamma$. Accordingly, the Toms--Winter conjecture holds for such algebras.
\end{theorem}

There are now more general results where, under trace space conditions, the implication \eqref{it:TW.3}$\implies$\eqref{it:TW.2} of the Toms--Winter conjecture holds (\cite{Lin:AIM,Lin:arXiv22,Zhang:JFA}).
We expect, but have not checked, that simple separable nuclear non-elementary {$C^*$}-algebras with these trace space conditions automatically have uniform property $\Gamma$.

Since uniform property $\Gamma$ is automatic when the trace space of $A$ is small, there are certainly non-$\mathcal Z$-stable simple nuclear {$C^*$}-algebras which have uniform property $\Gamma$ (such as Villadsen’s second type examples with unique trace and higher stable rank from \cite{Villadsen:JAMS}).  But, for reasons we explain further in Section \ref{sec:TW2}, we are particularly interested in the stable rank one case. We have been unable to determine whether the non-$\Z$-stable simple nuclear {$C^*$}-algebras with large trace spaces have uniform property $\Gamma$.\footnote{In \cite{CETW:IMRN}, two of us asserted that the relevant {$C^*$}-algebras did have uniform property $\Gamma$ -- this was based on an erroneous calculation of their trace spaces.  See the discussion following \cite[Question~5.29]{CCEGSTW}. In fact this problem has been a bit of a trap; the first version of the preprint \cite{ElliottNiu25} claimed that the relevant $C^*$-algebras do not have property $\Gamma$, but that argument, while introducing a lot of very promising ideas, also contains a gap.}

\begin{question}\label{q:VilladsenGamma}
    Do the Villadsen algebras of the first type (from \cite{Villadsen:JFA}) have uniform property $\Gamma$?
\end{question}

Given a Villadsen algebra $A$ of the first type which fails to have strict comparison, let $D$ be the canonical Cartan subalgebra in $A$ (coming from the inductive limit construction).  A recent preprint of Elliott and Niu shows that the inclusion $D\subseteq A$ fails to have uniform property $\Gamma$; i.e.\ one cannot find projections $p$ witnessing property $\Gamma$ for $\overline{A}^{T(A)}$ inside the subalgebra $\overline{D}^{T(A)}$ (\cite[Theorem~1.1 and Proposition~4.8]{ElliottNiu24}).
This rules out one natural way to try to prove property $\Gamma$ for this example -- one might have hoped to find the approximately central projection inside $\overline{D}^{T(A)}$ as one can certainly divide the unit in an approximately central fashion there.
Some of the authors view this as evidence that Problem~\ref{q:VilladsenGamma} has a negative answer.

Elliott and Niu's work also shows that for Villadsen algebras of the first type, tractability of the trace space ensures classifiability (\cite[Theorem 1.4]{ElliottNiu24}): these {$C^*$}-algebras are $\Z$-stable when they have a Bauer simplex of traces.\footnote{The particular form of these algebras guarantees that the unit is approximately centrally divisible in trace, so that when the trace space is Bauer, they automatically have uniform property $\Gamma$. See \cite[Proposition 5.10]{CETW:IMRN} for diagonal AH algebras (though note that the paragraph following this proposition is now known to be wrong) and the sketch found in the proof of \cite[Theorem~4.6]{ElliottNiu24} for crossed products $C(X)\rtimes\mathbb Z^d$ arising from free and minimal actions, using arguments from the preprint \cite{LiNiu:StableRank}. In fact, this latter result works for all algebras $C(X)\rtimes \Gamma$ coming from free and minimal actions of amenable groups with Niu's uniform Rokhlin property from \cite{Niu:JAM}.} 

In another direction, Lin's recent notion of a stably finite $C^*$-algebra having \emph{tracial approximate oscillation zero} (and various modifications of this concept developed by Fu and Lin  in \cite{FuLin:CJM}) power his stable rank one theorem (Theorem \ref{thm:Lin} from \cite{Lin23}, described in Section \ref{Sec:CuReg}), and his examples of further trace spaces for which \eqref{it:TW.3}$\implies$\eqref{it:TW.2} of the Toms--Winter conjecture holds (\cite{Lin:AIM,Lin:arXiv22}). The original definition of tracial approximate oscillation zero is a bit technical, but has now been cleanly reformulated as real rank zero of the uniform tracial ultrapower (\cite[Section 6]{FuLin:CJM} and \cite{Fu:Preprint}).  

Just as Connes' original proof that injective II$_1$ factors have property $\Gamma$ crucially uses arbitrarily small projections, we wonder whether such a real rank zero condition in fact characterises uniform property $\Gamma$.  From \cite{PereraRordam:JFA}, real rank zero for the tracial ultrapower should entail a certain amount of divisibility of the unit, giving rise to small projections. 

\begin{question}\label{NewRR0Problem}
    Let $A$ be a unital simple separable nuclear non-elementary stably finite {$C^*$}-algebra such that $(\overline{A}^{T(A)})^\omega$ has real rank zero.  Must $A$ have uniform property $\Gamma$?
\end{question}

We suspect Problem \ref{NewRR0Problem} really lives at the level of tracially complete algebras as formalised in \cite{CCEGSTW}. In that framework, the formulation becomes: given a type II$_1$ factorial amenable tracially complete {$C^*$}-algebra $(\mathcal M,X)$ such that $(\mathcal M,X)^\omega$ has real rank zero, does $(\mathcal M,X)$ have property $\Gamma$?

\section{Intermezzo: the trace problem}

Having introduced the uniform tracial completion in the previous section, let us take an interlude to describe a problem in the foundations of this subject. In \cite{Ozawa:JMSUT}, Ozawa introduced the notion of $W^*$-bundles (topological bundles of finite von Neumann algebras) as an abstract framework for studying the uniform tracial completions of {$C^*$}-algebras whose tracial state space is Bauer (i.e.\ has a compact extremal boundary). In the preprint \cite{CCEGSTW}, together with Castillejos, Carri\'on, Evington, and Gabe, we developed an abstract framework for the study of uniform tracial completions and established structure and classification results for these.

By construction, traces on $A$ extend by $\|\cdot\|_{2,T(A)}$-continuity to traces on its uniform tracial completion $\overline{A}^{T(A)}$, and indeed, $T(A)$ identifies with the set of \mbox{$\|\cdot\|_{2,T(A)}$}-continuous traces on $\overline{A}^{T(A)}$. In our work, we have been much irked by the question as to whether these are all the traces; in particular, we would very much like the uniform tracial completion $\overline{A}^{T(A)}$ to be uniformly tracially complete with respect to all of its traces.

\begin{question}[{cf.\ \cite[Question~1.1]{CCEGSTW}}]\label{Q6}
    Let $A$ be a {$C^*$}-algebra with $T(A)$ non-empty and compact. Are all traces on $\overline{A}^{T(A)}$ automatically $\|\cdot\|_{2,T(A)}$-continuous? Equivalently, is the canonical embedding $T(A)\subseteq T\big(\overline{A}^{T(A)}\big)$ an equality? 
\end{question}

In \cite{CCEGSTW}, we set the question out in even more generality (explaining how it can be viewed as analogous to the fact that the trace on a II$_1$ factor is unique amongst all traces, and not just amongst the normal traces). When $T(A)$ is a finite-dimensional simplex, $\overline{A}^{T(A)}$ is the finite part of the bidual of $A$ -- a finite direct sum of factors. Here the problem has a positive answer -- essentially because in this case, all traces on the finite part of the bidual are normal.  In \cite{Evington:draft}, Evington gives a positive solution to Problem~\ref{Q6} assuming that $\overline{A}^{T(A)}$ has the regularity property of complemented partitions of unity (this is a technical property which follows from  property $\Gamma$; see \cite[Section 6]{CCEGSTW}). Since the first version of this paper, Farah and Vaccaro have found further situations where the trace problem has a positive answer, such as when $T(A)$ is a Bauer simplex whose extremal boundary is at most $1$-dimensional (\cite{FarahVaccaro:arXiv}). In general, the question is open, and it is also open when $A$ is nuclear (although, the only reason we would expect nuclearity to simplify things is if Problem~\ref{q:Gamma} had a positive answer).

\section{Ranks of operators}
\label{sec:TW2}

The fundamental fact that the trace $\tau$ on a II$_1$ factor $\mathcal M$ induces an identification of the Murray--von Neumann classes of projections in $\mathcal M$ with $[0,1]$ amounts to two conditions:
\begin{enumerate}[(a)]
    \item for projections $p,q\in\mathcal M$ with $\tau(p)\leq \tau(q)$, we have $p\precsim q$;\label{Proj1}
    \item for any $t\in [0,1]$, there exists a projection $p\in\mathcal M$ with $\tau(p)=t$.\label{Proj2}
\end{enumerate}
Of these, strict comparison is the appropriate {$C^*$}-algebra analogue of \eqref{Proj1}, but what about \eqref{Proj2}?  We will describe this for a unital simple {$C^*$}-algebra $A$, so we can work with the collection of normalised (quasi)traces $QT(A)$; however, all of what we have to say applies generally by working with the cone of densely defined lower semicontinuous quasitraces; see \cite{ERS:AJM,APRT:Duke}.  Every positive element $a\in A\otimes \mathcal K$ gives rise to a lower semicontinuous affine function
\begin{equation}\label{DefRank}
\mathrm{Rank}(a)\colon QT(A)\to [0,\infty]\colon \tau \mapsto d_\tau(a).
\end{equation}
(here $\tau$ is extended canonically to a densely defined lower semicontinuous quasitrace on $A\otimes \mathcal K$) which will be strictly positive whenever $a\neq 0$, since $A$ is simple. As this function measures the rank of $a$ in all (quasi)traces, it is viewed as the \emph{rank function}\footnote{In the literature, the rank function is sometimes simply called the ``rank'' of $a$. Also, note that in \cite{BlackadarHandelman:JFA}, Blackadar and Handelman use the terminology ``rank function'' to mean something different, though their meaning of the term is no longer is use.}
 induced by $a$.  Just as \eqref{Proj2} gives that all possible trace values of projections occur in a II$_1$ factor, the appropriate {$C^*$}-analogue would be for all possible rank functions to arise from positive elements.  To our knowledge, this problem was first made explicit in talks by Nate Brown in the late 2000s motivated by \cite[Section~5]{BPT:Crelle}.
 
\begin{question}[Rank Problem]\label{q:Rank}
    Let $A$ be a unital simple separable stably finite and non-elementary  {$C^*$}-algebra.  Given any lower semicontinuous affine map $f\colon QT(A)\to (0,\infty]$, does there exist a positive element $a\in A\otimes\mathcal K$ such that $d_\tau(a)=f(a)$ for all $\tau\in QT(A)$?
\end{question}

We say that \emph{all ranks occur} in a {$C^*$}-algebra $A$ when it has a positive answer to Problem \ref{q:Rank}. This is the case when $A$ has a unique (quasi)trace,\footnote{It is straightforward to see that the set of realised ranks (in this case a subset of $[0,\infty]$) is dense. Additionally, the ranks that occur are always closed under infinite sums.} and this extends to {$C^*$}-algebras with finitely many extremal quasitraces.

  The connection between the rank problem and the Toms--Winter conjecture comes as all ranks occur in simple $\Z$-stable algebras (\cite[Corollary 6.8]{ERS:AJM}, extending \cite[Theorem 5.5]{BPT:Crelle}); thus if \eqref{it:TW.3}$\implies$\eqref{it:TW.2} holds in Conjecture \ref{TWConjecture}, one must obtain a solution to the rank problem in the simple nuclear case from strict comparison. We will return to this in the next section. Dadarlat and Toms (\cite{DadarlatToms:JFA}) showed that all ranks occur in a unital simple separable non-elementary stably finite {$C^*$}-algebra $A$ with strict comparison whose quasitraces are traces\footnote{Dadarlat and Toms set up the rank problem using traces, and implicitly work with strict comparison with respect to traces.} satisfying $T(A)$ has compact extreme boundary of finite covering dimension.  Note that this is the same condition on traces as in Theorem \ref{FiniteDimBoundaryImpliesGamma}, and both results are instances of a fairly typical phenomenon: results that hold for unital simple {$C^*$}-algebras with unique trace, can sometimes be extended to unital simple {$C^*$}-algebras whose trace space has a compact extreme boundary of finite covering dimension by means of partition of unity arguments.
  
 A weaker version of the rank question asks whether the uniform closure of the rank functions associated to positive elements in $A\otimes\mathcal K$ contains all continuous affine maps $QT(A)\to (0,\infty)$ -- in this case, we say that \emph{all ranks almost occur}.  This holds for unital simple non-elementary AH algebras\footnote{This is implicit in the proof of \cite[Theorem 5.3]{BPT:Crelle}.  Note that while this theorem has a stable rank one hypothesis, it is only used to ensure the existence of suprema of bounded increasing sequences in the incomplete Cuntz semigroup used in that paper.  One of the main points of working with the complete Cuntz semigroup $\Cu$ from \cite{CEI:Crelle} is that every increasing sequence in $\Cu(A)$ has a suprema.\label{Footnote24}} and for unital simple non-elementary ASH algebras of slow dimension growth (\cite{Toms:IM}). It is open, and seemingly challenging, whether some condition like slow dimension growth is necessary in the ASH case.

\begin{question}\label{ASH.Ranks}
Do all ranks almost occur in every unital simple non-elementary ASH algebra?
\end{question}

A major breakthrough was made by Thiel in 2017, demonstrating a surprising connection to stable rank. Stable rank one is a strong form of finiteness, characterised (in the unital case) by the density of invertible elements (which accordingly obstructs the existence of a proper isometry).   Simple $\Z$-stable stably finite {$C^*$}-algebras have stable rank one (\cite[Theorem 6.5]{Rordam:IJM} gives the unital case; the non-unital case was obtained only recently in \cite{FLL:JLMS}).

\begin{theorem}[{Thiel; \cite[Theorem 8.11]{Thiel:CMP}}]\label{RankThm}
Problem \ref{q:Rank} has a positive answer when $A$ has stable rank one.
\end{theorem}

Theorem \ref{RankThm} was later extended by Antoine, Perera, Robert, and Thiel to  unital {$C^*$}-algebras of stable rank one which are non-simple (\cite[Theorem~7.14]{APRT:Duke}): all ranks occur (provided the {$C^*$}-algebra has no non-zero finite-dimensional quotients). Both of these papers take full advantage of the vast amount of work developing an axiomatic approach to the Cuntz semigroup (initiated in \cite{CEI:Crelle}; see particularly the memoir \cite{APT:MAMS}) and demonstrate the power these tools provide.

On the subject of stable rank one, it is unclear how else this property might give rise to regularity, noting that Villadsen’s first construction of exotic simple nuclear {$C^*$}-algebras with perforation in $K$-theory from \cite{Villadsen:JFA} all have stable rank one (\cite{EHT:JFA}). Crossed products of the form $C(X)\rtimes \mathbb Z^d$ also have stable rank one whenever they are simple (\cite{AL:MJM,LiNiu:StableRank}), and like Villadsen's examples, these can also be non-$\mathcal Z$-stable (\cite{GK:Crelle}).
That said (and unlike the second type of Villadsen counterexamples with higher stable rank), the counterexamples of the first type all have large trace simplices -- the Poulsen simplex (\cite[Theorem~4.5]{ElliottLiNiu:JFA}).
This provokes the following question.

\begin{question}
    Let $A$ be a simple separable nuclear non-elementary {$C^*$}-algebra with unique trace and with stable rank one.
    Must $A$ be $\mathcal Z$-stable?
\end{question}

\section{Cuntz semigroup regularity}\label{Sec:CuReg}

In the absence of a complete resolution to \eqref{it:TW.3}$\implies$\eqref{it:TW.2} in the Toms--Winter conjecture (Problem \ref{q:TW}), it is very natural to seek any positive element characterisation of $\Z$-stability within the class of simple separable nuclear non-elementary {$C^*$}-algebras, i.e.\ a characterisation at the level of the Cuntz semigroup.  There is an obvious necessary condition for $A$ to be $\Z$-stable: namely that $\Z$-stability holds at the level of the Cuntz semigroup. A striking theorem of Winter shows that this is sufficient under the additional hypothesis of locally finite nuclear dimension (Theorem \ref{Thm:Winter2} was extended to the non-unital case in \cite{Tikuisis:MA}).\footnote{$A$ has locally finite nuclear dimension when, for any finite subset $\mathcal F\subset A$ and $\e>0$, there exists a subalgebra $B\subseteq A$ of finite nuclear dimension such that $\mathcal F\subset_\e B$, i.e.\ for all $x\in \mathcal F$, there exists $y\in B$ with $\|x-y\|<\e$.  The point of course is that $\dim_{\mathrm{nuc}}(B)$ can depend on $\mathcal F$ and $\e$. The locally finite nuclear dimension hypothesis, while technical, should be thought of as providing an abstract version of being an ASH algebra.  Indeed, any inductive limit of subhomogeneous algebras could be replaced with one where the building blocks are finitely generated and hence have finite-dimensional spectrum.  It is not known whether every separable nuclear {$C^*$}-algebra has locally finite nuclear dimension.}

\begin{theorem}[{Winter; \cite{Winter:IM12}}]\label{Thm:Winter2}
Let $A$ be a unital simple separable stably finite\footnote{The theorem holds without the stably finite assumption and, in the infinite case, with the a priori weaker condition of nuclearity in place of locally finite nuclear dimension. Indeed, $\Cu(A)\cong \Cu(A\otimes\Z)$ ensures that $A$ has strict comparison, then one can use Kirchberg's results collected as Theorem \ref{Thm:Tracelessstructure}.   Winter's paper \cite{Winter:IM12} claims that (at least in the setting of locally finite nuclear dimension) one can simultaneously handle the tracial and traceless cases.  Unfortunately, as explained in \cite{RobertTikuisis:TAMS} (see \cite[Example~3.5]{RobertTikuisis:TAMS}), it does not presently seem possible to unify the proofs in the way Winter intended, and Winter's theorem only proves the stably finite case.\label{Foot43}} {$C^*$}-algebra with locally finite nuclear dimension. Then $A\cong A\otimes \mathcal Z$ if and only if $\Cu(A)\cong \Cu(A\otimes\mathcal Z)$.
\end{theorem}

While we are unsure as to whether \eqref{it:TW.3}$\implies$\eqref{it:TW.2} in the Toms--Winter conjecture (Problem \ref{q:TW}) will hold, we, and many others, expect that locally finite nuclear dimension is not needed in the above theorem, producing the following question (which we believe was first asked by Winter in his 2012 CBMS lectures).

\begin{question}\label{Q7}
If $A$ is a unital simple separable stably finite nuclear {$C^*$}-algebra such that $\Cu(A)\cong \Cu(A\otimes \mathcal Z)$, does it follow that $A$ is $\mathcal Z$-stable?
\end{question}


So it is natural to wonder for simple {$C^*$}-algebras, what is the difference between $\Cu(A)\cong\Cu(A\otimes\mathcal Z)$ and strict comparison? It turns out that if there is a difference between these properties, then this must lie in divisibility conditions for the Cuntz semigroup.
It is too much to expect to be able to divide a positive element $x\in A_+$ exactly in half (or into $n$ pieces for any $n\in\mathbb N$); for example, there is no element $x\in\Cu(\Z)$ with $2x=[1_\Z]$. Instead, one can ask for \emph{almost-divisibility}: for all $x\in \Cu(A)$ and $n\in\mathbb N$, there exists $y\in \Cu(A)$ with $ny\leq x\leq (n+1)y$.  Winter called a simple {$C^*$}-algebra \emph{pure} if it has strict comparison and its Cuntz semigroup is almost divisible. Since $\Z$-stable {$C^*$}-algebras have almost divisible Cuntz semigroups (a result essentially going back to \cite{Rordam:IJM}), a simple {$C^*$}-algebra $A$ with $\Cu(A)\cong\Cu(A\otimes\Z)$ is pure. 

Problem \ref{Q7} predicts that, at least in the nuclear case, properties of simple separable $\Z$-stable {$C^*$}-algebras should in fact follow from pureness. An important facet to consider is R\o{}rdam's theorem that a unital simple $\Z$-stable  finite {$C^*$}-algebra has stable rank one (\cite[Theorem 6.7]{Rordam:IJM}, extending his earlier UHF-stable result from \cite{Rordam:JFA91}). This was extended to the non-unital setting, initially through Robert's notion of `almost stable rank one' of non-unital simple $\Z$-stable finite {$C^*$}-algebras (\cite{Robert:AIM}), and more recently by Fu, Li, and Lin to obtain the full force of stable rank one for these algebras (\cite[Corollary 6.8]{FLL:JLMS}).  In a notable breakthrough, Lin was able to obtain the following Cuntz semigroup version of R\o{}rdam's theorem. As a consequence, he obtains a dichotomy theorem, since traceless pure {$C^*$}-algebras are purely infinite.

\begin{theorem}[{Lin \cite[Corollary~1.3]{Lin23}}]\label{thm:Lin}
    Let $A$ be a simple separable finite pure {$C^*$}-algebra. Then $A$ has stable rank one. In particular, simple separable pure {$C^*$}-algebras are either stably finite or purely infinite.
\end{theorem}

We view Lin's theorem (together with Theorem \ref{Thm:Tracelessstructure}) as providing even more evidence on top of Theorem \ref{Thm:Winter2} that Problem \ref{Q7} should have a positive answer.

For simple separable finite $C^*$-algebras, it turns out that pureness is precisely the condition of $\Z$-absorption at the level of the Cuntz-semigroup,  $\Cu(A)\cong\Cu(A\otimes \Z)$. We had believed this result to be folklore before we looked into it. As noted above, the passage from $\Z$-stability at the level of the Cuntz semigroup to pureness is a straightforward application of R\o{}rdam's work on $\Z$-stable {$C^*$}-algebras.  To go back, we want to use the computation of the Cuntz semigroup of a simple pure $C^*$-algebra $A$ (going back to \cite{PereraToms:MA,BPT:Crelle}) as $V(A)\sqcup \mathrm{LAff}_{> 0}(QT(A))$, where $QT(A)$ here means the cone of densely defined lower semicontinuous $2$-quasitraces, and $\mathrm{LAff}_{> 0}(QT(A))$ the lower semicontinuous affine functions $QT(A)\to[0,\infty]$ which are strictly positive except at $0$.   However, to arrange for this to match up with the corresponding computation for $A\otimes \Z$, we need to know that $QT(A)\cong QT(A\otimes\mathcal Z)$ and $V(A)\cong  V(A\otimes\Z)$ canonically.  The latter is not automatic and our argument below requires strict comparison and cancellation of projections.
\begin{lemma}\label{CR:Lem}
Let $A$ be a separable $C^*$-algebra.  
\begin{enumerate}
    \item 
The first factor embedding $A\to A\otimes\Z$ induces an affine homeomorphism $QT(A)\cong QT(A\otimes\Z)$;\footnote{In general if $B$ is a simple $C^*$-algebras with a unique tracial state, we do not know whether $QT(A)\cong QT(A\otimes B)$ (with the minimal tensor product).  Indeed, taking $B=C^*_r(\mathbb F_2)$, Haagerup's work \cite{Haagerup:CRMASSRC} shows that for a unital $C^*$-algebra $A$, the conditions that $A$ has no tracial state and $A\otimes C^*_r(\mathbb F_2)$ is properly infinite are equivalent (see also \cite[Theorem 3.1]{MilhojRordam}).  Therefore, with thanks to the referee for pointing this out, Kaplansky's conjecture that all quasitraces are traces (or equivalently all type II$_1$ AW$^*$-factors are von Neumann algebras) is equivalent to the first factor embedding inducing an affine homeomorphism $QT(A)\cong QT(A\otimes C^*_r(\mathbb F_2))$. (If there exists an AW$^*$ type II$_1$ factor $A$ which is not a von Neumann algebra, then it has a quasitrace but no trace, so that $QT(A\otimes C_r^*(\mathbb F_2))=\emptyset$ from proper infiniteness). We wonder how generally $QT(A\otimes B)$ can be computed as the convex product of $QT(A)$ and $T(B)$ when $B$ is nuclear?  Might it help if $B$ has (locally) finite nuclear dimension?}
\item Let $A$ be simple, stably finite, with strict comparison and cancellation of projections.  Then the first factor embedding $A\to A\otimes\Z$ gives an isomorphism $V(A)\to V(A\otimes\Z)$.
\end{enumerate}
\end{lemma}
\begin{proof}
There are many ways to prove the first part. To our knowledge it is implicitly contained in \cite{R:Scand}, through a Cuntz semigroup approach. Firstly, Elliott, Robert and Santiago established a canonical identification of $QT(A)$  with the cone $F(\Cu(A))$ of functionals on $\Cu(A)$ from \cite[Theorem 4.4]{ERS:AJM}. Then the combination of \cite[Proposition 3.1.1]{R:Scand}, \cite[Theorem 5.2.1]{R:Scand} shows that 
\begin{equation}
    F(\Cu(A))\cong F(\Cu(A)_{\mathbb R})\cong F(\Cu(A\otimes\mathcal W))
\end{equation} (see \cite[Section 3]{R:Scand} for the definition of $\Cu(A)_{\mathbb R}$, and note that in \cite{R:Scand}, the algebra $\mathcal W$ is denoted by $\mathcal R$). That is $QT(A)\cong QT(A\otimes\mathcal W)$. Since $A\otimes\mathcal W$ is $\Z$-stable, the first factor embedding $A\otimes \mathcal W\to A\otimes \mathcal W\otimes\mathcal Z$ is approximately unitarily equivalent to an isomorphism and hence induces an isomorphism on the Cuntz semigroup. In this way the first factor embedding $A\to A\otimes\mathcal Z$ induces an isomorphism $QT(A)\cong QT(A\otimes\mathcal Z)$.

For the second part,  suppose $A$ is simple and has strict comparison and cancellation of projections. For projections $p,q\in M_n(A)$  with $p\otimes 1_{\mathcal Z} \sim q \otimes 1_{\mathcal Z}$, as the map $K_0(A)\to K_0(A\otimes\mathcal Z)$ is injective, we have $[p]_0=[q]_0$ in $K_0(A)$, and hence $p\sim q$ by cancellation in $A$.  Accordingly the map $V(A)\to V(A\otimes\mathcal Z$) induced by the first factor embedding is injective.  

For surjectivity, we first reduce to the case $A$ is unital.  If $V(A \otimes \mathcal Z) = 0$, there is nothing to show.  Otherwise, there is a non-zero projection $r \in M_m(A\otimes \mathcal Z)$ for some integer $m \geq 1$.  Writing $\mathcal Z$ as a direct limit of dimension drop algebras $\mathcal Z_{k, k+1}$, the stability of projections implies there is a non-zero projection $r_0 \in M_m(A \otimes \mathcal Z_{k, k+1})$ for some $k$.  Evaluating at a point in $[0, 1]$ gives a non-zero projection in $r_1 \in M_m(A \otimes M_k \otimes M_{k+1})$.  Then as $A$ is simple and separable, $A$ is stably isomorphic to the unital $C^*$-algebra $r_1 M_n(A \otimes M_k \otimes M_{k+1})r_1$.  Since the Murray--von Neumann semigroup is a stable isomorphism invariant, we may assume $A$ is unital.

Now, assuming again $r \in A \otimes \mathcal Z$ is a non-zero projection, since $K_0(A)\to K_0(A\otimes \mathcal Z)$ is surjective, there are an integer $m \geq 1$ and projections $p, q \in M_m(A)$ such that $[r]_0=[q\otimes 1_{\mathcal Z}]_0-[p\otimes 1_{\mathcal Z}]_0$. 
Then for $\tau \in QT(A)$, $\tau(q-p)=(\tau\otimes\tau_\Z)(r)>0$ so by strict comparison for $A$, we have $p\precsim q$ in $M_m(A)$.  Therefore, there exists a projection $s\in M_m(A)$ with $p\oplus s\sim q$.  Applying cancellation in $A\otimes \mathcal Z$ (which follows as $A\otimes\Z$ has stable rank one),\footnote{For unital simple $\Z$-stable $C^*$-algebras, stable rank one was obtained in \cite[Theorem 6.7]{Rordam:IJM}.} we have $s\otimes 1_\Z\sim r$.  Since the orders on $V(A)$ and $V(A\otimes\mathcal Z)$ are algebraic (i.e., $x\leq y$ if and only if there exists $z$ with $x+z=y$), the monoid isomorphism induced by the first factor embedding is also an order isomorphism.
\end{proof}

In order to get comparison of projections from pureness so as to use the previous lemma, we found ourselves going through Lin's Theorem \ref{thm:Lin}, using \cite{Rieffel:PLMS} to obtain cancellation from stable rank one. In the literature, the implication that simple pure {$C^*$}-algebras absorb $\Z$ at the level of the Cuntz semigroup has been attributed to Toms, but Toms' result (\cite[Theorem 1.2]{Toms:IM}) is for simple approximately subhomogeneous algebras of slow dimension growth, where this cancellation was established by Phillips (\cite[Theorem 0.1]{Phillips:TAMS}).  We would be interested to know if there is a Cuntz semigroup-theoretic way to obtain cancellation of projections from pureness directly at the level of the Cuntz semigroup -- without passing through stable rank one. By this, we mean a proof for $\Cu$-semigroups with additional axioms known to hold for Cuntz semigroups arising from (simple) stably finite separable {$C^*$}-algebras (cf.\ the example in \cite[Chapter 9, Question 8]{APT:MAMS}).

\begin{proposition}\label{prop:cu-regular}
    Let $A$ be a simple separable stably finite {$C^*$}-algebra. The following are equivalent:
    \begin{enumerate}
        \item\label{cu-reg1} the first-factor embedding $A \rightarrow A \otimes \mathcal Z$ induces an isomorphism on the Cuntz semigroup;
        \item\label{cu-reg2} $\Cu(A)\cong \Cu(A\otimes\Z)$;
        \item\label{cu-reg3} $\Cu(A)\cong\Cu(B\otimes\Z)$ for some {$C^*$}-algebra $B$;
        \item\label{cu-regextra}$\Cu(A)\cong \Cu(A)\otimes_{\Cu}\Cu(\Z)$;\footnote{The Cuntz semigroup tensor product used here, and in the next condition, is developed in \cite{APT:MAMS}.}
        \item \label{cu-reg4}$\Cu(A)\cong S\otimes_{\Cu}\Cu(\Z)$ for some abstract Cuntz semigroup $S$;
        \item\label{cu-reg5} $A$ is pure;
        \item\label{cu-reg-new} $A$ has strict comparison and all ranks almost occur;
        \item\label{cu-reg6} $\Cu(A)=V(A)\amalg \mathrm{LAff}_{> 0}(QT(A))$, with the ordered semigroup structure as described in \cite[Definition~5.4]{AraPereraToms}.\footnote{Strictly speaking, this means that the map $\Cu(A)\to V(A)\amalg \mathrm{LAff}_{> 0}(QT(A))$ which is defined for $A$ finite (see footnote \ref{Footnote:Welldefined}) by taking classes in $\Cu(A)$ which are Cuntz equivalent to a projection $p$, to the corresponding class in $V(A)$, and taking classes $\langle a\rangle$ which are not equivalent to projections, to the rank function $\mathrm{Rank}(a)$ from \eqref{DefRank}. The addition and order are as specified just before \cite[Theorem 6.2]{TT:CMB}.}
    \end{enumerate}
\end{proposition}

\begin{proof}
    \ref{cu-reg1}$\implies$\ref{cu-reg2}$\implies$\ref{cu-reg3} are tautologies.  
    \ref{cu-reg3}$\implies$\ref{cu-reg5} is essentially proven by R{\o}rdam in \cite{Rordam:IJM} (see \cite[Proposition~3.7]{Winter:IM12}).  Conditions \ref{cu-reg4}, \ref{cu-regextra} and \ref{cu-reg5} are all equivalent by \cite[Theorem~7.3.11]{APT:MAMS} (and in fact this holds without assuming simplicity of $A$).\footnote{To get \ref{cu-reg4}$\implies$\ref{cu-regextra}, use that $\Cu(\Z)\otimes_{\Cu}\Cu(\Z)\cong \Cu(\Z)$ (\cite[Proposition 7.3.3]{APT:MAMS}).}
    \ref{cu-reg5}$\implies$\ref{cu-reg-new} can be found as (3)$\implies$(1) of \cite[Theorem~2.13]{Lin23}, and then  \ref{cu-reg-new}$\implies$\ref{cu-reg6} is \cite[Theorem 6.2]{TT:CMB} (though that result is stated with an exactness hypothesis so as to work with traces rather than quasitraces).\footnote{As noted above, these sorts of calculations have appeared before, though with more hypotheses (such as stable rank 1, as in \cite[Theorem (or really exercise) 5.27]{AraPereraToms}, or \cite[Theorem 2.6]{BT:IMRN}, building on \cite{BPT:Crelle}).  One of us relearnt about \cite[Theorem 6.2]{TT:CMB} -- which works in the generality we need -- from \cite[Theorem 2.13]{Lin23}. In fact \cite{TT:CMB} goes too far in removing hypotheses: there is an implicit stable finiteness hypothesis missing (which does not affect our use of the theorem). Without stable finiteness, Cuntz equivalent projections need not be Murray--von Neumann equivalent (as opposed to mutually Murray--von Neumann subequivalent), and the proposed map may not be well-defined. \label{Footnote:Welldefined}}

Finally for \ref{cu-reg6}$\implies$\ref{cu-reg1}, observe that the order structure on 
\begin{equation}
    V(A)\amalg \mathrm{LAff}_{> 0}(QT(A))
\end{equation}
is almost unperforated and almost divisible, i.e.\  $A$ is pure (essentially because the order structure on $\mathrm{LAff}_{> 0}(QT(A))$ has these properties).     Therefore, by Lin's Theorem (Theorem~\ref{thm:Lin}), $A$ has stable rank one and hence cancellation (by \cite{Rieffel:PLMS}).   We now have all the pieces to assemble the following diagram of isomorphisms, where the horizontal maps are induced by the first factor embeddings.
\begin{equation} 
\begin{tikzcd}
    V(A)\amalg \mathrm{LAff}_{> 0}(QT(A))\ar[r,"\cong"]&V(A\otimes \mathcal Z)\amalg \mathrm{LAff}_{> 0}(QT(A\otimes \mathcal Z))\\
    \Cu(A)\ar[u,"\cong"]\ar[r]&\Cu(A\otimes\Z)\ar[u,"\cong"']
\end{tikzcd}
\end{equation}
Lemma \ref{CR:Lem} ensures the top horizontal map is an isomorphism from, the left vertical map is an isomorphism by the hypothesis \ref{cu-reg6}, and the right vertical map is an isomorphism by the already estabilished implication \ref{cu-reg1}$\implies$\ref{cu-reg6} applied to $A\otimes\Z$. Accordingly the bottom horizontal map is an isomorphism, verifying \ref{cu-reg1}.
\end{proof}

When $A$ is non-simple, we do not know if pureness implies any of the conditions \ref{cu-reg1}-\ref{cu-reg3} in Proposition~\ref{prop:cu-regular} -- we will come back to this in Section \ref{sec:NonsimpleTW} (see Problem~\ref{q:PureCUReg}). In general, we view the first condition above as the most natural version of $\Z$-stability at the level of the Cuntz semigroup.
\begin{definition}
    We say a separable {$C^*$}-algebra $A$ is \emph{Cuntz semigroup regular} when Condition \ref{cu-reg1} of Proposition~\ref{prop:cu-regular} holds.
\end{definition}


Returning to the concept of pureness, almost-divisibility is not the only way of formulating a divisibility condition on the Cuntz semigroup.   Indeed, the viewpoint at the beginning of the previous section suggests thinking of the condition that all rank functions occur as a kind of divisibility-type condition, and this is borne out by the equivalence \ref{cu-reg5}$\Leftrightarrow$\ref{cu-reg-new} of Proposition~\ref{prop:cu-regular}.  There are further divisibility conditions still, such as in trace (e.g.\ \cite[Definition~3.5]{Winter:IM12} or Dadarlat and Toms' condition (iii) from \cite[Theorem~1.1]{DadarlatToms:JFA})  or higher-dimensional/coloured versions (see \cite[Definition~3.5]{Winter:IM12}).  

It is folklore that in the presence of strict comparison, all these (and other) divisibility conditions coincide and can be used to define pureness (see \cite[Section~5]{CETW:IMRN}).  
Without strict comparison, we would expect some of these divisibility-type conditions to differ, though there are certainly implications between them.\footnote{Both almost-divisibility and Winter's tracial divisibility property will give condition (iii) from \cite[Theorem 1.1]{DadarlatToms:JFA}, and hence all ranks almost occur. Interestingly, when the trace space is a Bauer simplex one only needs almost-divisibility of the unit (or in fact the weaker rank condition of \cite[Theorem 1.1(ii)]{DadarlatToms:JFA}) to obtain the same conclusion.} When preparing this paper, we realised less was known (at least to us) than we thought. In particular, we were unable to find an example of a unital simple separable nuclear non-elementary {$C^*$}-algebra without almost-divisibility for which all ranks are known to occur, though we very much believe such an example should exist. To our knowledge it, is open whether almost-divisibility holds for the Villadsen algebras of first type which fail to have strict comparison (these have stable rank one, and hence all ranks occur by Thiel's theorem) or the Villadsen algebras of second type  from \cite{Villadsen:JAMS} (which have unique trace, so all ranks occur) -- cf.\ Problem~\ref{Q:CuVilToms}. 

It is unclear whether strict comparison gives rise to some form of divisibility, or indeed (and much more speculatively), whether sufficiently strong divisibility conditions give rise to strict comparison. We also do not know to what extent nuclearity may affect the answer to these questions.

\begin{question}\label{ComparisonVsDivisibility}
\begin{enumerate}[(1)]
\item Let $A$ be a unital simple separable (nuclear) non-elementary {$C^*$}-algebra with strict comparison.  Must $A$ be pure? \label{ComparisonVsDivisibility.1}
\item Let $A$ be a unital simple separable (nuclear) {$C^*$}-algebra whose Cuntz semigroup is almost divisible. Must $A$ be pure?\label{ComparisonVsDivisibility.2}
\end{enumerate}
\end{question}

As discussed above, Problem \ref{ComparisonVsDivisibility}(\ref{ComparisonVsDivisibility.1}) has a positive answer when all ranks (almost) occur in $A$. Since ASH algebras have locally finite nuclear dimension, the combination of Thiel’s rank theorem (Theorem \ref{RankThm}) and Winter’s $\Z$-stability theorem (Theorem \ref{Thm:Winter2}) establishes the remaining part of the Toms--Winter conjecture for ASH algebras of stable rank one.  The Toms--Winter conjecture holds in a very similar fashion for unital simple non-elementary AH algebras as all ranks almost occur in these algebras.\footnote{As discussed in footnote~\ref{Footnote24}, the stable rank one hypothesis of \cite[Theorem 5.3]{BPT:Crelle} used in \cite[Corollary 2.2]{Toms:CRMASSRC} is not needed.}

For a simple unital non-elementary ASH algebra $A$, Toms showed that if $A$ has slow dimension growth, then $A$ has strict comparison (\cite{Toms:CMP}) and all ranks almost occur in $A$ (\cite[Theorem 3.4]{Toms:IM}).  Accordingly, slow dimension growth ensures $\Z$-stability (by Theorem \ref{Thm:Winter2}) and so classifiability. Conversely, tensoring an arbitrary ASH system with an appropriate system of dimension drop algebras shows that a simple $\mathcal Z$-stable {$C^*$}-algebra has slow dimension growth.\footnote{Combining range-of-invariant results with classification, one can in fact get bounded dimension growth.}
Altogether, for simple unital non-elementary ASH algebras, $\mathcal Z$-stability, Cuntz semigroup regularity, and slow dimension growth are equivalent.
However, the Toms--Winter conjecture remains open for simple finite ASH algebras without assuming stable rank one, essentially due to Problem~\ref{ASH.Ranks}.

The McDuff-type characterisation of $\Z$-stability can be phrased in terms of the Cuntz semigroup of the central sequence algebra: in the unital case, $A$ is $\Z$-stable if and only if its central sequence algebra $A_\omega\cap A'$ is pure.\footnote{We thank Hannes Thiel for bringing this formulation of $\Z$-stability to our attention. As $A_\omega \cap A'$ is non-simple when $A$ is finite and non-elementary (\cite[Theorem~2.12]{Kirchberg:Abel}), one needs to define pureness for non-simple algebras. An appropriate definition is that $B$ is pure if its Cuntz semigroup is almost divisible and almost unperforated; almost 	unperforation is equivalent to strict comparison both in the simple case and (with the right definition of strict comparison) in the non-simple case as well.  We will discuss pureness further for non-simple {$C^*$}-algebras in Section \ref{sec:NonsimpleTW}.\label{Footnote:DefPure} }
Certainly, if $A_\omega \cap A'$ is pure, then almost-divisibility leads to an order zero map $\phi:M_n \to A_\omega \cap A'$, such that $1_{A_\omega \cap A'} \preceq \phi(1_{M_n})\oplus\phi(e_{11})$.
Strict comparison (as appropriately defined in the non-simple setting -- see \cite[Proposition~3.2]{Rordam:IJM}) then implies that $1_{A_\omega \cap A'} -\phi(1_{M_n}) \preceq \phi(e_{11})$, which gives rise to $\mathcal Z$-stability of $A$.\footnote{In the non-unital case, the same argument can be used with Kirchberg's central sequence algebra $F(A)=(A_\omega \cap A')/A^\perp$, so that $\mathcal Z$-stability of $A$ is equivalent to pureness of this algebra.}
The proof of Winter's Theorem~\ref{Thm:Winter2} for stably finite {$C^*$}-algebras essentially factors through this argument, with both almost-divisibility and strict comparison weakened.
More precisely, for a simple {$C^*$}-algebra $A$ with locally finite nuclear dimension, Winter proves
\begin{enumerate}
\item\label{WintersProof.2} a suitable divisibility condition on $A$ implies a suitable divisibility condition on $A_\omega\cap A'$ (\cite[Section 5]{Winter:IM12}), and 
\item\label{WintersProof.1} strict comparison for $A$ implies a suitable small-to-large comparison for $A_\omega\cap A'$ (\cite[Section 6]{Winter:IM12}).
\end{enumerate}
Step~\ref{WintersProof.1} can be seen as a precursor to Matui and Sato's property (SI), and one can now bypass this step by using Matui and Sato's Theorem~\ref{Thm:SI}, which also removes the need for locally finite nuclear dimension here.
It is natural to ask whether the locally finite nuclear dimension hypothesis is also needed in \ref{WintersProof.2}, which amounts to the following question.

\begin{question}\label{Q7b}
Let $A$ be a unital simple separable nuclear non-elementary stably finite {$C^*$}-algebra which has Winter’s tracial divisibility property: for all $k,n\in\mathbb N$, $\epsilon>0$, and non-zero $a\in M_k(A)_+$, there exists a c.p.c.\ order zero map $\phi\colon M_n\to \overline{aM_k(A)a}$ with $\tau(\phi(1_{M_n}))>\tau(a)-\epsilon$ for all $\tau\in T(A)$. Does $A$ have uniform property $\Gamma$?
\end{question}

We end the section by commenting on how Winter used the ideas in his Theorem~\ref{Thm:Winter2} to prove (i)$\implies$(ii) in Theorem \ref{Thm:Structure}.  This result must involve a `dimension reduction' to pass from an arbitrary finite value of the nuclear dimension to the low-dimensional condition of $\Z$-stability.  In Winter's argument, this is done by weakening pureness to higher dimensional versions -- $(m,m')$-pureness\footnote{The constants $m$ and $m'$ keep track of the appropriate higher-dimensional versions of comparison and divisibility; see \cite[Section 3]{Winter:IM12}.} --  which is entailed by finite nuclear dimension.  Winter in fact proves a stronger version of Theorem \ref{Thm:Winter2} obtaining $\Z$-stability for
unital simple separable non-elementary stably finite {$C^*$}-algebras from $(m,m')$-pureness in the presence of locally finite nuclear dimension. The dimension reduction happens in the midst of Winter's argument at the level of central sequences through a geometric series trick.  In a recent preprint, Antoine, Perera, Robert, and Thiel made this dimension reduction more conceptual by showing that it occurs directly at the Cuntz semigroup level.\footnote{Even more recently, Antoine, Perera, Thiel and Vilalta have extended this result to the non-simple setting in the preprint \cite{APTV24}. Note that in both of these works, the definition of $m'$-almost-divisibility is slightly different than Winter's. See the fourth paragraph of \cite[Section~2.3]{RobertTikuisis:TAMS} for a thorough comparison.}
\begin{theorem}[{\cite[Theorem D]{APRT:JFA}}]\label{PureThm}
Any simple $(m,m')$-pure {$C^*$}-algebra is already pure.
\end{theorem}

Today, one would prove (i)$\implies$(ii) in Theorem \ref{Thm:Structure} by obtaining $(m,m')$-pureness from finite nuclear dimension (\cite[Theorem~3.1]{RobertTikuisis:TAMS}), applying Theorem~\ref{PureThm} to get pureness and hence Cuntz semigroup regularity, and following Step~\ref{WintersProof.2} in Winter's proof of Theorem \ref{Thm:Winter2} (taking advantage of the lack of need for higher comparison and divisibility conditions), and finally end by using Matui and Sato's property (SI) from Theorem \ref{Thm:SI}.

\section{Does real rank zero give rise to regularity?}\label{sec:rr0regularity}

In \cite{Pedersen:JOT,BrownPedersen:JFA}, Larry Brown and Pedersen introduced the notion of real rank zero, a property which implies an abundance of projections.
(In \cite{BrownPedersen:JFA}, the property is the zero-dimensional case of a more general non-commutative dimension theory, from which the name arises.)
The property is enjoyed by many important classes of {$C^*$}-algebras, including von Neumann algebras, AF algebras, Kirchberg algebras (and indeed, all simple purely infinite algebras), and examples such as irrational rotation algebras.  On the other hand, many regular {$C^*$}-algebras such as $\Z$ contain few projections, and consequently cannot have real rank zero.  For a unital simple exact  $\Z$-stable finite {$C^*$}-algebra $A$ (for example, if $A$ is a unital finite classifiable {$C^*$}-algebra) real rank zero is characterised by the image of the pairing map $\rho_A\colon K_0(A)\to\mathrm{Aff}(T(A))$ being dense in $\mathrm{Aff}(T(A))$ by \cite[Theorem~7.2]{Rordam:IJM}.

Real rank zero is often a helpful simplifying assumption -- for example, the Cuntz semigroup of a real rank zero {$C^*$}-algebra is determined by its Murray--von Neumann semigroup (projections in the stabilisation modulo Murray--von Neumann equivalence). On the other hand, while real rank zero gives rise to some amount of regularity, it is unclear exactly how much. This is a topic which Shuang Zhang heavily investigated from the early '90s; see 
\cite{Zhang:CJM}--\cite{Zhang:Ann}.
\nocite{Zhang:CJM,Zhang:PJM2,Zhang:JOT,Zhang:TAMS,Zhang:CM,Zhang:IJM,Zhang:PJM,Zhang:KT,Zhang:Ann}
Major challenges remain; indeed, the following `real rank zero dichotomy problem' demonstrates how little we know about the full implications of real rank zero even for simple {$C^*$}-algebras.

\begin{question}[{R{\o}rdam; \cite[Question 7.6]{Rordam:Acta}}]\label{q:RR0dichotomy}
Is every simple (nuclear) {$C^*$}-algebra with real rank zero  either stably finite or purely infinite?\footnote{R{\o}rdam's \cite[Question 7.6]{Rordam:Acta} is the contrapositive of this question: is there a simple nuclear {$C^*$}-algebra which is real rank zero and contains both a finite and an infinite projection? If so, then it is neither stably finite nor purely infinite. 
If not, then every simple nuclear {$C^*$}-algebra with real rank zero which is not stably finite must have only infinite projections and therefore an infinite projection in every hereditary subalgebra, which is a characterisation of pure infiniteness for simple {$C^*$}-algebras.}
\end{question}

This problem was posed by R\o rdam in the context of his examples of simple nuclear {$C^*$}-algebras from \cite{Rordam:Acta} containing both finite and infinite projections (which he later showed do not have real rank zero in \cite{Rordam05}).
The question is open even without the nuclearity hypothesis.
In \cite{OrtegaPereraRordam:TAMS} (and independently by Zhang in unpublished work), it is shown that combining real rank zero with a very mild regularity property -- the corona factorisation property -- is sufficient to conclude the stably finite/purely infinite dichotomy.

The following asks whether real rank zero is a regularity property in an even stronger sense.

\begin{question}\label{q:RR0Z}
    Is every simple separable nuclear non-elementary {$C^*$}-algebra with real rank zero automatically $\mathcal Z$-stable?
\end{question}

A positive answer to this question would also solve Problem~\ref{q:RR0dichotomy} (in the affirmative) using Kirchberg's dichotomy. While serious attempts have been made, it remains unclear whether any variant of Villadsen's constructions might yield an example of a non-$\mathcal Z$-stable real rank zero {$C^*$}-algebra.

In \cite[Theorem~5.8]{PereraRordam:JFA}, Perera and R\o rdam show that a non-type $I$ separable {$C^*$}-algebra $A$ with real rank zero is weakly divisible (defined in \cite[Definition~5.1]{PereraRordam:JFA}). They pose what turns out (by \cite[Theorem 2.3]{TomsWinter:CJM}) to be a stronger version of Problem \ref{q:RR0Z}: is every unital simple separable nuclear non-elementary {$C^*$}-algebra $A$ of real rank zero approximately divisible?\footnote{Approximate divisibility says that for any $n$, there is a unital embedding $M_n\oplus M_{n+1}\rightarrow A_\omega\cap A'$.  As approximate divisibility is really a central divisibility condition, at least one of the authors would like to go back in time and change the terminology (especially as it is easy to get it confused with almost-divisibility, which is a non-central condition). In particular, it is important to realise that approximate divisibility should not be thought of in the same spirit as the divisibility-type conditions discussed in Section~\ref{Sec:CuReg}, though it implies almost-divisibility.} 

Somewhat related to these problems is the question of whether real rank zero is a property of the Cuntz semigroup of a {$C^*$}-algebra. 
If $A$ has real rank zero, then every hereditary subalgebra of $A$ (and of $A \otimes \mathcal K$) has an approximate unit consisting of projections (this, in fact, characterises real rank zero). 
It follows that every element of $\mathrm{Cu}(A)$ is a supremum of an increasing sequence of compact elements (those $x \in \mathrm{Cu}(A)$ with $x \ll x$), as was noticed early on by Perera in \cite[Theorem~2.8]{Perera:IJM}.  Such Cuntz semigroups are called \emph{algebraic} in \cite[Definition~5.5.1]{APT:MAMS}, and this property seems the likely candidate (if any) for detecting real rank zero at the Cuntz semigroup level.

\begin{question}\label{q:cu-algebraic}
    If $A$ is a simple stably finite {$C^*$}-algebra and $\mathrm{Cu}(A)$ is algebraic, must $A$ have real rank zero?
\end{question}

In the non-simple, purely infinite case, $\mathrm{Cu}(A)$ can be algebraic without $A$ having real rank zero, as the following example due to Hannes Thiel shows.
Take an extension $0 \to I \to E \to \mathcal O_\infty \to 0$ such that $I$ is a stable Kirchberg algebra and the exponential map $K_0(\mathcal O_\infty) \to K_1(I)$ is an isomorphism (so that $K_1(I)\cong \mathbb Z$).
Then since $E$ is $\mathcal O_\infty$-stable, its Cuntz semigroup contains three elements (one for each ideal), and is easily seen to be algebraic.
However, it cannot have real rank zero, because this would force the exponential map to vanish by \cite[Theorem~3.14]{BrownPedersen:JFA}.

When $A$ has stable rank one, Problem \ref{q:cu-algebraic} has a positive answer by the second paragraph of \cite[Corollary~5]{CEI:Crelle}.

We ask further questions, related to $K$-theoretic regularity of real rank zero {$C^*$}-algebras, and connected to the dichotomy question (Problem \ref{q:RR0dichotomy}) in Section~\ref{sec:K1injectivity} as Problem~\ref{q:K1SurjectRR0}.
    
\section{Classifiability of {$C^*$}-algebras associated to commutative dynamics}\label{sec:crossed-product}

The power of Connes’ structural theorem is seen through its ease of application in examples. Given a free ergodic measure-preserving action of a countable discrete amenable group $G$ on a non-atomic standard probability space $(X,\mu)$, it is straightforward to verify that the crossed product II$_1$ factor $L^\infty(X,\mu)\rtimes G$ is injective, so Connes’ theorem shows that it is the hyperfinite II$_1$ factor.\footnote{In this framework, it was later possible to see hyperfiniteness directly through Ornstein and Weiss’ Rokhlin lemma (\cite{OrnsteinWeiss:BAMS} and \cite[Theorem~6.1]{Dye:AJM}).}
A lot of effort has gone into the analogous question of which {$C^*$}-algebras arising from group actions are classifiable.
Although we focus our attention on group actions, the questions posed here can be generalised to \'etale groupoids (and even twisted groupoid {$C^*$}-algebras).

Most of the classification hypotheses are well-understood for reduced crossed products $C(X)\rtimes G$: given an action of a countable discrete group $G$ on a compact metrisable space $X$, the reduced crossed product {$C^*$}-algebra $C(X)\rtimes G$ is automatically separable and unital; it is nuclear precisely when the action is topologically amenable (\cite[Corollary 6.2.14]{ADR:Book}).\footnote{Generalising the notion of amenability for groups, amenable actions of groups on {$C^*$}-algebras were defined by Anantharaman-Delaroche in \cite{AD:MA}; in the commutative case, this provides the notion of an amenable action on a locally compact Hausdorff space.
See \cite{OzawaSuzuki:Selecta} for characterizations.
Amenable groups always act amenably.  For an amenable action, the reduced and full crossed products agree (\cite[Proposition~4.8]{AD:MA}), although the converse is open (see \cite{Suzuki:JNCG,BEW:ContempMath}).

One should be warned that the terminology ``amenable action'' is not consistent across the literature (for actions on non-commutative {$C^*$}-algebras); see \cite[Remark~2.2]{BEW:ContempMath}.}
In this case, the full and reduced crossed products agree and satisfy the UCT (Theorem~\ref{thm:Tu}, by Tu), and the crossed product is simple precisely when the action is topologically free and minimal (\cite[Corollary to Theorem~2]{AS:PEMS}).  Moreover, a dichotomy between being stably finite or being traceless is given by amenability and non-amenability of $G$ (see the folklore lemma \cite[Lemma 2.2]{GGKN:Crelle}).  So the fundamental challenge is to determine when topologically free minimal actions give rise to $\Z$-stable crossed products.

The precise analysis and, in particular, the expected role of the space $X$ differs on the two sides of the finite/infinite dichotomy. In the stably finite case, Giol and Kerr provide an example of a free minimal $\mathbb Z$-action on an infinite-dimensional compact metrisable space $X$ such that $C(X)\rtimes \mathbb Z$ is not $\mathcal Z$-stable (\cite{GK:Crelle}), whereas early positive results showed that when $X$ is finite-dimensional (and has infinitely many points), all minimal $\mathbb Z$-actions (which are automatically free) give $\Z$-stable crossed products (\cite{TW:GAFA}). 
 In contrast, when $G$ is non-amenable it seems likely that the space $X$ does not matter: there are no known examples of topologically free minimal amenable actions of non-amenable groups whose crossed product is not a Kirchberg algebra, and recently, many examples of groups (including all non-amenable free groups) have been found for which all such actions give classifiable crossed products (\cite{GGKN:Crelle}).  These and subsequent developments have given particular prominence to the following problem.

\begin{question}\label{q:ZCrossedProd}
\begin{enumerate}[(1)]
    \item    Let $G$ be a countably infinite discrete amenable group, let $X$ be a compact metrisable space with finite covering dimension, and let $\alpha\colon G \curvearrowright X$ be a free minimal  action. Does it follow that $C(X)\rtimes G$ is $\mathcal Z$-stable?\label{q:ZCrossedProd.1}
    \item   Let $G$ be a countable discrete exact but non-amenable group, let $X$ be a compact metrisable space, and let $\alpha\colon G \curvearrowright X$ be a free minimal amenable action. Does it follow that $C(X)\rtimes G$ is $\mathcal Z$-stable?\label{q:ZCrossedProd.2}
\end{enumerate}
\end{question}


These questions have been answered for many classes of groups, and there have been two main approaches. The first is to proceed by directly bounding the nuclear dimension of the crossed product via dimensional conditions on the action -- e.g.\ Rokhlin dimension (\cite{HWZ:CMP,Szabo:PLMS,SWZ:ETDS}) or Hirshberg and Wu's recent long-thin covering dimension (\cite{HW:arXiv}); $\Z$-stability then follows from the structure theorem (Theorem~\ref{Thm:Structure}). A culminating result is that Problem~\ref{q:ZCrossedProd}(\ref{q:ZCrossedProd.1}) has a positive answer provided that $G$ has polynomial growth\footnote{\label{ft:Gromov}The result in \cite{SWZ:ETDS} is stated for finitely generated virtually nilpotent groups.  By Gromov's celebrated result in \cite{Gromov:IHES}, a finitely generated group has polynomial growth if and only if it is virtually nilpotent.} (\cite[Theorem~8.8]{SWZ:ETDS}). The above problems can also be asked for actions which are only topologically free rather than free. The long-thin covering dimension is well-suited to non-free actions (which finite Rokhlin dimension excludes); one of many results in Hirshberg and Wu's preprint \cite{HW:arXiv} is that all actions of groups of polynomial growth on finite-dimensional spaces have finite long-thin covering dimension, which implies finite nuclear dimension for the crossed product. The structure theorem (Theorem \ref{Thm:Structure}) then converts this to $\Z$-stability whenever the action is additionally topologically free and minimal. 
However, it seems likely that a general approach to all actions of a given group via dimension-bounding will only work under polynomial growth.\footnote{We note, on the other hand, that the long-thin covering dimension method from \cite{HW:arXiv} does give results for certain actions of groups outside this class, including some examples of non-free actions which cannot be handled by almost finiteness (described below).}

The second approach aims to go to $\Z$-stability (or, in the non-amenable case, pure infiniteness) more directly. State-of-the-art results in this direction include positive answers to Problem~\ref{q:ZCrossedProd}(\ref{q:ZCrossedProd.1}) for all groups locally of subexponential growth (\cite[Theorem~8.1]{KerrSzabo:CMP}, using \cite[Theorem~6.33]{DownarowiczZhang:MAMS}), all elementary amenable groups (in the preprint \cite[Corollary~B]{KerrNaryshkin}), and many classes of non-amenable groups all of whose free minimal amenable actions behave as predicted in  Problem~\ref{q:ZCrossedProd}(\ref{q:ZCrossedProd.2}), (\cite{GGKN:Crelle}); these include non-amenable hyperbolic and Baumslag--Solitar groups).  Note though that the results from \cite{KerrSzabo:CMP,KerrNaryshkin} do not apply to actions which are only topologically free (without requiring minimality, as in \cite{GGKN:Crelle}).

An analogue of strict comparison is pervasive in the second approach to both parts of Problem~\ref{q:ZCrossedProd}. For open sets $U,V \subseteq X$, write $U\precsim V$ if every compact subset of $U$ can be covered by finitely many open sets such that these open sets can be translated (via the group action) to a family of disjoint subsets of $V$.
The action $G \curvearrowright X$ has \emph{dynamical comparison} if 
for non-empty open sets $U,V\subseteq X$, if $\mu(U)<\mu(V)$ for every invariant probability measure $\mu$, then $U\precsim V$.
Whereas strict comparison asks that a gap in the rank functions associated to two positive elements always has a good {$C^*$}-algebraic explanation (Cuntz subequivalence), dynamical comparison asks that a gap in the invariant measures of two open subsets has a good dynamical explanation.

\begin{question}\label{q:DynComparison}
Do all free minimal amenable actions of countable discrete groups on compact metrisable spaces have dynamical comparison?
\end{question}

As pointed out to us by Julian Kranz, no examples of actions without dynamical comparison are known even in the more general setting of discrete groups acting on compact spaces (with no further assumptions).
When there are no invariant measures, Ma has shown that topologically free minimal actions with dynamical comparison (in this case, really dynamical pure infiniteness: $X \precsim V$ for any non-empty open set $V$) give simple purely infinite crossed products (\cite[Theorem~1.1]{Ma:TAMS}).\footnote{The challenge in Ma's result is to go from the condition that $1 \precsim f$ in $C(X)\rtimes G$ for all non-zero $f \in C(X)_+$ to $1 \precsim a$ for all non-zero $a \in (C(X)\rtimes G)_+$.}  Accordingly, a positive solution to Problem~\ref{q:DynComparison} for non-amenable groups would therefore give a positive solution to Problem~\ref{q:ZCrossedProd}(\ref{q:ZCrossedProd.2}). This is the approach taken in \cite[Theorem B]{GGKN:Crelle}, which shows that minimal actions of non-amenable groups with a certain paradoxical tower condition (and products of these with arbitrary groups) automatically have dynamical comparison.

In the stably finite setting (i.e.\ for actions of amenable groups), Kerr's notion of \emph{almost finiteness} (\cite[Definition~8.2]{Kerr:JEMS}) has proved extremely instrumental. Generalising a notion due to Matui for actions on the Cantor set, almost finiteness  is a property for actions of amenable groups that looks a bit like a dynamical version of tracial $\Z$-stability, provided you squint hard enough.  Kerr used an intricate argument with Ornstein--Weiss quasitilings (\cite{OrnsteinWeiss:BAMS}) to achieve approximate centrality from F{\o}lner conditions, establishing the following (via tracial $\Z$-stability and Matui and Sato's Theorem \ref{Thm:SI}).

\begin{theorem}[{Kerr; \cite[Theorem~12.4]{Kerr:JEMS}}]\label{Thm:Kerr}
    Let $G$ be a countable discrete amenable group, $X$ a compact metrisable space, and  $G \curvearrowright X$ a free\footnote{This theorem is generalised to essentially free minimal actions in the recent preprint \cite{GGGKN24}. For virtually cyclic groups, this was obtained earlier in the \cite{LiMa23}.} minimal action that is almost finite.
    Then $C(X)\rtimes G$ is $\mathcal Z$-stable.
\end{theorem}

In \cite[Theorem~B]{KerrSzabo:CMP}, Kerr and Szab\'o proved that for an amenable group $G$, if every free action of $G$ on a zero-dimensional space is almost finite, then the same is true for every free action of $G$ on any finite-dimensional space.
Tantalisingly, Conley, Jackson, Kerr, Marks, Seward, and Tucker-Drob prove that for any amenable group $G$,  generic free minimal actions of $G$ on the Cantor set are almost finite (\cite[Theorem~4.2]{CJKMST:MA}).
Although it seems impossible to merge this with the aforementioned result of Kerr and Szab\'o to establish a genericity result on arbitrary finite-dimensional spaces, we nonetheless regard these developments as strong evidence that Problem~\ref{q:ZCrossedProd}(\ref{q:ZCrossedProd.1}) has a positive answer (which we hope is also the case for topologically free actions; particularly in the light of \cite{HW:arXiv}).

Elliott and Niu established the first general result for actions of amenable groups on infinite-dimensional spaces in \cite{ElliottNiu:Duke}, showing that all free minimal $\mathbb Z$-actions with mean dimension zero give rise to $\Z$-stable crossed products; this was generalised by Niu to $\mathbb Z^d$-actions in \cite{Niu:CJM}.
These results suggest mean dimension zero or the small boundary property (which are equivalent for free minimal $\mathbb Z^d$-actions by \cite[Theorem~1.3 and Corollary~5.4]{GLT:GAFA} and conjecturally equivalent in general) as a dynamical condition to ensure $\mathcal Z$-stability beyond the case of finite-dimensional base space.\footnote{Note that in the finite-dimensional base space case, the small boundary property and mean dimension zero are easily verified.}
Kerr and Szab\'o linked this idea back to dynamical comparison with the following result.

\begin{theorem}[{\cite[Theorem~A]{KerrSzabo:CMP}}]\label{Thm:KerrSzabo}
    Let $G$ be a countably infinite amenable group, $X$ a compact metrisable space, and  $G \curvearrowright X$ a free action.
    The following are equivalent.
    \begin{enumerate}
        \item $G\curvearrowright X$ is almost finite.
        \item $G\curvearrowright X$ has dynamical comparison and the small boundary property.
    \end{enumerate}
\end{theorem}

Hence a positive answer to Problem~\ref{q:DynComparison} would confirm that all free minimal actions of amenable groups with the small boundary property give rise to $\Z$-stable crossed products, hence giving a positive answer to both parts of Problem~\ref{q:ZCrossedProd}(\ref{q:ZCrossedProd.1}). Naryshkin provided a positive answer to Problem~\ref{q:DynComparison} for groups with polynomial growth (\cite{Naryshkin21}), which, combined with Theorem \ref{Thm:KerrSzabo}, generalises Niu's result from $\mathbb Z^d$ to groups with polynomial growth groups (albeit using the small boundary property instead of mean dimension zero).

So far, the discussion has been around establishing {$C^*$}-regularity from dynamical regularity, but another intriguing challenge concerns going in the other direction.

\begin{challenge}\label{q:dynamicalcharacteriseZstable}
    Characterise, in terms of dynamical properties of $G\curvearrowright X$, when the crossed product $C(X)\rtimes G$ is $\mathcal Z$-stable for free minimal actions of countable discrete amenable groups.
\end{challenge}

Optimistically, one might expect the dynamical condition to be almost finiteness, mean dimension zero, and/or the small boundary property, although even for $\mathbb Z$-actions this is open.

The small boundary property for a free minimal action $G\curvearrowright X$ of an amenable group $G$ now has a {$C^*$}-algebraic description in terms of the inclusion of $C(X)$ into $C(X)\rtimes G$: the small boundary is equivalent to a relative version of uniform property $\Gamma$ for $(C(X)\subseteq C(X)\rtimes G)$ as described at the end of Section \ref{sec:TW} (see \cite{KerrSzabo:CMP} for one direction and the recent \cite{KLTV} or preprint \cite{ElliottNiu24} for the other).
As a consequence, almost finiteness for a free minimal action is equivalent to a relative version of tracial $\Z$-stability for the inclusion $(C(X)\subseteq C(X)\rtimes G)$, a clear strengthening of (tracial) $\mathcal Z$-stability of the crossed product. 

The {$C^*$}-algebraic characterisation of the small boundary property allows Theorem \ref{Thm:Kerr} to be viewed as a dynamical version of the characterisation of $\mathcal Z$-stability (in the nuclear setting) given by Theorem~\ref{Thm:ZStableEquivGammaSC}.

The following table summarises these regularity properties in (simple separable nuclear non-elementary) {$C^*$}-algebras and topological dynamics (of free minimal actions of countable discrete amenable groups on compact metrisable spaces).
Interestingly, the status of the properties is reversed in the two settings.

    \begin{tabular}{c|c|c}
         & {$C^*$}-algebras & Topological dynamics \\ \hline
        & \textit{Strict comparison} & \textit{Dynamical comparison} \\
         Comparison & Known to fail in examples & Always holds? \\
         & (Villadsen, ...) & (Problem~\ref{q:DynComparison}) \\ \hline
         & \textit{Uniform property $\Gamma$} & \textit{Relative $\Gamma$ = SBP} \\
         Property $\Gamma$ & Always holds? & Known to fail in examples \\
         & (Problem~\ref{q:Gamma}) & (Lindenstrauss--Weiss \cite{LindenstraussWeiss:IJM})\footnote{In \cite[Proposition~3.5]{LindenstraussWeiss:IJM}, Lindenstrauss and Weiss give an example of a minimal $\mathbb Z$-action with non-zero mean dimension, and then in \cite[Theorem~5.4]{LindenstraussWeiss:IJM} they show that this implies it does not have the small boundary property.}
    \end{tabular}
    \label{table}
\medskip

For the Giol--Kerr example of a free minimal $\mathbb Z$-action with a non-$\mathcal Z$-stable crossed product, the action has dynamical comparison (by Naryshkin's result -- \cite[Theorem~A]{Naryshkin21}), but the inclusion $(C(X)\subseteq C(X)\rtimes \mathbb Z)$ does not have uniform property $\Gamma$ (because the action does not have mean dimension zero; see \cite[Section~3]{GK:Crelle}).
The crossed product does not have strict comparison (\cite[Theorem~3.1]{GK:Crelle}), and it is not known whether it has uniform property $\Gamma$. It is also unclear to the authors whether this is actually the same question as Problem~\ref{q:VilladsenGamma}.

\begin{question}\label{q:TCiso}
    Let $A$ be a non-$\mathcal Z$-stable crossed product from \cite{GK:Crelle}. Is the tracial completion of $A$ the same as the tracial completion of a Villadsen algebra of the first type?
\end{question}

Going beyond classifiability, the Phillips--Toms conjecture predicts a deep connection between the mean dimension of a free minimal action and the radius of comparison of the crossed product.

\begin{question}[Phillips--Toms; cf.\ the introduction to \cite{Niu:CJM}]\label{q:PhillipsToms}
    Let $\alpha\colon \mathbb Z\curvearrowright X$ be a minimal action.  Is the radius of comparison of $C(X)\rtimes_\alpha\mathbb Z$ equal to one half of the mean dimension of $\alpha$?
\end{question}

Once it is known that a crossed product of the form $C(X)\rtimes G$ is classifiable, it becomes very relevant to actually classify it -- i.e.\ to compute its invariant $KT_u(C(X)\rtimes G)$.

\begin{question}\label{q:ComputeKTheory}
    Compute $KT_u(C(X)\rtimes G)$ for those actions $G\curvearrowright X$ which give classifiable {$C^*$}-algebras.
\end{question}

In the (essentially) free case, the trace simplex identifies with the set of invariant probability measures (\cite[Corollary 1.12]{Ursu:Adv}, using amenability to work with the full crossed product); for actions of free groups, the Pimsner--Voiculescu exact sequence helps to compute the $K$-theory.
But even in the cases when these two tools compute the $K$-theory and traces, determining the pairing requires more work.
For this, there seem to be very limited general methods -- the only one known by the authors is due to Pimsner for crossed products of {$C^*$}-algebras by free groups, and in particular for crossed products by $\mathbb Z$ (\cite{Pimsner:OACTET}).\footnote{Pimsner does not explicitly describe the pairing for $A\rtimes \mathbb Z$, but the methods do give a calculation of the pairing, which is written explicitly as $\Lambda_\tau$ in \cite[Section 10.10.1]{Blackadar-kbook}. This will be further fleshed out in forthcoming work of Neagu and the first-named author.}

We found relatively few complete calculations of the invariant of the crossed products in Problem \ref{q:ComputeKTheory} in the literature. For example, the case of $\mathbb Z$-actions of Cantor minimal systems was done by Boyle and Handelman in \cite[Theorem~5.2]{BoyleHandelman:IJM} (see also \cite{GPS}), and an irrational rotation on the circle is a classical result of Rieffel along with Pimsner and Voiculescu (\cite{Rieffel:PJM,PimsnerVoiculescu:JOT}). Note the importance of the pairing here: irrational rotation algebras all have the same $K$-theory and traces; it is the pairing that classifies them.
The Pimsner--Voiculescu exact sequence is also used to compute the $K$-theory of certain actions of free groups on the Cantor set in \cite{Suzuki:Crelle}.
More recently, triangulated category  methods of Meyer and Nest growing out of work on the Baum--Connes conjecture have been used by Proietti and Yamashita to derive a spectral sequence converging to the $K$-theory of certain crossed product (and groupoid) {$C^*$}-algebras (\cite[Theorem~A]{ProiettiYamashita:ETDS}); more generally, work around Matui's HK conjecture has prompted new $K$-theory calculations for groupoid {$C^*$}-algebras (for example, \cite{Scarparo:ETDS,Deeley:ETDS,FKPS:MJM}).
We expect the Baum--Connes conjecture with coefficients (which is a theorem in the amenable case -- see \cite{HigsonKasparov:IM, Tu99})  to continue to feature in further $K$-theory computations.
Calculations of the invariant of crossed products by boundary actions of certain hyperbolic groups can be found in \cite{LacaSpielberg:Crelle,MeslundSengun:JNCG,GeffenKranz:MJM} (the invariant is just $K$-theory in these cases) leading to surprising isomorphisms that do not come from the underlying groupoids.

Needless to say, we recognise that Problem \ref{q:ComputeKTheory} is vast in scope.
We expect progress to be made on specific classes of classifiable crossed products rather than a sweeping result for all at once.

\section{Classifiability of {$C^*$}-algebras associated to non-commutative dynamics}\label{Sect:noncomdynamics}

We now turn to non-commutative dynamics -- groups acting on {$C^*$}-algebras -- and their resulting crossed products. Given an action of a group $G$ on a simple separable nuclear (and classifiable) {$C^*$}-algebra, when is the resulting crossed product classifiable?  The corresponding question in the von Neumann setting is answered as a straightforward consequence of Connes' theorem for the same reason as described in the previous section: crossed products of actions of amenable groups preserve injectivity.  Following Connes' theorem, there was significant interest in the structure and classification of group actions on factors (\cite{Connes75, Connes77, Jones79, Jones80, Ocneanu}).  A landmark result of Ocneanu (\cite{Ocneanu}) shows that for every countable discrete amenable group $G$ there is, up to cocycle conjugacy\footnote{Two actions $\alpha\colon G\curvearrowright A$ and $\beta\colon G\curvearrowright B$ are \emph{cocycle conjugate} if there is an isomorphism $\sigma\colon A\to B$ and a function $w \colon G \rightarrow U(B)$ (or $U(\mathcal M(B))$ when $B$ is non-unital) such that
    $\sigma \alpha_g \sigma^{-1} = \mathrm{Ad}(w_g) \beta_g$ and $w_{gh}=w_g\beta_g(w_h)$ for all $g,h\in G$.} -- the appropriate notion of equivalence for group actions -- a unique outer (cocycle) action of $G$ on the hyperfinite II$_1$ factor $\mathcal R$. One of many consequences of Ocneanu's theorem is that any outer action $\alpha\colon G\curvearrowright \mathcal R$ is cocycle conjugate to $\alpha\otimes 1_{\mathcal R}\colon G\curvearrowright\mathcal R\otimes\mathcal R\cong \mathcal R$ -- this is equivariant $\mathcal R$-stability of the action.\footnote{In fact, this is a key step in Ocneanu's proof (\cite[Theorem~8.5]{Ocneanu}) and holds more generally; the theorem is valid for centrally free actions of amenable groups on separably acting McDuff factors. Szab\'o and Wouters returned to this theme recently, removing all assumptions save separability of the predual and amenability of the action: equivariant $\mathcal R$-stability holds automatically for all amenable actions of discrete groups on separably acting McDuff von Neumann algebras (\cite[Theorem A]{SzaboWouters:JIMJ}).  In particular, factoriality is not required.}

This section is devoted to classifiability of crossed products $A\rtimes G$, for simple $A$, with emphasis on the case where $A$ is classifiable.  This question has been around for a long time, with early work focusing on $\mathbb Z$-actions with the Rokhlin property on particular algebras such as UHF algebras  (see \cite[Theorem 1.3]{Kishimoto:Crelle}, for example).  In Section \ref{sec:classaction} we will turn to classifying the actions themselves: that is, the search for {$C^*$}-analogues of Ocneanu's theorem.

It is known that the (full or reduced) crossed product $A \rtimes G$ is nuclear if and only if $A$ is nuclear and the action $G \curvearrowright A$ is amenable (\cite[Th\'eor\`eme~4.5]{AD:MA}).
Sufficient conditions are known for simplicity: for instance, Kishimoto showed in \cite{Kishimoto:CMP} that an outer action of a discrete group on a simple {$C^*$}-algebra always produces a simple reduced crossed product.\footnote{Broader results on the ideal structure of crossed products of non-commutative {$C^*$}-algebras can be found in \cite{KennedySchafhauser} and the  \cite{GeffenUrsu}.} So the heart of the  classifiability question amounts to deciding when $A \rtimes G$ is $\mathcal Z$-stable and satisfies the UCT.  

The UCT problem is rather subtle.  While Tu's theorem gives the UCT for crossed products arising from amenable actions on commutative {$C^*$}-algebras, this is not known when the base algebra is a non-commutative nuclear {$C^*$}-algebra  satisfying the UCT (see the remarks following Problem~\ref{q:UCT}, for example).  
In the positive direction, the Pimsner--Voiculescu sequence can be used to show crossed products by free groups (including $\mathbb Z$)  preserve the UCT.  
More generally, using Higson and Kasparov's formidable result on the Baum--Connes conjecture (\cite{HigsonKasparov:IM}), Meyer and Nest  have shown that if $G$ is a torsion-free discrete amenable group acting on a separable {$C^*$}-algebra $A$ satisfying the UCT, then the crossed product $A \rtimes G$ satisfies the UCT (\cite[Corollary~9.4]{MeyerNest:Top}).

We now turn to $\Z$-stability; the following general problem is open even when $G$ is $\mathbb Z$. 

\begin{question}\label{q:z-stable-product}
Let $\alpha\colon G\curvearrowright A$ be an outer\footnote{It is not clear if outerness will be necessary to obtain $\mathcal Z$-stability -- see Problem~\ref{q:z-stable-action} below. It is a natural restriction in Problem~\ref{q:z-stable-product} as it provides simplicity.}
amenable action of a countable discrete group on a unital simple separable nuclear $\Z$-stable stably finite {$C^*$}-algebra. When is the crossed product $A\rtimes G$ $\Z$-stable?
\end{question}

In the case of a Kirchberg algebra $A$, Problem \ref{q:z-stable-product} has a positive answer. Indeed,  any crossed product of a simple purely infinite algebra by an outer action is again simple and purely infinite by \cite[Lemma~10]{KK:OAA} (see also \cite[Lemma~6.3]{Suzuki:IMRN}); hence when $A$ is also nuclear and the action is amenable, the crossed product is nuclear and thereby $\mathcal O_\infty$-stable (by Kirchberg's Theorem \ref{Thm:Oinftystable}). In the case when $A$ is unital and stably finite and the group is amenable, at least one trace will be fixed by the action, and so the crossed product by such an outer action will also be stably finite. Just as in the foundational work focusing on explicit examples, Rokhlin-type conditions have played substantial roles in the abstract setting: all Rokhlin actions of $\mathbb Z$ give $\Z$-stable crossed products (\cite{HirshbergWinter:PJM}), and more generally, the same holds for actions of groups of polynomial growth with finite Rokhlin dimension (\cite{SWZ:ETDS}).\footnote{When the action is not assumed outer, one uses \cite[Theorem F]{SWZ:ETDS}, and one needs the strong form of Rokhlin dimension with commuting towers.  For outer actions, one can use \cite[Theorems~B and~C]{SWZ:ETDS} together with the structure theorem (Theorem~\ref{Thm:Structure}) to move between $\Z$-stability and finite nuclear dimension for simple separable nuclear non-elementary {$C^*$}-algebras.}  In \cite{MatuiSato:AJM}, Matui and Sato develop an action version of the `von Neumann lifting strategy' discussed in Section~\ref{sec:TW}, aiming to exploit properties of actions of amenable groups on the hyperfinite II$_1$ factor, to obtain weak positive element versions of the Rokhlin property at the von Neumann level (see \cite[Theorem 3.6]{MatuiSato:AJM}, for example) from strong outerness.\footnote{Strong outerness is the condition that for every non-trivial $g\in G$, and every $\alpha_g$-invariant trace $\tau$, the automorphism induced on $\pi_{\tau}(A)''$ is outer. 
 } This, like some of the results in Section \ref{sec:TW}, depends on trace space conditions. Matui and Sato then combine this `weak Rokhlin property' with their property (SI) -- satisfied by all unital simple separable nuclear $\Z$-stable {$C^*$}-algebras -- to obtain $\Z$-stability of the crossed product (\cite[Theorem~4.9]{MatuiSato:AJM}). Putting things together, Matui and Sato answer Problem \ref{q:z-stable-product} for strongly outer actions of elementary amenable groups on unital simple separable nuclear $\Z$-stable {$C^*$}-algebras with finitely many extremal traces (\cite[Corollary 4.11 and Remark 4.12]{MatuiSato:AJM}).\footnote{Note that throughout Sections 4.1 and 4.2 of \cite{MatuiSato:AJM}, there is a standing hypothesis omitted from the statements of the main theorems that $A$ is a unital simple separable nuclear non-elementary stably finite {$C^*$}-algebra with finitely many extremal traces.  These results do not handle arbitrary trace spaces.}

This last result and many subsequent developments all pass through \emph{equivariant $\mathcal Z$-stability}.\footnote{An action 
$\alpha \colon G \curvearrowright A$ is called \emph{equivariantly $\mathcal Z$-stable} if it is cocycle conjugate to the action $\alpha \otimes 1_{\mathcal Z} \colon G \curvearrowright A \otimes \mathcal Z$.  As cocycle conjugate actions induce isomorphic crossed products, it follows that an equivariantly $\mathcal Z$-stable action induces a $\mathcal Z$-stable crossed product.} 
Establishing equivariant $\mathcal Z$-stability of actions is also relevant to classifying actions (discussed more in Section~\ref{sec:classaction}) as an action $\alpha$ on a $\mathcal Z$-stable {$C^*$}-algebra is indistinguishable on many invariants from its $\mathcal Z$-stabilization $\alpha \otimes 1_{\mathcal Z}$.
Hence we focus on the following conjecture of Szab\'o, which is analogous to the automatic equivariant McDuffness of actions of amenable groups on amenable factors: Ocneanu's original result from \cite{Ocneanu}, and its significant generalisation in \cite{SzaboWouters:JIMJ}, where no outerness condition is present. Correspondingly, note that  no outerness conditions are imposed on the action in Problem \ref{q:z-stable-action}.

\begin{question}[{\cite[Conjecture~A]{Szabo:AnalPDE}}]\label{q:z-stable-action}
    If $G$ is a countable discrete amenable group and $A$ is a unital  simple separable nuclear $\mathcal Z$-stable {$C^*$}-algebra, is every action of $G$ on $A$ equivariantly $\mathcal Z$-stable?
\end{question}

In the traceless setting, Problem~\ref{q:z-stable-action} has a definitive answer for outer actions: the equivariant $\mathcal O_\infty$-stability theorem of Szab\'o in \cite[Theorem~3.4]{Szabo:CMP} implies that all (cocycle) actions of countable discrete amenable groups on Kirchberg algebras are equivariantly $\mathcal O_\infty$-stable (and hence equivariantly $\mathcal Z$-stable).\footnote{In \cite[Corollary~3.6]{Szabo:MJM}, Szab\'o extends this result to amenable actions of non-amenable groups on Kirchberg algebras, and so it is natural ask whether Problem \ref{q:z-stable-action} extends to amenable actions of non-amenable groups on simple separable nuclear $\Z$-stable {$C^*$}-algebras (this is also asked in the introduction \cite{Szabo:MJM}).  We will come back to these actions below.}  In the finite setting, a generic action of any countable discrete group on a $\Z$-stable {$C^*$}-algebra is equivariantly $\Z$-stable (\cite[Corollary 11.6 and Remark 11.15]{SWZ:ETDS}) and so has $\Z$-stable crossed product.

Substantial progress has been made on Problem~\ref{q:z-stable-action} (some time before it appeared explicitly in \cite{Szabo:AnalPDE}). In the unique trace setting, Sato showed how to lift Ocneanu's equivariant $\mathcal R$-stability back from the von Neumann level to obtain equivariant $\Z$-stability through equivariant versions of property (SI).\footnote{Both equivariant $\mathcal R$-stability and equivariant $\Z$-stability have McDuff type characterisations in terms of embeddings into the fixed point algebra of the central sequence algebra, and it is through these and an `equivariant central surjectivity' result corresponding to \eqref{centralsurjectivity} that Sato's result is obtained in parallel to the lifting strategy described in Section \ref{sec:TW}.} This solves Problems~\ref{q:z-stable-product} and~\ref{q:z-stable-action} when the {$C^*$}-algebra involved has a unique trace (\cite{Sato:ASPM}). This holds more broadly, but just as in Section \ref{sec:TW}, conditions on the traces and often the induced action are currently needed. Sato's result holds when the extreme traces are compact and of finite dimension and the action is trivial on traces. The condition on the action was further relaxed in \cite{GHV:JMPA} to allow for uniformly bounded orbits of $G$ on the compact finite-dimensional $\partial_eT(A)$; this includes the case when there are finitely many extremal cases, which is also covered in Szab\'o's \cite[Theorem~C]{Szabo:AnalPDE}.\footnote{Szab\'o's work gives a self-contained very general approach to equivariant (SI) and also handles the non-unital situation.}  But even with the assumption that the extreme traces are compact and finite-dimensional, there are few cases where full answers to Problems~\ref{q:z-stable-product} and~\ref{q:z-stable-action} are known.  The only such results we are aware of are the case of finite groups, which is covered by \cite[Theorem B]{GHV:JMPA}, and Wouters' \cite[Theorem A]{Wouters} for $G\coloneqq\mathbb Z$.

For compact -- but no longer finite-dimensional -- extreme boundaries, the recent work \cite{GGNV:Adv} obtains $\Z$-stability of the crossed product (but not equivariant $\mathcal Z$-stability) from a Rokhlin-type condition on the induced action $G\curvearrowright \partial_eT(A)$ on the tracial boundary (\cite[Corollary D]{GGNV:Adv}). By \cite{Szabo:PLMS}, this condition is automatic whenever  $G\curvearrowright \partial_eT(A)$ is both free and has the small boundary property and, in particular, for free actions when the boundary also has finite covering dimension (\cite[Corollary E]{GGNV:Adv}). Along the way, the authors prove a general weak dynamical comparison result for actions of countable discrete amenable groups on unital simple separable $\Z$-stable {$C^*$}-algebras (\cite[Theorem C]{GGNV:Adv}), akin to the dynamical comparison property for commutative actions described in the previous section.  
The authors of \cite{GGNV:Adv} leave open the general question of equivariant $\mathcal Z$-stability, though it is handled for groups of polynomial growth by \cite[Theorem B]{Wouters}.\footnote{The result in \cite{Wouters} is stated for virtually nilpotent groups, but see footnote \ref{ft:Gromov}.}

No non-amenable group can act amenably on a von Neumann factor, but the same obstruction does not apply to amenable actions on simple {$C^*$}-algebras.\footnote{If an amenable action on a von Neumann algebra leaves a state on the centre invariant, then the acting group must be amenable (\cite[Proposition 3.6]{Anantharaman-Delaroche:MathScand}). The same principle obstructs `strongly amenable' actions of non-amenable groups on simple {$C^*$}-algebras.}  This possibility became a reality in \cite{Suzuki:JNCG,OzawaSuzuki:Selecta}: every non-amenable locally compact group acts amenably on a simple nuclear {$C^*$}-algebra. These examples were initially on Kirchberg algebras, but recently, Suzuki gave examples of such actions on simple stably finite algebras (\cite{Suzuki:AJM}), on stably projectionless algebras (\cite{Suzuki:MJM}), and in his recent paper \cite{Suzuki24}, also on simple nuclear {$C^*$}-algebras of `R\o{}rdam-type' with infinite and finite projections.  Just as a group has to be exact to admit an amenable action on a compact space (in fact, this characterises exactness; see \cite[Theorem~7.2]{AD:MA}), so too  must a group be exact to act amenably on a unital {$C^*$}-algebra (\cite[Corollary 3.6]{OzawaSuzuki:Selecta}, and in fact every exact group acts amenably on a unital Kirchberg algebra; \cite[Corollary 6.2(1)]{OzawaSuzuki:Selecta}).  At present, to the best of our knowledge, no amenable actions of any non-amenable exact groups on unital stably finite {$C^*$}-algebras are known to exist. Nonetheless, the recent developments just mentioned lead us to expect that they should, and further that the following question should have a positive answer.  

\begin{question}
Does every countable discrete exact but not amenable group act amenably on a unital stably finite classifiable {$C^*$}-algebra? 
\end{question}

For such amenable actions (if any exist), \cite[Theorem C]{GGKNV:MathAnn} gives a positive answer to Problem \ref{q:z-stable-product} whenever the acting group contains the free group on $2$-generators.   Suzuki's construction of actions of on {$C^*$}-algebras of R\o{}rdam type have crossed products also of R\o{}rdam type (\cite[Theorem 1]{Suzuki24}; this result requires the group to be free when the algebra is unital). Taking $G\coloneqq\mathbb Z$ in Suzuki's \cite[Theorem 1]{Suzuki24} gives an outer $\mathbb Z$-action on a non-$\Z$-stable unital separable nuclear {$C^*$}-algebra whose crossed product is not $\Z$-stable. In contrast, Hirshberg has constructed examples of an outer $\mathbb Z$-action on a non-$\Z$-stable unital separable nuclear {$C^*$}-algebra whose crossed product is $\Z$-stable (\cite[Theorem A]{Hirshberg:JFA}). At present, it is far from clear what a potential dynamical condition might be which characterises $\Z$-stability of the crossed product, leading to the following non-commutative version of Problem \ref{q:dynamicalcharacteriseZstable}.

\begin{question}
   Characterise when an outer amenable action $G\curvearrowright A$  of a discrete group on a unital simple separable  nuclear but non-$\Z$-stable {$C^*$}-algebra $A$ gives rise to a $\Z$-stable crossed product $A\rtimes G$.
\end{question}

In \cite{Hirshberg:JFA}, Hirshberg produces his action on a {$C^*$}-algebra $A$ constructed by suitably modifying a Villadsen algebra of the first type; in particular, $T(A)$ will be very complex and certainly not Bauer.\footnote{We have not attempted to compute $T(A)$.} Is such a construction possible on a monotracial $A$?  The main motivation for Hirshberg's work is to demonstrate that the condition that $G$ is discrete in Problems~\ref{q:z-stable-product} and \ref{q:z-stable-action} is necessary: by taking the dual action $\mathbb T\curvearrowright (A\rtimes \mathbb Z)$, he obtains an action of the circle on a unital classifiable {$C^*$}-algebra whose crossed product is not $\Z$-stable (since it is stably isomorphic to $A$); see \cite[Corollary B]{Hirshberg:JFA}.  Hirshberg asks whether such a circle action can be found on a monotracial classifiable {$C^*$}-algebra (\cite[Question 3]{Hirshberg:JFA}).

Just as in the previous section, with the rapid growth in tools giving rise to classifiability of crossed products, it becomes increasingly pressing to calculate the invariant. As with Problem \ref{q:ComputeKTheory}, this is vast in scope. 

\begin{question}\label{q:computeKcrossedprod}
    Compute $KT_u(A\rtimes G)$ for outer amenable actions $G\curvearrowright A$ on unital simple separable nuclear $\Z$-stable {$C^*$}-algebras. 
\end{question}

When $G\coloneqq \mathbb Z$ (or more generally for free groups), the Pimsner--Voiculescu sequence gives a tool for computing $K$-theory.  As studied thoroughly in \cite{Ursu:Adv}, care is needed before concluding that all traces on a crossed product $A\rtimes G$ are canonical, i.e.\ are of the form $\tau \circ E$ for some invariant trace $\tau$ on $A$ (where $E$ is the expectation from $A\rtimes_r G$ onto $A$). One instance where this holds is when the induced action of the amenable radical of $G$ on $\pi_\tau(A)''$ is properly outer\footnote{An automorphism $\theta$ of a von Neumann algebra $\mathcal M$ is \emph{properly outer} if there does not exist a non-zero central invariant projection $p\in \mathcal M$ such that $\theta|_{\mathcal Mp}$ is inner. An action $\alpha\colon G\curvearrowright \mathcal M$ is properly outer when $\alpha_g$ is properly outer whenever $g\neq 1$.} for every invariant trace $\tau$ on $A$ (\cite[Corollary~1.7]{Ursu:Adv}).  In particular, when $A$ has a unique trace and $G\curvearrowright A$ is a strongly outer action, then the reduced crossed product has a unique trace.

Non-commutative Bernoulli shifts provide particularly natural examples to consider. Given a unital {$C^*$}-algebra $A$ and countable discrete group $G$, one can equip the tensor product $A^{\otimes G}$ with the natural Bernoulli shift action by permuting the tensor factors to obtain $A^{\otimes G}\rtimes G$. This is especially natural when $A$ is strongly self-absorbing, so that $A^{\otimes G}\cong A$.  For example -- one of our favourite examples --  $\Z^{\otimes \mathbb Z}\rtimes\mathbb Z$ has the integers in both $K_0$ and $K_1 $ (via the Pimsner--Voiculescu sequence) and a unique trace.  This Bernoulli shift turns out to be a particularly nice example for visualising the role of algebraic $K_1$ in the uniqueness of automorphisms (see the discussion in \cite[Example 9.11]{CGSTW}). 

Various $K$-theory computations for Bernoulli shifts (and generalisations coming from groups acting on a set indexing the tensor product) by groups satisfying the Baum--Connes conjecture with coefficients were recently given in \cite{CEKN:MathAnn}, generalising computations in the special case $G\coloneqq\mathbb Z_2$ in \cite{Izumi:ASPM}. The main technical result handles the situation where $A$ is unital, separable, satisfies the UCT, and the inclusion $\mathbb C\to A$ gives a split injection at the level of $K_0$. This covers irrational rotation algebras, $\mathcal O_\infty$, $\Z$, and other examples. More examples beyond this context, covering the cases when $A$ is a Cuntz algebra or an AF algebra (or more general inductive limits), are given in \cite[Sections~4 and~5]{CEKN:MathAnn}. 

Another situation which has been recently examined comes from the canonical action of $SL_2(\mathbb Z)$ on the irrational rotation algebras $A_\theta$.\footnote{An element $\begin{pmatrix}a&b\\c&d\end{pmatrix}$ acts on the canonical generators $u$ and $v$ of $A_\theta$ by $u\mapsto e^{i\pi ac\theta}u^av^c$ and $v\mapsto e^{i\pi bd\theta}u^bv^d$.} 
For finite subgroups of $SL_2(\mathbb Z)$, the complete invariant of the crossed product (including the pairing!) is calculated in \cite[Theorem~4.9]{ELPW:Crelle}, and classification machinery already available at the time was deployed to show that these crossed products are AF algebras.
Given an element of infinite order in $SL_2(\mathbb Z)$, the Pimsner--Voiculescu sequence can be used to compute the $K$-theory of the crossed product $A_\theta\rtimes\mathbb Z$. This crossed product falls within the scope of classification\footnote{As $A_\theta$ is $\Z$-stable and monotracial, this holds using Sato's \cite{Sato:ASPM}, for example.} and has a unique trace, so the main work in \cite{BCHL:JFA} is to understand the pairing, which leads to a concrete (and computable) conditions on the angle and element of $SL_2(\mathbb Z)$ for two such crossed products to be isomorphic.  With the full force of classification now available, we feel there is much scope to analyse other natural examples.

\section{Dynamical presentations, groupoids, and Cartan subalgebras}\label{S14}

With a definitive classification theorem in place, one can use it to discover structure inside classifiable {$C^*$}-algebras. This is in the spirit of how Blackadar produces a symmetry of the CAR algebra whose fixed point algebra is not AF (\cite{Blackadar:Ann}) by means of a construction of what turns out to be the CAR algebra, and in which the required symmetry is baked in, followed by use of a classification theorem -- in this case Elliott's AF classification -- to recognise it as the CAR algebra.  Today, through the combination of the classification theorems and Elliott's range-of-invariant result (\cite{Elliott:CMS}), all stably finite classifiable {$C^*$}-algebras are approximately subhomogeneous of dimension at most $2$, whereas in earlier days it was necessary to obtain internal inductive limit structure to access classification.\footnote{A classic example of this is the result of Elliott and Evans that irrational rotation algebras are A$\mathbb T$ (\cite{EE:Ann}) and so are encompassed within the larger class of A$\mathbb T$ algebras of real rank zero classified in \cite{Elliott:Crelle}.}   Complementary to questions about when a (classical) dynamical system gives rise to a classifiable {$C^*$}-algebra is the question of which classifiable {$C^*$}-algebras have hitherto unseen dynamical presentations, such as being the {$C^*$}-algebra of a (twisted) groupoid.  Particular groupoid constructions may be helpful to solve other problems about classifiable {$C^*$}-algebras; see our discussion around Problem \ref{q:ExistActions} and the introduction to \cite{CFaH:EM} for some further examples of this in the purely infinite case. 

\begin{question}\label{groupoidq}
Given a classifiable {$C^*$}-algebra $A$, does there exist an \'etale groupoid $\mathcal G$ such that $A \cong C^*(\mathcal G)$?
\end{question}

This question has already received significant attention, beginning with the Jiang--Su algebra as a special case (\cite{DPS:Crelle}).
For purely infinite {$C^*$}-algebras, Spielberg resolved the question affirmatively in the non-unital case (\cite{Spielberg:JOT}), and adaptations of his construction handle the unital case (\cite{LiRenault:TAMS,CFaH:EM}).

Allowing twists in the groupoid, Xin Li gave a positive answer (\cite{Li:IM}), so Problem~\ref{groupoidq} is asking whether one can do without the twist. Li's construction makes use of Elliott's inductive limit range-of-invariant result \cite{Elliott:CMS}; he produces a compatible system of Cartan subalgebras in the building blocks of this inductive limit so that they produce a Cartan subalgebra (and thereby a twisted groupoid presentation by Renault's reconstruction theorem from \cite{RenaultIMSB}) in the limit.  When there is no torsion in $K_0$, Li's construction does not require a twist, \cite[Corollary 1.8(i)]{Li:IM}. Further attention to this question can be found for example in \cite{Putnam:MA,AustinMitra:NYJM,DPS:PAMS,DPS:GGD,DPS:MA,Li:IMRN}.  While Li's result is definitive for twisted groupoids, it has the drawback that as the construction is underpinned by the known inductive limit approach to the range-of-the-invariant, it does not give a fundamentally new description of the stably finite classifiable {$C^*$}-algebras. An earlier approach by Deeley, Putnam and Strung is more dynamical; in \cite{DPS:Crelle}, they obtain a groupoid model for $\Z$ (and hence a Cartan subalgebra of $\Z$) by breaking the orbits in a suitable minimal dynamical system.  Extending these orbit breaking ideas, Deeley, Putnam and Strung obtained further examples of untwisted groupoid models, allowing for quite a lot of torsion in $K_0$, albeit at present with restrictions on the possible trace spaces; see \cite[Theorem 6.3]{DPS:MA}.\footnote{We thank Robin Deeley and Karen Strung for helpful comments on these themes.} At present these methods do not yet exhaust the invariant, so we would welcome further (twisted) groupoid models for all classifiable {$C^*$}-algebras, especially if these could be constructed in a more directly dynamical fashion without requiring inductive limits.

One can also strengthen Problem \ref{groupoidq} to ask if (or when) further structure can be imposed on the groupoid (e.g.\ principal, restrictions on the unit space, etc.).   In his recent work \cite{Wu:arXiv24},  Victor Wu obtains the most dynamical possible presentation for stable UCT Kirchberg algebras: they arise as crossed products by group actions on spaces.

\begin{theorem}[{Wu; \cite[Theorem B]{Wu:arXiv24}}]
    Let $A$ be a stable UCT Kirchberg algebra. Then there is a (necessarily amenable) action $G\curvearrowright X$ of a discrete (necessarily non-amenable) group on a locally compact Hausdorff space such that $A\cong C_0(X)\rtimes G$.\footnote{Wu's result is constructive; $G$ is built using the fundamental group of a directed graph of groups from the $K$-theory data of $A$, and $X$ is a directed Bass--Serre tree for this graph.} 
\end{theorem}

It is natural to ask for a corresponding result in the stably finite setting.

\begin{question}
Is every stable, stably finite classifiable {$C^*$}-algebra of the form $C_0(X)\rtimes G$ for a suitable action $G\curvearrowright X$ of an amenable group on a locally compact Hausdorff space?
\end{question}

One could even ask the analogous question for unital classifiable algebras (it appears to be open, even for UCT Kirchberg algebras, as far as we know). While here there are naturally occurring examples, such as UHF algebras (\cite[Example~8.1.24]{GKPT:Book}), or $\mathcal O_2$ (which is a crossed product coming from a minimal action of $(\mathbb Z/2\mathbb Z)*(\mathbb Z/3\mathbb Z)$; see \cite[Example 4.4.7]{Rordam:Book}, which attributes this to Archbold and Kumjian, independently), it also seems likely that there are restrictions. A very recently announced result of Ma and Jianchao Wu shows that if $\Z$ is a crossed product $C(X)\rtimes G$, then $G$ must be amenable, torsion-free, and rationally acyclic; see \cite[Theorem C]{MaWu24}.  It seems open whether any such groups exist.\footnote{We are grateful to Dawid Kielak for explaining to us that finitely generated infinite amenable groups have Euler characteristic zero, so those of type F (i.e.\ those groups whose classifying space has the homotopy type of a finite CW-complex) cannot be rationally acyclic (as the alternating sum of the rational Betti numbers must be zero).  It is open whether all finitely generated infinite amenable torsion-free groups are of type F.}

\begin{question}
When is a unital classifiable {$C^*$}-algebra of the form $C(X)\rtimes G$ for a suitable action $G\curvearrowright X$ of an amenable group on a compact Hausdorff space?
\end{question}

This question can also be asked for particular groups $G$, such as $G=\mathbb Z$, when there are certainly obstructions arising from the Pimsner--Voiculescu sequence to writing general unital classifiable {$C^*$}-algebras as $\mathbb Z$-crossed products -- no unital {$C^*$}-algebra with trivial $K_1$-group can have such a crossed product picture. This question has been taken up by Deeley, Putnam and Strung in \cite{DPS:GGD,DPS:MA}. They determine all possible $K$-groups of {$C^*$}-algebras $C(X)\rtimes \mathbb Z$ arising as crossed products from any actions on compact metrisable spaces with finitely generated $K$-theory and then construct uniquely ergodic minimal actions on such a space realising any possible pair of $K$-groups in a classifiable crossed product, computing the pairing (\cite[Theorem 4.5, Corollary 4.6]{DPS:MA}).  While their main statement obtains unique ergodicity, and so a unique trace on the crossed product, crossed products with more complex trace spaces are certainly possible. The full range of $KT_u(\cdot)$ of such crossed products (allowing also for spaces with infinitely generated $K$-theory) is not currently understood.

One can also ask similar questions regarding the presentation of classifiable {$C^*$}-algebras through non-commutative dynamics, particularly as crossed products of other classifiable {$C^*$}-algebras. 
 This question is closely tied to Problem \ref{q:computeKcrossedprod}.

\begin{question}\label{q:Whencrossedproduct}
When can a classifiable {$C^*$}-algebra be written in the form $A\rtimes G$ for some  outer amenable action of a non-trivial countable discrete group $G$ on a classifiable {$C^*$}-algebra $A$?
\end{question}

Stable UCT Kirchberg algebras are $\mathbb Z$-crossed products of simple separable stable A$\mathbb T$ algebras of real rank zero (see \cite[Proposition~4.3.3]{Rordam:Book}; one can use an AF algebra in the case when the Kirchberg algebra has torsion-free $K_1$-group).  Moreover, as $\Z\otimes \mathcal K$ is a crossed product of the form $B\rtimes\mathbb Z$ for some action $\alpha\colon \mathbb Z\curvearrowright B$ on a classifiable {$C^*$}-algebra $B$ (this is worked out over the course of \cite[Section~3]{EST:CMP}), so too is $A\otimes \mathcal K$ for any stable classifiable $A$ (consider the action $1 \otimes \alpha$ on $A\otimes B$).  Constructions of this form (and relaxations to allow $B$ to fall into a class of non-simple nuclear {$C^*$}-algebras which is suitably classified, such as AF algebras or A$\mathbb T$ algebras of real rank zero) have been applied to build flows on certain classifiable {$C^*$}-algebras exhibiting arbitrary KMS behaviour (\cite{ET:CRMASSRC,ElliottSato22,EST:CMP,Neagu24}). Due to the obstructions to realising unital classifiable {$C^*$}-algebras as $\mathbb Z$-crossed products (described above for commutative base algebras), we are interested in answers to Problem \ref{q:Whencrossedproduct} both ranging across all groups and actions and also for particular fixed groups such as $\mathbb Z$.   In \cite[Theorem B]{Jacelon:Sigma}, Jacelon gives a positive answer for $\mathbb Z$-actions within the setting of $KK$-contractible stably finite classifiable algebras, with compact tracial state space.

Following Connes' theorem, an analogous result was obtained by Connes, Feldman, and Weiss for amenable equivalence relations. Phrased in von Neumann algebraic language, a separably acting injective von Neumann algebra $\mathcal M$ has a unique Cartan subalgebra (in the sense of von Neumann algebras) up to conjugacy by an automorphism of $\mathcal M$ (\cite{CFW:ETDS}).  
The most naive analogues of this are patently false for classifiable {$C^*$}-algebras, as it is generally not hard to construct Cartan subalgebras with different spectra, and there are examples of algebras with very many non-conjugate Cartans  with the \emph{same} spectra (see \cite[Theorem 1.4]{Li:IMRN} or \cite[Proposition~C]{GeffenKranz:MJM}).  One situation where uniqueness is known -- going back to Power (\cite[Theorem 5.7]{Power:Book}) -- is for so-called AF Cartans.\footnote{An AF Cartan is a Cartan subalgebra which arises as an inductive limit of a system of finite-dimensional Cartans.}  The downside is that the condition is very much in terms of an inductive limit and not abstract.  Indeed, while we can characterise AF algebras within the classifiable class in terms of the invariant, there is not even a candidate abstract characterisation of AF algebras within all separable {$C^*$}-algebras (by such an abstract characterisation, we mean something akin -- in terms of its ease of checkability -- to injectivity for finite von Neumann algebras; such a characterisation has been sought at least as far back as \cite[Problem 6]{Effros:Kingston}).\footnote{Inspired by Glimm's analogous result for UHF algebras (\cite{Glimm:TAMS}), Bratteli demonstrated that for separable {$C^*$}-algebras being AF (i.e.\ an inductive limit of finite dimensional {$C^*$}-algebras) is equivalent to the somewhat more checkable condition of being `locally AF' (i.e.\  every finite subset of $A$ can be approximated by elements of a finite-dimensional subalgebra); (\cite{Bratteli:TAMS}).  This is not really an abstract characterisation, like injectivity is of hyperfiniteness, but more in the spirit of Murray and von Neumann's various equivalent characterisations of hyperfiniteness for II$_1$ factors with separable preduals in \cite{MurrayVonNeumann4}. Nevertheless, the local AF-test is certainly useful in some situations. For example, Blackadar uses it in \cite{Blackadar:Ann} to give access to Elliott's classification of AF algebras when constructing the exotic symmetry on the CAR algebra we described in the opening of this section (see \cite[Section 4.2 and in particular Remark 4.2.3]{Blackadar:Ann}). Of course, Effros was certainly well aware of Bratteli's result when he asked for an `effective criteria for determining whether or not a [separable] {$C^*$}-algebra is AF' in \cite{Effros:Kingston}.  For similar reasons we do not regard having nuclear dimension $0$ as being a suitable abstract characterisation of being AF for this purpose. In the unital case, nuclear dimension $0$ is very quickly seen to be the same as local AF; the non-unital case is trickier and goes back to \cite[Theorem~3.4]{Winter:JFA}.}
Does the corresponding question for Cartans become more tractable when we assume the underlying algebra is already known to be AF? Quoting directly from \cite{Effros:Kingston}, `although the question is rather vague, there will be no problem recognizing when we have found the correct answer.'

\begin{challenge}
    Is there an abstract criterion of when a Cartan with Cantor spectrum in an AF algebra is an AF Cartan? 
\end{challenge}

While AF Cartans in unital {$C^*$}-algebras have Cantor spectrum (at least in the case when there are no non-zero finite-dimensional representations), they need not be the unique Cantor Cartans. In \cite{MitscherSpielberg:TAMS}, Mitscher and Spielberg constructed new Cantor Cartans in the irrational rotation AF algebras, i.e.\ those AF algebras with the same ordered $K_0$ as an irrational rotational algebra. Their example is distinguished from the AF Cartan as the underlying groupoid is not principal, so their Cartan  fails to have the unique extension property (and so is not a {$C^*$}-diagonal); see \cite[Corollary 7.12 and Remark 7.13]{MitscherSpielberg:TAMS}. The strategy goes through classification; they build an appropriate groupoid so that the {$C^*$}-algebra has the same classification invariant as an irrational rotation AF algebra, and then estimate the nuclear dimension (initially obtaining the upper bound of $3$) through decomposing building blocks of their groupoid as extensions. This showcases how building well-chosen groupoid models and appealing to classification can be used to uncover new structure. 

At present, it is not known how broadly the phenomena of \cite{MitscherSpielberg:TAMS} extend to other AF algebras nor how to obtain a non-AF Cantor diagonal in an AF algebra.  Even so, while one might hope for some kind of classification of Cantor diagonals or Cartans, or at least those satisfying some yet-to-be-determined regularity condition, some of the authors expect there to be many such Cartans -- potentially so many that they cannot be reasonably classified. For this reason the next problem is stated very vaguely.

\begin{challenge}
    Is there a reasonable framework for classifying suitable Cartan subalgebras with Cantor spectrum in AF algebras up to conjugacy by an automorphism?
\end{challenge}

In a different direction, Connes and Jones exhibited the first (necessarily non-amenable) II$_1$ factor with multiple Cartan subalgebras, using relative property $\Gamma$ as the distinguishing invariant (\cite{ConnesJones:BAMS}). Can such phenomena be found within the class of simple separable nuclear {$C^*$}-algebras? 

\begin{question}\label{twocartans}
    Does there exist a unital simple separable nuclear {$C^*$}-algebra $A$ which contains Cartan subalgebras both with and without relative uniform property $\Gamma$? 
\end{question}

The loose conjecture following Problem \ref{q:dynamicalcharacteriseZstable} suggests that a Cartan pair $C(X)\subset C(X)\rtimes G$ coming from a free minimal action of a discrete group should remember the small boundary property for the action or (equivalently by \cite{ElliottNiu24,KLTV}) uniform property $\Gamma$ for the pair. So if this conjecture holds, the answer to Problem~\ref{twocartans} is no for Cartan subalgebras coming from group actions. At least in part, Problem \ref{twocartans} is related to whether this conjecture should hold much more generally -- for all twisted groupoids giving rise to a simple separable nuclear {$C^*$}-algebra -- and whether the equivalence of relative uniform property $\Gamma$ and the small boundary property extends to the level of generality of Cartan subalgebras.



\section{Finding loops of automorphisms}

For a unital Kirchberg algebra $A$, Dadarlat identified the homotopy groups $\pi_k(\mathrm{Aut}(A))$ with $KK^1(C_uA, S^kA)$ for $k \geq 1$; \cite{Dadarlat:JNCG}.\footnote{Here $\mathrm{Aut}(A)$ is equipped with the point-norm topology, $C_uA$ is the mapping cone of the unital inclusion $\mathbb C \rightarrow A$ and $SA$ is the suspension of $A$. Explicitly, $C_uA = \{f \in C([0, 1], A) : f(0) \in \mathbb C1_A \text{ and } f(1) = 0\}$ and $SA = C_0((0, 1), A)$.}  The isomorphism is obtained abstractly, using the Kirchberg--Phillips theorem to classify maps $A \rightarrow C(S^k, A)$, where $S^k$ is the $k$-dimensional sphere, and even when $k = 1$, there is no systematic way of `seeing' the loops in $\mathrm{Aut}(A)$.

In the case when $A=\mathcal O_n$, a direct computation of Dadarlat's invariant shows that $\pi_k(\mathrm{Aut}(A))$ can be identified with $\mathbb Z / (n-1)\mathbb Z$ when $k$ is odd and $0$ when $k$ is even.
In the case when $k = 1$, the generating loop in $\mathrm{Aut}(\mathcal O_n)$ is given by the canonical gauge action $\mathbb T \curvearrowright \mathcal O_n$.  This is one of the few cases where we know an explicit description of a non-trivial fundamental group of the automorphism group.

In the stably finite setting, work in progress of Jamie Gabe and the first-named author computes the homotopy groups of $\mathrm{Aut}(A)$ for all unital finite classifiable $A$.  In particular, Dadarlat's computation also holds in the unique trace setting so, for example, we have $\pi_1(\mathrm{Aut}(A_\theta)) \cong \mathbb Z^6$, where $A_\theta$ is an irrational rotation algebra.  The natural action $\gamma \colon \mathbb T^2 \curvearrowright A_\theta$ given by rotating the canonical unitary generators, defines a group homomorphism $\pi_1(\gamma) \colon \mathbb Z^2 \rightarrow \pi_1(\mathrm{Aut}(A_\theta))$.  Further, $\pi_1(\gamma)$ is injective.  This can be seen through the composition
\begin{equation}
\begin{tikzcd}
    \pi_1(\mathbb T^2) \arrow{r}{\pi_1(\gamma)} & \pi_1(\mathrm{Aut}(A_\theta)) \arrow{r} & \pi_1(U(A_\theta)) \arrow{r}{\mathrm{Bott}} & K_0(A_\theta) \arrow{r}{\mathrm{tr}} & \mathbb R,
\end{tikzcd}
\end{equation}
where the second map is evaluation at one of the two canonical generators of $A_\theta$.  These maps are integer-valued and precisely encode the winding number of a loop in the torus around the two copies of the circle.  So the direct sum of these maps is injective.  However while we strongly suspect these loops define a $\mathbb Z^2$ direct summand of the fundamental group (and we would welcome a proof), it is unclear where the remaining $4$ generators comes from.  

The analogous question can also be asked in the corresponding Kirchberg algebras $A_\theta\otimes\mathcal O_\infty$, where the action $\gamma \otimes \mathrm{id} \colon \mathbb T^2 \curvearrowright A_\theta \otimes \mathcal O_\infty$ also gives a group homomorphism $\pi_1(\gamma \otimes \mathrm{id})$ from $\mathbb Z^2$ into $\pi_1(\mathrm{Aut}(A_\theta \otimes \mathcal O_\infty))$.  The embedding $\mathrm{Aut}(A_\theta) \rightarrow \mathrm{Aut}(A_\theta \otimes \mathcal O_\infty) \colon \alpha \mapsto \alpha \otimes \mathrm{id}$ is a weak homotopy equivalence due to the the $KK$-theoretic computation of the homotopy groups, so $\pi_1(\gamma \otimes \mathrm{id})$ is also injective.

\begin{question}\label{q:pi1}
    Find explicit loops generating the group $\pi_1(\mathrm{Aut}(A_\theta))\cong \mathbb Z^6$, or the group $\mathrm{Aut}(A)$, where $A$ is the unital UCT-Kirchberg algebra with the same $K$-theory as $A_\theta$.  More generally obtain explicit descriptions of generators of $\pi_1(\mathrm{Aut}(A))$ for other UCT-Kirchberg or monotracial classifiable $C^*$-algebras $A$.
\end{question}

  In the spirit of the results of Mitscher and Spielberg (\cite{MitscherSpielberg:TAMS}) described in the previous section, a potential approach to finding the other generators is to produce a groupoid (or some other dynamical/geometric object) whose {$C^*$}-algebra is isomorphic to $A_\theta \otimes \mathcal O_\infty$ (via a non-constructive isomorphism coming from classification), where one can see the extra loops of automorphisms in the underlying groupoid. We would expect this to be easier in the Kirchberg setting, where, for example, one could look for generators in Cuntz-Krieger models, compared with the stably finite setting where the presently available range of groupoid models is less extensive.

\section{Classifying actions on {$C^*$}-algebras}\label{sec:classaction}

In the von Neumann algebra setting, the fact that $\mathcal R$ is the unique injective $\mathrm{II}_1$ factor leads to many different constructions of $\mathcal R$.  By exploiting symmetries of the various constructions of $\mathcal R$, this leads to many (cocycle) actions of groups on $\mathcal R$ which, at face value, appear to be very different.  
This is part of what makes Ocneanu's uniqueness theorem so striking.  For a countable discrete group $G$, constructing an outer action of $G$ on $\mathcal R$ is easy: let $\mathcal R^{\otimes G}$ denote the tensor product of $G$-many copies of $\mathcal R$, and consider the \emph{non-commutative Bernoulli shift} $G \curvearrowright \mathcal R^{\otimes G}$ given by permuting the tensor factors.

In the {$C^*$}-algebra setting, both the existence and uniqueness problems for group actions are far more delicate.  
Just as with the Bernoulli shift on $\mathcal R^{\otimes G}$ just described, any discrete group acts outerly on the Jiang--Su algebra as $G \curvearrowright \mathcal Z^{\otimes G} \cong \mathcal Z$.
By tensoring with the trivial action, $G$ also acts outerly on every $\mathcal Z$-stable {$C^*$}-algebra.  However, even in the case that $G$ is amenable, this action will typically be far from unique up to cocycle conjugacy.  For example, an action of this form on a $\mathcal Z$-stable {$C^*$}-algebra $A$ will always induce the trivial action on $K$-theory and traces.

One situation where uniqueness can be expected -- at least in the torsion-free case\footnote{There are $K$-theoretic obstructions in the presence of torsion; see \cite[Theorem 4.8]{Izumi:Duke}.} -- is for strongly outer actions on strongly self-absorbing {$C^*$}-algebras. The following problem goes back to \cite[Conjecture 2]{Izumi:ICM} and is reiterated as \cite[Conjecture A]{Szabo:CMP19}.

\begin{question}\label{q:UniqueActionSSA}
Let $\mathcal D$ be a strongly self-absorbing {$C^*$}-algebra and let $G$ be a countable discrete torsion-free amenable group.  Are all strongly outer actions $G\curvearrowright\mathcal D$ cocycle conjugate?
\end{question}

As one of many consequences of their spectacular dynamical Kirchberg--Phillips theorem (Theorem \ref{ThmDynKP} below), Gabe and Szab\'o give a positive answer to Problem~\ref{q:UniqueActionSSA} when $\mathcal D$ is Kirchberg (in that case, they even extend it further to amenable actions of exact groups with the Haagerup property). Historically, the first result was on the stably finite side: Kishimoto's solution for integer actions on UHF algebras.  Subsequently, there has been substantial development through \cite{Matui:Crelle,IzumiMatui:Adv,MatuiSato:AJM} and other works for actions on UHF algebras and on $\Z$,  to the current state-of-the-art in \cite[Theorem C]{Szabo:CMP19}, which gives a positive answer for the class of torsion-free elementary amenable groups generated by the trivial group and closed under directed unions and extensions by $\mathbb Z$. No UCT hypothesis is required in \cite{Szabo:CMP19}.

Returning to situations where there are non-trivial automorphisms of the invariant available, it is natural to ask which actions on the invariant can be realised by group actions on the underlying {$C^*$}-algebra. In \cite[10.11.3]{Blackadar-kbook}, Blackadar asked whether every automorphism of order $n$ of the scaled ordered $K_0$-group of an AF algebra $A$ can be lifted to an order $n$ automorphism of $A$. This remains open in general, even if $n=2$ and $A$ is simple!  In general, more information than just $K$-theory and traces is needed to determine the approximate unitary equivalent class of an automorphism of a unital classifiable algebra; one also needs to know the behaviour on total $K$-theory and a certain algebraic $K_1$-group developed by Thomsen \cite{Thomsen:RIMS}. The \emph{total invariant} $\underline{K}T_u(\,\cdot\,)$, as set out in \cite[Sections 2 and 3]{CGSTW}, collects all the data required for uniqueness and is the appropriate invariant for a general version of Blackadar's question. 

\begin{question}\label{q:ExistActions}
Let $\alpha$ be an action of a countable discrete amenable group $G$ on the total invariant $\underline{K}T_u(A)$ (or on $KT_u(A)$) of a unital classifiable {$C^*$}-algebra $A$. When is there an action of $G$ on $A$ inducing $\alpha$?
\end{question}

There is a general abstract procedure for extending (iso)morphisms of the invariant $KT_u(\,\cdot\,)$ to $\underline{K}T_u(\,\cdot\,)$ (\cite[Theorem 3.9]{CGSTW}). Very recently, this has been extended to actions in \cite[Corollary 8.3]{Nielsen:Thesis}; so a positive solution to Problem \ref{q:ExistActions} for $\underline{K}T_u(\,\cdot\,)$ would also solve the problem for actions on $KT_u(\,\cdot\,)$ (the converse is unclear, however).

By the existence part of the classification of automorphisms on unital classifiable algebras (\cite[Theorem 9.8]{CGSTW}), Problem \ref{q:ExistActions} has a positive answer when $G$ is $\mathbb Z$ or, more generally, is a free group (though in Problem \ref{q:ExistActions} above, we only asked the question in the amenable setting).  

Problem \ref{q:ExistActions} has a positive answer for finite groups in the presence of enough UHF-stability. This was established by Barlak and Szab\'o (\cite[Theorem 2.3]{BarlakSzabo:TAMS}) whenever automorphisms are classified up to approximate unitary equivalence. Combining their work with the classification theorems, gives the following result.

\begin{theorem}[{\cite[Theorem 9.14]{CGSTW}}]\label{thm:existractions}
Let $G$ be a finite group and let $A$ be a unital classifiable {$C^*$}-algebra with $A\cong A\otimes M_{|G|^\infty}$. Then any action $G\curvearrowright\underline{K}T_u(A)$ lifts to an action $G\curvearrowright A$ (which additionally can be chosen to have the Rokhlin property).  In particular when $A$ is a unital classifiable {$C^*$}-algebra which is stable under tensoring by the universal UHF algebra $\mathcal Q$, then any finite group action on $\underline{K}T_u(A)$ lifts to $A$.
\end{theorem}

Theorem \ref{thm:existractions} comes with a uniqueness counterpart for Rokhlin actions on $M_{|G|^\infty}$-stable classifiable algebras by their induced actions on $\underline{K}T_u(\,\cdot\,)$. This is \cite[Corollary 9.15]{CGSTW}, heavily using results of Izumi from \cite{Izumi:Duke}.

Without UHF-stability, some positive results for Problem \ref{q:ExistActions} have been obtained in the purely infinite case.  For example, actions of cyclic groups on the $K$-theory of UCT Kirchberg algebras are induced by actions on the {$C^*$}-algebras (\cite{Katsura08}).  The approach is to build new models of UCT Kirchberg algebras with the required actions baked in and then appeal to classification.  Can this approach now be used to make significant progress on the stably finite side?

Turning now to the general uniqueness problem, Gabe and Szab\'o's very recent dynamical version of the Kirchberg--Phillips theorem (\cite{Gabe-Szabo}) is revolutionary for the classification of {$C^*$}-dynamics, just as the Kirchberg--Phillips theorem was for {$C^*$}-algebras.  Specialising to the discrete amenable case, their theorem reads as follows. 

\begin{theorem}[{\cite[Theorem B]{Gabe-Szabo}}]\label{ThmDynKP}
Let $G$ be a countable discrete amenable group and let $\alpha\colon G\curvearrowright A$ and $\beta\colon G\curvearrowright B$ be outer actions of $G$ on stable Kirchberg algebras $A$ and $B$. Then $\alpha$ and $\beta$ are  cocycle conjugate if and only if they are equivalent in the $G$-equivariant $KK$-category $KK^G$.
\end{theorem}

As set out in the introduction to \cite{Gabe-Szabo}, historically, the approach taken to uniqueness of actions typically relied on the Rokhlin property together with an intertwining technique of Evans and Kishimoto (\cite{EvansKishimoto}); see \cite{IzumiMatuiI,IzumiMatuiII}.  Gabe and Szab\'o approach the problem in a completely different fashion, which is much more tightly linked to the Elliott intertwining technology used in the classification of {$C^*$}-algebras
through Szab\'o's `cocycle category' (\cite{Szabo:JFA}) and tools such as an equivariant stable uniqueness theorem (\cite{GabeSzabo:AJM}).

On the stably finite side, outerness is not a strong enough condition to expect uniqueness -- one should work with strong outerness, so that one at least has uniqueness at the von Neumann level via Ocneanu.   We can, and should, now hope for a stably finite counterpart to the Gabe--Szab\'o theorem and also that such a result doesn't take the full 20 years required to complete the stably finite part of the unital classification theorem following Kirchberg--Phillips!  

\begin{challenge}
    Let $A$ be a unital simple separable nuclear stably finite $\mathcal Z$-stable {$C^*$}-algebra and let $G$ be a countable discrete amenable group. Classify strongly outer (cocycle) actions of $G$ on $A$ up to cocycle conjugacy in terms of some combination of equivariant $K$-theory (or $KK$-theory), traces, and group (co)homology, perhaps under suitable equivariant regularity and/or UCT hypotheses. 
\end{challenge}

The Gabe--Szab\'o theorem works beyond countable discrete groups (\cite[Theorem F]{Gabe-Szabo}), for example recapturing the uniqueness of a Rokhlin flow on Kirchberg algebras from \cite{Szabo:CMP21} and extending this to $\mathbb R^d$ actions (\cite[Corollary 6.15]{Gabe-Szabo}). However, outside the discrete setting, one needs more than just outerness of the action, and the required condition -- \emph{isometric shift absorption} -- does not apply in all cases of interest. Indeed, isometric shift absorption implies equivariant $\mathcal O_\infty$-stability (\cite[Proposition 3.9]{Gabe-Szabo}), and hence $\mathcal O_\infty$-stability of the crossed product.  Accordingly, the Gabe--Szabó theorem does not cover the situation where the crossed product is stably finite. One particularly prominent class of examples where this happens is the canonical gauge actions of the circle on simple Cuntz--Krieger algebras coming from irreducible adjacency matrices (which are not periodic). The study of these actions goes back to Cuntz and Krieger's foundational paper \cite{CuntzKrieger:IM}, constructing their now eponymous algebras from symbolic dynamics. 

We are grateful to Jamie Gabe for drawing our attention to Bratteli and Kishimoto's work (\cite{BratteliKishimoto:QJM}, building on \cite{EvansKishimoto,ElliottEvansKishimoto:MS}) on trace-scaling automorphisms of simple separable AF-algebras whose cone of densely defined lower semicontinuous traces has a base with finitely many extreme points.  Bratteli and Kishimoto show that such automorphisms have the Rokhlin property, and obtain a classification up to cocycle conjugacy by the induced action on $K_0$ (\cite[Theorem 1.1, Corollary 1.2]{BratteliKishimoto:QJM}).  Transferring their classifications to the dual action, they classified (both up to conjugacy, and up to cocycle conjugacy) actions $\alpha\colon\mathbb T\curvearrowright A$ of $\mathbb T$ on a unital Kirchberg\footnote{Bratteli and Kishimoto do not explicitly include nuclearity or the UCT in their hypotheses, but applying the dual action to the AF crossed product, it follows that $A$ must be nuclear and satisfy the UCT.} algebra $A$ such that the crossed product $A\rtimes\mathbb T$ is a simple AF-algebra with a unique tracial ray.\footnote{Having a unique tracial ray means there is a unique non-zero lower semicontinuous tracial weight up to scaling.} See \cite[Corollaries 4.1 and 4.2]{BratteliKishimoto:QJM} for precise statements including the description of the classification invariants in terms of the conjugacy of the induced actions on $K_0$ of the crossed product in a fashion compatible with the class of the spectral projection at $0$ of the unitary implementing the action. A stabilised version is given as \cite[Corollary 4.3]{BratteliKishimoto:QJM}.  Importantly, but not mentioned in \cite{BratteliKishimoto:QJM}, gauge actions on Cuntz-Krieger algebras coming from irreducible and aperiodic adjacency matrices satisfy all the required hypotheses.\footnote{The required unique trace on the crossed product follows from \cite{EW:MJ}, which shows that these actions have a unique KMS-state.} The implications of Bratteli and Kishimoto's result for Cuntz--Krieger algebras are set out in the introduction to \cite{CDE:APDE}. The point is that for Cuntz-Krieger algebras, the invariant in \cite[Corollary 4.3]{BratteliKishimoto:QJM} is precisely Krieger's dimension group from his classification of irreducible and aperiodic adjacency matrices up to shift equivalence from \cite[Theorem 4.2]{Krieger:IM} (see \cite[Theorem 6.4]{Effros:CBMS}).  As a consequence, shift equivalence of aperiodic irreducible  adjacency matrices $A$ and $B$ is determined by conjugacy of the gauge actions on the stabilisations $\mathcal O_A\otimes\mathcal K$ and $\mathcal O_B\otimes\mathcal K$ (see \cite[Remark 7.5]{CDE:APDE}).  

While the Bratteli--Kishimoto classification is extremely effective in this setting, ideally one would be able to access classification through an abstract condition on the action rather than through internal structure -- particularly that of being AF -- of the fixed-point algebra.

\begin{question}
Find a general classification framework for $\mathbb T$-actions on Kirchberg algebras with stably finite crossed products which encompasses the gauge actions on Cuntz--Krieger algebras described above as well as gauge actions arising from other more general constructions of UCT Kirchberg algebras.   
\end{question}

Another natural family to consider are the quasifree flows $\alpha^{(\lambda)}\colon \mathbb R\curvearrowright \mathcal O_2$ on $\mathcal O_2$, indexed by $\lambda\in\mathbb R$, given by $\alpha_t^{(\lambda)}(s_1)=e^{it}s_1$ and $\alpha_t^{(\lambda)}(s_2)=e^{i\lambda t}s_2$ (where $s_1,s_2$ are the canonical generators of $\mathcal O_2$).  These were analysed in \cite{Kishimoto:YMJ} and give rise to simple crossed products when $\lambda$ is irrational. When $\lambda<0$ is irrational, the flow is Rokhlin (\cite{Kishimoto:IJM}) and so is the unique such up to cocycle conjugacy.  Things are trickier when $\lambda>0$ and is irrational; in this case, the crossed product is stably projectionless with a unique tracial ray (\cite{KK:CJM}).  In \cite{Dean:CJM}, Dean showed that for generically many irrational $\lambda>0$, the crossed product is the stabilised Razak--Jacelon algebra $\mathcal W\otimes\mathcal K$ (but it remains open to establish this for all irrational $\lambda>0$).

\begin{question}
Determine when the quasifree flows $\alpha^{(\lambda)}$ on the Cuntz algebra $\mathcal O_2$ described above are cocycle conjugate for irrational $\lambda>0$.
\end{question}

Finally in this section, there are far fewer results for cocycle actions.  In the von Neumann algebra setting, every outer cocycle action of a discrete amenable group on any II$_1$ factor is cocycle conjugate to a genuine action (\cite{Popa21}).  
The following untwisting question for cocycle actions on {$C^*$}-algebra was promoted by Shlyakhtenko at several conferences.

\begin{challenge}[Shlyakhtenko]\label{q:Untwist}
    Determine when a strongly outer cocycle action of a countable discrete amenable group on a classifiable {$C^*$}-algebra $A$ is cocycle conjugate to a genuine action
    in terms of some combination of $K$-theoretic invariants of the algebra and (scalar-valued) cohomological invariants of the group.
\end{challenge}

There are known examples of strongly outer cocycle actions of amenable groups on classifiable {$C^*$}-algebras which cannot be untwisted.  For example, see \cite[Examples~8.19 and~8.20]{IzumiMatuiI}.
Matui and Sato studied Problem \ref{q:Untwist} for actions on UHF algebras and on $\mathcal Z$; in \cite[Theorems~8.6 and 8.8]{MatuiSato:CMP} they give a condition for untwisting cocycles for strongly outer actions of $\mathbb Z^2$ on UHF algebras and $\mathcal Z$.
In \cite{MatuiSato:AJM}, they show that all strongly  outer cocycle actions of the Klein bottle group on $\mathcal Z$ are cocycle conjugate.

\section{Classification without the UCT}

In the setting of Kirchberg algebras, classification is possible without the UCT.  One has that two unital Kirchberg algebras $A$ and $B$ are isomorphic if and only if there is an invertible element $\kappa \in KK(A, B)$ that preserves the class of the unit on $K_0$.  The UCT can then be used to compute this $KK$-group in terms of $K_*$, which leads to the traceless case of Theorem~\ref{Thm:UnitalClassification} -- the most well-known (and well-used) form of the Kirchberg--Phillips theorem.  
By contrast, in the stably finite classification theorem, the UCT plays a more pervasive role -- see \cite[Section~1.3.6]{CGSTW} for a discussion. 
This combined with the more involved invariant in which traces are entwined with $K$-theory makes it less clear what a stably finite classification theorem should look like without a UCT assumption.

\begin{question}[{\cite[Conjecture~D]{Schafhauser24}}]\label{q:KKclassification}
    Let $A$ and $B$ be unital separable simple nuclear $\mathcal Z$-stable {$C^*$}-algebras.  If there is an invertible $\kappa \in KK(A, B)$ 
    preserving the unit on $K_0$ and an affine homeomorphism $\gamma \colon T(B) \rightarrow T(A)$ such that $\kappa$ and $\gamma$ are compatible in a suitable sense, do we have $A \cong B$?
\end{question}

Finding the correct compatibility condition should be viewed as part of the problem.  The minimal compatibility condition (which is sufficient under the UCT by Theorem~\ref{Thm:UnitalClassification}) is that $\kappa_0$ and $\gamma$ are adjoints under the natural pairing $T \times K_0 \rightarrow \mathbb R$; i.e.\
\begin{equation}\label{eq:pairing}
    \langle \gamma(\tau), x \rangle = \langle \tau, \kappa_0(x) \rangle, \quad \tau \in T(B),\ x \in K_0(A).
\end{equation}
This is the compatibility condition optimistically suggested in \cite[Conjecture~D]{Schafhauser24}.  Without the UCT, it's unclear whether this condition would be strong enough to capture the full interaction between $KK$-theory and traces.
This is closely related to Problem~\ref{q:RUCT}; every trace $\sigma \in T(A)$ induces an element $[\sigma] \in KK(A,\mathcal R^\omega)$;\footnote{Using separability and nuclearity of $A$, Connes' theorem provides a $^*$-homomorphism $\theta \colon A \rightarrow \mathcal R^\omega$ with $\mathrm{tr}\circ \theta = \sigma$, which is unique up to unitary equivalence.
Then $[\sigma] \coloneqq [\theta] \in KK(A, \mathcal R^\omega)$ is well-defined.} without the UCT for $A$ (or, at least, the injectivity part of the UCT for the pair $(A,\mathcal R^\omega)$ -- cf.\ Problem~\ref{q:RUCT}), agreement in the sense of \eqref{eq:pairing} might not imply that $[\gamma(\tau)]=[\tau]\circ\kappa$ for all $\tau \in T(B)$.

The following condition may then be the right compatibility condition between $\kappa$ and $\gamma$, in the absence of the UCT:
\begin{equation}
    [\gamma(\tau)]=[\tau]\circ\kappa,\quad \tau \in T(B).
\end{equation}
It is plausible that this condition might be sufficient to obtain classification without the UCT, at least in the unique trace setting.

Some partial progress was recently announced by the first-named author in the preprint \cite{Schafhauser24}.  Note that whatever the `correct' $KK$/trace compatibility condition might be, it should be satisfied by a pair $\kappa = [\phi]$ and $\gamma = T(\phi)$ for a unital $^*$-homomorphism $\phi \colon A \rightarrow B$. By \cite[Theorem~A]{Schafhauser24}, if the isomorphisms $\kappa$ and $\gamma$ are induced by a $^*$-homomorphism in this sense, then $A \cong B$.

\section{$K_1$-injectivity and $KK$-uniqueness}
\label{sec:K1injectivity}

By definition, $K$-theory is built out of equivalence classes of projections and unitaries at all matrix levels and so is stable: $K_*(A)\cong K_*(A\otimes\mathcal K)$. Moreover, working in the unital case, 
\begin{equation}
\pi_n(U_\infty(A))\cong \begin{cases}
    K_1(A),&n\text{ even},\\K_0(A),&n\text{ odd},
\end{cases}
\end{equation}
via Bott periodicity (here, $U_\infty(A)$ is the union of the unitary groups of all matrix amplifications of $A$). Non-stable $K$-theory is concerned with the relationship between $K$-theory and projections and unitaries in the original {$C^*$}-algebra, through properties such as cancellation of projections and examination of (in the unital case) the homotopy groups $\pi_n(U(A))$.\footnote{The latter was formalised by Thomsen in \cite{Thomsen:KTheory}. This makes non-stable $K$-theory into a homology theory on the category of {$C^*$}-algebras (albeit that $\pi_0(U(A))$ need not be abelian), and includes how to handle the non-unital case.} 

It is natural to ask under what circumstances does the algebra $A$ contain enough relevant information about $K$-theory without the need for matrix amplification. A unital {$C^*$}-algebra $A$ is called \emph{$K_1$-injective} or \emph{$K_1$-surjective} when the natural map $\pi_0(U(A))\to K_1(A)$ is injective or surjective, respectively, and \emph{$K_1$-bijective} when it is both $K_1$-injective and $K_1$-surjective. There are numerous important classes of $K_1$-bijective {$C^*$}-algebras, including all simple purely infinite {$C^*$}-algebras -- as shown by Cuntz in \cite[Theorem 1.9]{Cuntz:Ann} as part of his calculation of $K$-theory for Cuntz algebras -- and all classifiable {$C^*$}-algebras (via Theorem \ref{Jiang} below). More generally, Thomsen introduced the notion of \emph{$K$-stability}, to describe the situation when all the natural maps $\pi_n(U_m(A))\to \pi_n(U_{m+1}(A))$ for $n\geq 0$ and $m\geq 1$ are isomorphisms  and gave a number of examples. In particular, Thomsen pioneered the use of tensorial absorption to obtain $K$-stability, showing that algebras which have a simple infinite-dimensional AF algebra or a Cuntz algebra $\mathcal O_n$ as a tensor factor are $K$-stable. This set the scene for the following result of Jiang.

\begin{theorem}[{Jiang; \cite{Jiang:arXiv} -- cf.\ \cite[Section 4.2]{CGSTW} and \cite{Hua:arXiv}}]\label{Jiang}
All $\Z$-stable {$C^*$}-al\-ge\-bras are $K$-stable and hence $K_1$-bijective.
\end{theorem}

In relation to this result, Thiel raises the following question.
We view it as particularly interesting in the non-nuclear case (where it is known that pureness does not imply $\mathcal Z$-stability).
Pureness is defined in Section~\ref{Sec:CuReg}, and see Section~\ref{sec:NonsimpleTW} for more discussion about pureness in the non-simple case.

\begin{question}
    Is every pure {$C^*$}-algebra $K_1$-bijective (or even $K$-stable)?
\end{question}

As Thomsen noted in his calculations in \cite[Section 4]{Thomsen:KTheory}, non-stable $K$-theory is `very sensitive’ to whether a unital {$C^*$}-algebra $A$ is homotopy equivalent to an abelian {$C^*$}-algebra, as in this case, $\pi_n(U(A))=0$ for $n\geq 2$. Thus no commutative {$C^*$}-algebra with non-trivial $K$-theory can be $K$-stable, and in particular, a (non-zero) unital commutative C$^*$-algebra cannot be $K$-stable.  Villadsen's 
  breakthrough constructions (\cite{Villadsen:JFA,Villadsen:JAMS}) showed how to build inductive limits to replicate certain commutative phenomena  in simple {$C^*$}-algebras (such as perforation in $K_0$ and high stable rank). Through these techniques, he gave examples of simple non-elementary non-$K$-stable algebras; indeed, for any $n_0\in\mathbb N$, Villadsen constructs a simple separable unital nuclear {$C^*$}-algebra $A$ such that for all $k\geq 0$, the maps $\pi_k(U_n(A))\to \pi_k(U_{n+1}(A))$ fail to be surjective for $n\leq n_0$ and are bijective for $n>n_0$ (\cite[Theorem 12]{Villadsen:Crelle}).\footnote{We thank Andrew Toms for bringing this reference to our attention.} In particular, simple {$C^*$}-algebras need not be $K_1$-surjective (this sort of question was asked at least as far back as \cite[End of Section~3]{Thomsen:RIMS}). To the best of our knowledge, the corresponding question for $K_1$-injectivity is open.

\begin{question}
    Are all unital simple {$C^*$}-algebras $K_1$-injective?  
\end{question}

Beyond the simple setting, all of stable rank one, real rank zero, and infiniteness give rise to some $K$-stability phenomena.
Rieffel's introduction of stable rank was motivated by non-stable $K$-theory: {$C^*$}-algebras with stable rank one have cancellation (\cite{Rieffel:PLMS}; see \cite[Proposition V.3.1.24]{Blackadar-encyclopedia}) and are $K_1$-injective (\cite[Theorem~2.10]{Rieffel:JOT}).

For general {$C^*$}-algebras of real rank zero, $K$-theoretic regularity first appeared in Zhang's work \cite{Zhang:JOT}, which shows that a unital C$^*$-algebra $A$ of real rank zero is $K_0$-surjective, i.e. the classes of projections from $A$ generate $K_0(A)$.\footnote{Let $A$ be a real rank zero C$^*$-algebra.  Zhang's main result in \cite{Zhang:JOT} is that the dimension range $\mathcal D(A)$ (the Murray--von Neumann equivalence classes of projections in $A$ with a partially defined order) has the Riesz decomposition property: if $x\leq y+z$, then $x=x_1+x_2$ for some $x_1\leq y$ and $x_2\leq z$.  The $K_0$-surjectivity statement is not explicitly given in the paper. It is obtained by applying the Riesz decomposition property to $M_n(A)$ (which inherits real rank zero), so that an induction argument shows that any projection $p\in M_n(A)$ is equivalent to a sum of orthogonal projections from $A$.  Note that Zhang works with property (FS) -- self-adjoints of finite spectrum are dense in the self-adjoints -- which was shortly afterwards shown to be equivalent to real rank zero in \cite{BrownPedersen:JFA}.} Subsequently, Lin established $K_1$-injectivity (\cite[Lemma 2.1]{Lin:PJM}, heavily using his earlier work \cite{Lin:JFA} on the somewhat unfortunately named weak property (FU)).  Substantial work has been put into the question of whether unital real rank zero {$C^*$}-algebras are $K_1$-surjective (which appears to be an unpublished conjecture of Shuang Zhang; see \cite[Section 3]{AGOR:PJM}).  As part of their work aiming to unify results for simple purely infinite {$C^*$}-algebras with those of stable rank one (\cite{BrownPedersen:Crelle,BrownPedersen:JFA14}), Larry Brown and Pedersen introduce the notion of \emph{weak cancellation}: $A$ has weak cancellation if for all projections $p,q\in A$ which generate the same closed ideal $I$, one has $[p]_0=[q]_0$ in $K_0(I)$ $\implies p\sim q$.  This property spiritually goes back to Cuntz's work on the $K$-theory of simple purely infinite {$C^*$}-algebras, which have weak cancellation  (\cite[Section 1]{Cuntz:Ann}); likewise, stable rank one gives rise to cancellation, and so certainly to weak cancellation.  This is related to the notion of separativity\footnote{A {$C^*$}-algebra $A$ is \emph{separative} if for projections $p,q\in A\otimes\mathcal K$
\begin{equation}p\oplus p\sim p\oplus q \sim q\oplus q\implies p\sim q.
\end{equation}}
investigated in \cite{AGOP:IJM,AGOR:PJM}: $A$ is separative if and only if $A$ has stable weak cancellation, i.e.\ $A\otimes\mathcal K$ has weak cancellation (\cite[Paragraph 2.3]{BrownPedersen:JFA14}, \cite[Section 3]{AGOR:PJM}).

\begin{question}\label{q:K1SurjectRR0}
\begin{enumerate}[(1)]
    \item Do real rank zero {$C^*$}-algebras have weak cancellation?\label{q:K1SurjectRR0.1}
    \item Are real rank zero {$C^*$}-algebras $K_1$-surjective?\label{q:K1SurjectRR0.2}
\end{enumerate}
\end{question}

The connection between these problems was given by Ara, Goodearl, O'Meara and Raphael, who showed that separative {$C^*$}-algebras of real rank zero are $K_1$-surjective (\cite[Theorem 3.1]{AGOR:PJM}) and asked Problem \ref{q:K1SurjectRR0}(\ref{q:K1SurjectRR0.1}) (as \cite[Question 3.2]{AGOR:PJM}, in the equivalent form of whether real rank zero algebras are separative). In slightly earlier work, Ara, Goodearl, O'Meara, and Pardo  showed that for separative unital real rank zero {$C^*$}-algebras, finiteness implies stable finiteness (\cite[Theorem 7.6]{AGOP:IJM}) (and in the simple case, this also implies stable rank one).  As is known to experts, including the referee (but wasn't to us), this means that if simple {$C^*$}-algebras of real rank zero have weak cancellation, then R\o{}rdam's dichotomy problem (Problem~\ref{q:RR0dichotomy}) has a positive solution.\footnote{If R\o{}rdam's dichotomy fails, then there is an infinite simple unital {$C^*$}-algebra $A$ of real rank zero with a finite projection $p$.  Then $pAp$ has real rank zero, so (assuming weak cancellation) $pAp \otimes \mathbb K \cong A\otimes \mathbb K$ is stably finite, a contradiction.}  

For a unital properly infinite\footnote{Recall that a unital {$C^*$}-algebra $B$ is \emph{properly infinite} if there exist orthogonal projections $p,q\in B$ with $p\sim q\sim 1_B$. Equivalently, there is a unital embedding of $\mathcal O_\infty$, or a unital embedding of some (or all) of the Cuntz--Toeplitz algebras $\mathcal T_n$ for $n\geq 2$.} 
$C^*$-algebra $A$, any element of $K_0(A)$ is realised as the class of a properly infinite full projection in $A$, and two properly infinite full projections in $A$ which agree in $K$-theory are Murray--von Neumann equivalent. Every unital properly infinite {$C^*$}-algebra is $K_1$-surjective.\footnote{Both of these results are due to Cuntz in \cite{Cuntz:Ann}, though they are not explicitly stated there; see \cite[Exercises~4.6,~4.9 and~8.9]{Rordam:KBook}.} Prominently missing is an answer to the following question. 

\begin{question}\label{PropInfiniteK1Injective}
Are all properly infinite unital {$C^*$}-algebras $K_1$-injective?
\end{question}

This question seems to have been first formally asked by Blanchard, Rohde, and R{\o}rdam in \cite{BRR:JNCG},\footnote{As noted there, it has been implicitly around for much longer, with $K_1$-injectivity of properly infinite {$C^*$}-algebras appearing as a hypothesis in \cite{Rordam:Crelle} to obtain certain uniqueness theorems.} which develops a number of equivalent characterisations of $K_1$-injectivity for a unital properly infinite {$C^*$}-algebra $A$ (such as that for projections $p,q\in A$ with $p$, $q$, and their respective complements full and properly infinite,\footnote{In his manuscript \cite[Page 488]{Kirchberg:Book}, Kirchberg called such projections \emph{splitting}.} Murray--von Neumann equivalence and homotopy coincide).
Blanchard, Rohde, and R{\o}rdam further show that it suffices to answer Problem~\ref{PropInfiniteK1Injective} in certain particular cases, such as the unital full free product $\mathcal O_\infty*\mathcal O_\infty$ (see \cite[Section~5]{BRR:JNCG}). Many natural examples of properly infinite {$C^*$}-algebras are known to be $K_1$-injective (such as coronas of stable {$C^*$}-algebras, recorded in \cite[Proposition 4.9]{GabeRuiz:GMJ}, using earlier lifting results). 
Kirchberg's interest in Problem~\ref{PropInfiniteK1Injective} prompted his \emph{squeezing property} (see \cite[Definition~4.2.14, Proposition~4.2.15 and Question~2.5.20]{Kirchberg:Book}) as a potential route to $K_1$-injectivity of properly infinite {$C^*$}-algebras.

One particular point of relevance of Problem \ref{PropInfiniteK1Injective} to classification is through the connection to uniqueness theorems for $KK$-theory via  Paschke duality. We'll set up the uniqueness problem first, then discuss the connection to $K_1$-injectivity.
Let $A$ and $B$ be {$C^*$}-algebras with $A$ separable and $B$ $\sigma$-unital and stable.  
The Cuntz pair picture of $KK(A,B)$ consists of  homotopy classes of Cuntz pairs $(\phi,\psi)\colon A\rightrightarrows \mathcal M(B)\rhd B$; i.e.\ pairs $\phi,\psi \colon A \rightarrow \mathcal M(B)$ of $^*$-homomorphisms such that $\phi(a)-\psi(a)\in B$ for all $a\in A$. The class of $(\phi,\psi)$ is thought of as a formal difference  $\phi-\psi$. If $\phi\colon A\to\mathcal M(B)$ is an absorbing representation,\footnote{There are many equivalent formulations of absorption.  The original, inspired by Voiculescu's theorem is that for any $\theta\colon A\to \mathcal M(B)$, the direct sum $\phi\oplus\theta$ (obtained by adding $\phi$ and $\theta$ diagonally in $M_2(\mathcal M(B))$ and using stability of $B$ to obtain a so-called `standard isomorphism' $M_2(\mathcal M(B))\cong \mathcal M(B)$, i.e.\ one induced by an isomorphism $M_2(\mathcal K)\cong\mathcal K$), is approximately unitarily equivalent to $\phi$ modulo $B$; i.e.\ there exists $(u_n)\in\mathcal M(B)$ so that $u_n(\phi(a)\oplus \theta(a))u_n^*-\phi(a)\in B$ and $\|u_n(\phi(a)\oplus \theta(a))u_n^*-\phi(a)\|\to 0$ for all $a\in A$.\label{def-absorb}} then any class $\kappa$ in $KK(A,B)$ can be realised as a formal difference with $\phi$, i.e.\  $\kappa$ is represented by some Cuntz pair of the form $(\phi,\psi)$ (see \cite[Theorem~5.4(i)]{CGSTW}, for example).
We view this as a `\emph{$KK$-existence}’ statement; but how unique is such a $\psi$?  Since $\psi$ will necessarily be absorbing, this amounts to asking what happens when $[\phi,\psi]=0$.

\begin{question}[$KK$-uniqueness ({\cite[Question 5.17]{CGSTW}})]\label{KKUniqueness}
    Suppose $A$ is a separable {$C^*$}-algebra, $B$ is a $\sigma$-unital stable {$C^*$}-algebra, and $\phi, \psi \colon A \rightarrow \mathcal M(B)$ are absorbing representations such that $\phi(a) - \psi(a) \in B$ for all $a \in A$.
    If $[\phi, \psi] = 0 \in KK(A, B)$, is there a continuous family of unitaries $(u_t)_{t\in[1,\infty)} \subseteq U(B + \mathbb C1_{\mathcal M(B)})$ such that 
    \begin{equation}
        \lim_{t \rightarrow \infty} \|u_t\phi(a)u_t^* - \psi(a)\| = 0
    \end{equation}
    for all $a \in A$?
\end{question}

 Given a pair of absorbing representations $(\phi,\psi)\colon A\rightrightarrows\mathcal M(B)\rhd B$ which agree modulo $B$, absorption gives a continuous path $(u_t)_{t\in[1,\infty)}$ of unitaries in $\mathcal M(B)$ with $u_t\phi(a)u_t^*-\psi(a)\in B$ and $\|u_t\phi(a)u_t^*-\psi(a)\|\to 0$ for all $a\in A$. The point in  Problem \ref{KKUniqueness} is to be able to find these unitaries in the minimal unitisation of $B$, assuming $[\phi,\psi]=0$. This gives Dadarlat and Eilers' relation of \emph{proper asymptotic unitary equivalence} between $\phi$ and $\psi$ (\cite{DadarlatEilers:KT}), which is very powerful in applications as it ensures that if $\phi$ and $\psi$ were obtained from maps into some {$C^*$}-algebra $E$ containing $B$ as an ideal (followed by the canonical map $E\to \mathcal M(B)$), then the asymptotic unitary equivalence also can be found in $E$. This is how, with Carri\'on and Gabe, we use such $KK$-uniqueness type statements in \cite{CGSTW}.

It is a slightly more convenient to explain how to go from Problem \ref{PropInfiniteK1Injective} to a version of Problem \ref{KKUniqueness} for \emph{unitally absorbing} representations.  Even when $A$ is unital, an absorbing representation $\phi\colon A \to \mathcal M(B)$ can never be unital (a unital map cannot absorb $0$); unital absorption is defined analogously to absorption as in footnote \ref{def-absorb}, working now with unital $\theta$. The connection between these notions is that $\phi\colon A\to \mathcal M(B)$ is absorbing if and only if the forced unitisation $\phi^\dagger\colon A^\dagger\to \mathcal M(B)$ is unitally absorbing; see \cite{Thomsen:PAMS}.\footnote{Given a unitally absorbing $\phi\colon A\to \mathcal M(B)$, where $A$ is unital and $B$ is $\sigma$-unital and stable, one gets an absorbing representation back as $\phi\oplus 0$}  The version of the $KK$-uniqueness question for unitally absorbing maps is exactly the same as Problem \ref{KKUniqueness}, with $A$ unital, and $\phi,\psi$ unitally absorbing.  Moreover, by taking a forced unitisation it suffices to prove the unitally absorbing version (see the proof of \cite[Lemma~5.15(i)]{CGSTW}).

Given a unitally absorbing representation $\phi\colon A\to \mathcal M(B)$ with $A$ separable and unital, and $B$ $\sigma$-unital and stable,  Paschke duality (\cite[Theorem 3.2]{Thomsen:PAMS}) gives an isomorphism \begin{equation}
KK(A,B)\cong K_1(\mathcal Q(B)\cap \overline{\phi}(A)')
\end{equation}
where $\mathcal Q(B)\coloneqq \mathcal M(B)/B$ is the corona of $B$ and $\overline{\phi}\colon A\to \mathcal Q(B)$ is induced by $\phi$.\footnote{This is not how Paschke duality is usually stated. In \cite[Section 3]{Thomsen:PAMS}, Thomsen sets this up for an absorbing map $\phi_0\colon A\to \mathcal M(B)$ (as ever, here  $A$ is separable and $B$ is $\sigma$-unital and stable).  His result gives the duality $KK(A,B)\cong K_1((\mathcal Q(B)\cap \overline{\phi_0}(A)')/\mathrm{Ann}(\overline{\phi_0}(A)))$.  Here, $\mathrm{Ann}(\overline{\phi_0}(A))=\{x\in \mathcal Q(B):x\overline{\phi_0}(A)=\overline{\phi_0}(A)x=0\}$ is the annihilator of $\overline{\phi_0}(A)$ in $\mathcal Q(B)$.  Applying this to the absorbing map $\phi_0=\phi\oplus 0$ associated to the unitally absorbing map $\phi$, one has $\mathcal Q(B)\cap\overline{\phi_0}(A)'\cong (\mathcal Q(B)\cap \overline{\phi}(A)')\oplus \mathcal Q(B)$, with the annihilator making up the second direct summand.} Tracking through the Paschke duality, a pair $[\phi,\psi]$ in $KK(A,B)$ with $\phi$ and $\psi$ unitally absorbing is mapped to $[\overline{u}_1]_1$ in $K_1(\mathcal Q(B)\cap \overline{\phi}(A)')$, where $(u_t)_{t\geq 1} \subseteq \mathcal M(B)$ asymptotically conjugates $\phi$ onto $\psi$ modulo $B$,\footnote{The fact that one can upgrade the approximate unitary equivalence modulo $B$ that one gets from absorption to asymptotic unitary equivalence modulo $B$ goes back at least to \cite{DadarlatEilers:KT}.} and $\overline{u}_t$ is the image of this path in $U(\mathcal Q(B))$.  Dadarlat and Eilers' argument from \cite[Theorem~3.12]{DadarlatEilers:KT} shows that $\phi$ and $\psi$ are properly asymptotically unitarily equivalent when $\overline{u}_1$ is homotopic to $1$ in $U(\mathcal Q(B)\cap \overline{\phi}(A)')$ (this is set out abstractly in \cite[Lemma~5.16]{CGSTW}).  For this reason, $K_1$-injectivity of $\mathcal Q(B)\cap \overline{\phi}(A)'$, gives rise to $KK$-uniqueness. In this way, Dadarlat and Eilers obtained a positive answer to the $KK$-uniqueness problem when $B=\mathcal K$ from Paschke's earlier $K_1$-injectivity of relative commutants relevant to this case (\cite[Lemma 3]{Paschke:PJM}).  Since $\mathcal Q(B)\cap\overline{\phi}(A)'$ is properly infinite (as a consequence of $\phi$ and $\psi$ being unitally absorbing), a positive answer to the $K_1$-injectivity problem (Problem \ref{PropInfiniteK1Injective}) gives a positive answer to the $KK$-uniqueness problem (Problem~\ref{KKUniqueness}). 

 As we don't have a view as to whether $KK$-uniqueness should hold in full generality, it may be worthwhile to consider absorbing representations associated to stabilised Villadsen algebras or other similar constructions.  


Recently, Loreaux and Ng have obtained positive results for the $KK$-uniqueness problem for some codomains $B$ with strict comparison, and unital separable nuclear (and sometimes also simple) $A$ through $K_1$-injectivity of the associated Paschke duals (see \cite[Theorem 2.4]{LoreauxNg:IEOT} for purely infinite $B$, and \cite[Lemma 2.9]{LoreauxNg:IEOT} and the preprint \cite{LNS:arXiv22}, authored together with Sutradhar, for various results with $B$ stably finite and with finitely many extremal traces).
In our abstract approach to classification with Carri\'o{}n and Gabe, we tensored on a copy of the Jiang--Su algebra to the Paschke dual in order to be able to use Theorem~\ref{Jiang} (due to Jiang) to get $K_1$-injectivity. This gives a `$\Z$-stable $KK$-uniqueness theorem' (\cite[Theorem~5.15]{CGSTW}), which has subsequently been put in a much cleaner framework by Farah and Szab\'o (\cite[Theorem~5.5]{FarahSzabo:arXiv}).

\subsection{Addendum: April 2026} We do now have a view as to $KK$-uniqueness: in January 2026, G\'abor Szab\'o gave a positive answer to the $KK$-uniqueness problem (Problem \ref{KKUniqueness}) and moreover, obtains general equivariant $KK$-uniqueness theorems which will be very broadly applicable (\cite{Szabo:preprint}).  This goes through $K_1$-injectivity of relative commutants but uses particular features of multiplier algebras (going back to work of Cuntz and Higson (\cite{CuntzHigson})), and so it doesn't touch on Problem \ref{PropInfiniteK1Injective}.

\section{Embeddings of $\mathcal Z$}\label{Sec:Zembed}

Most, if not all, modern {$C^*$}-classification results go via very general classification results for embeddings $A\hookrightarrow B$ up to approximate unitary equivalence together with complementary existence theorems describing the range of the invariant.  This strategy dates back at least to R\o{}rdam's work \cite{Rordam:JFA}, identifying a concrete class of Kirchberg algebras whose embeddings can be classified by $KL$ (introduced by R\o{}rdam for this purpose).  Kirchberg demonstrated that such classification of morphism results can hold in extreme generality with most of the hypotheses lying on the morphisms (retaining hypotheses on the domain $A$ and codomain $B$ only when it is essential to do so).  A fantastic example is given by the following underlying classification result Gabe obtains in his approach to the Kirchberg--Phillips theorem.

\begin{theorem}[{\cite{Gabe:MAMS}, cf.\ \cite{Kirchberg:Book}}]\label{OInftyStableMaps}
    Let $A$ be a unital separable exact {$C^*$}-algebra and $B$ a unital properly infinite {$C^*$}-algebra.  Then unital full nuclear $\mathcal O_\infty$-stable\footnote{A unital map $\theta\colon A\to B$ is \emph{$\mathcal O_\infty$-stable} if there is a unital embedding of $\mathcal O_\infty$ into the relative commutant $B_\omega\cap \theta(A)'$.  Similar definitions using Kirchberg's relative commutant algebra $F(B,\theta(A))$ from \cite{Kirchberg:Abel} are used in the non-unital setting.} morphisms $\theta\colon A \rightarrow B$, up to approximate unitary equivalence, are canonically in bijection with the elements of $KL_{\mathrm{nuc}}(A,B)$ that are compatible with the unit of $K_0$, via $\theta\mapsto [\theta]_{KL}$.
\end{theorem}

 The point is that the domain and codomain hypotheses are obviously necessary: separability of $A$ is crucial to appeal to intertwining arguments, the exactness of $A$ is equivalent to the existence of at least one full nuclear map out of $A$, and the proper infiniteness of $B$ is necessary for the existence of a unital $\mathcal O_\infty$-stable map into $B$ (and also evidently sufficient when $A\coloneqq \mathcal O_\infty$).  The simplicity hypothesis (fullness), nuclearity condition, and even the tensorial absorption condition needed for classification have all been transferred to the level of the maps.

A striking consequence of this one-sided Kirchberg--Phillips theorem is that for any unital {$C^*$}-algebra $B$, any two unital embeddings $\mathcal O_\infty \rightarrow B$ are approximately unitarily equivalent (see \cite[Proposition~4.4]{Gabe:Crelle}, which records the stronger uniqueness up to asymptotic unitary equivalence; see also \cite[Theorem~7.4.4]{Bouwen:Masters} for an alternative, more direct, proof).\footnote{The analogous result for $\mathcal O_2$ also holds by \cite[Theorem~3.6]{Rordam:Crelle} (cf.\ \cite[Lemma~5.4]{BlanchardKirchberg:JFA}).}  
Indeed, this follows since $\mathcal O_\infty$ satisfies the UCT and there is only one unital map out of $\mathcal O_\infty$ at the level of $K$-theory. This result (and its $\mathcal O_2$-counterpart) ensures that Gabe's natural definition of $\mathcal O_2$-stability and $\mathcal O_\infty$-stability characterise the unital maps which are approximately unitarily equivalent to maps factoring through $\mathcal O_2$-stable and $\mathcal O_\infty$-stable {$C^*$}-algebras (see \cite[Corollary~4.5]{Gabe:Crelle}).  

The lack of a counterpart result for $\mathcal Z$ makes it much less clear how far one-sided classification results can be pushed in the stable finite setting.  In the long term, what kind of regularity hypothesis on the codomain are necessary for classification theorems? 
While it is possible to define the notion of $\Z$-stable morphisms, without such a uniqueness result for embeddings of $\Z$, it is unclear whether such maps should be (under very general conditions) approximately unitary equivalent to maps factoring through $\Z$-stable {$C^*$}-algebras. Accordingly, we regard the following as an important test question.

\begin{question}\label{unique-z}
    Let $B$ be a unital {$C^*$}-algebra.  Are any two unital embeddings $\mathcal Z \rightarrow B$ approximately unitarily equivalent, or even asymptotically unitary equivalent?  If this fails, what are the minimal hypotheses on $B$ needed to obtain such a result?
\end{question}

If $B$ itself is $\mathcal Z$-stable, Problem~\ref{unique-z} has a positive answer by the general theory of strongly self-absorbing algebras.  Some other cases can be seen from Robert's classification theorem in \cite{Robert:AIM} -- see \cite[Proposition~6.3.1]{Robert:AIM} in particular. As Robert's work gives rise to a classification of maps from $\Z$ into {$C^*$}-algebras of stable rank one by Cuntz semigroup data (Robert's $\widetilde{\Cu}$), it is very natural to consider Problem~\ref{unique-z} when $B$ is a Villadsen algebra of the first type (\cite{Villadsen:JFA,Toms:Ann,TomsWinter:JFA}), which have stable rank one.

\begin{question}[{cf. \cite[Questions 16.1, 16.2, 16.3]{GardellaPerera}}]\label{Q:CuVilToms}
What is the Cuntz semigroup of the Villadsen/Toms counterexamples to classification with stable rank one, and in particular, do there exist two distinct Cuntz semigroup morphisms from $\Cu(\Z)$ into these Cuntz semigroups?
\end{question}

We note (with thanks to Jamie Gabe) that the version of Problem~\ref{unique-z} with $M_{2^\infty}$ in place of $\mathcal Z$ has a negative answer (with obstructions arising from torsion in $K_0$).  For example, consider a unital Kirchberg algebra $B$ with $K_0(B)=\mathbb Z[1/2]/\mathbb Z$ with $[1_B]_0=0$, and $K_1(B)=0$. Then the zero map on $K_0$ and the quotient map $K_0(M_{2^\infty})\cong \mathbb Z[1/2]\twoheadrightarrow \mathbb Z[1/2]/\mathbb Z\cong K_0(B)$ are distinct unit-preserving morphisms. These are realised by unital $^*$-homomorphisms, which cannot be approximately unitarily equivalent.

The problem of when unital embeddings of $\mathcal Z$ exist is also of interest.  It is known that receiving a unital map from $\mathcal Z$ does entail some amount of regularity, as it gives Cuntz semigroup divisibility conditions on the unit (see \cite[Lemma~4.2]{Rordam:IJM}). This sort of obstruction underpins the construction of a unital simple infinite-dimensional {$C^*$}-algebra $A$ such that there is no unital embedding $\mathcal Z \rightarrow A$ (\cite{DHTW:MRL}).  But any regularity gained from an embedding of $\Z$ is certainly far from the full force of strict comparison.  Indeed, the Villadsen-type counterexamples of Toms (\cite{Toms:Ann}) all contain unital copies of $\Z$.  We hope answering the following question will shed light on the appropriate condition to replace proper infiniteness in a stably finite version of Theorem~\ref{OInftyStableMaps}.

\begin{question}
Characterise those  unital (simple) {$C^*$}-algebras $B$ for which there exists an embedding $\Z\hookrightarrow B$.
\end{question}

However, it is not clear how much regularity is forced by a unital embedding of a well-behaved {$C^*$}-algebra.  For example, the following problem is open.

\begin{question}\label{M2InftyInfinite}
 Does there exist a unital embedding of $M_{2^\infty}$ into a simple infinite C$^*$-algebra $B$ such that the image contains finite projections?
\end{question}

When R\o{}rdam constructed a unital simple infinite {$C^*$}-algebra $B$ with a finite projection in \cite{Rordam:Acta}, he arranged for $1_B$ to decompose into two equivalent finite projections, giving rise to an embedding $M_2\to B$ containing a finite projection in its image.  By modifying the construction, for each $n$, one can obtain a  unital simple infinite $B$ with a unital embedding $M_n \to B$ such that $1_B$ is the only infinite projection in the image of the embedding, but it seems challenging to combine these to answer Problem \ref{M2InftyInfinite}.  An analogous question with the Jiang--Su algebra in place of the CAR algebra (asking for such an embedding whose range contains finite positive elements in the sense of the Cuntz comparison) is just as natural; we would expect it to have the same answer as Problem \ref{M2InftyInfinite}.

\section{Classification for non-simple {$C^*$}-algebras}\label{sec:nonsimpleclass}

Whereas a separably acting von Neumann algebra decomposes in a nice way as a direct integral of simple von Neumann algebras (factors), there is no corresponding statement that allows results for simple {$C^*$}-algebras to be transferred to the non-simple case.  Nevertheless, in a major tour de force, Kirchberg was able to extend the Kirchberg--Phillips classification theorem to the non-simple case, classifying all separable nuclear $\mathcal O_\infty$-stable {$C^*$}-algebras using ideal-related $KK$-theory.

\begin{theorem}[{\cite{Kirchberg:Book}; see also a new proof by Gabe in \cite{Gabe:MAMS}}]\label{OInftyStableClass}
    Let $A$ and $B$ be unital separable nuclear $\mathcal O_\infty$-stable {$C^*$}-algebras.
    Then $A \cong B$ if and only if they are unitally ideal-related $KK$-equivalent; i.e.\ $X\coloneqq \mathrm{Prim}(A)\cong \mathrm{Prim}(B)$, and this topological isomorphism is accompanied by a unital $KK_X$-equivalence between $A$ and $B$.
\end{theorem}

Translating Theorem \ref{OInftyStableClass} into an ideal-related $K$-theory classification theorem (under suitable UCT hypotheses) is primarily a problem in homological algebra and has seen considerable attention (see \cite{Meyer19,BentmannMeyer:DM}, for example, together with work on the relevant UCT conditions in \cite{MeyerNest:CJM,MeyerNest:MJM}). 

In the important special case of $\mathcal O_2$-stable algebras, ideal-related $KK$-theory reduces to the primitive ideal space, or equivalently, the ideal lattice.  The version of this we state uses the framework of Gabe's new approach to this result from \cite{Gabe:Crelle}.

\begin{theorem}[{\cite{Kirchberg:Book,Gabe:Crelle}}]\label{O2StableClass}
    Let $A,B$ be unital separable nuclear $\mathcal O_2$-stable {$C^*$}-algebras.  The following are equivalent:
    \begin{enumerate}
    \item $A\cong B$;
    \item $\mathrm{Prim}(A)$ is homeomorphic to $\mathrm{Prim}(B)$;
    \item $\Cu(A) \cong \Cu(B)$.\footnote{For purely infinite {$C^*$}-algebras, the Cuntz semigroup exactly encodes the ideal lattice.}
    \end{enumerate}
\end{theorem}

Kirchberg and Gabe's results in Theorems \ref{OInftyStableClass} and \ref{O2StableClass} suggest the possibility that classification could hold for all separable nuclear $\Z$-stable {$C^*$}-algebras.  But outside the $\mathcal O_\infty$-stable case (the case when both the ideal and quotient are purely infinite in Problem \ref{q:NonsimpleClass}), this seems challenging.  Indeed, even with just one ideal, while there are various classification results for extensions of classifiable algebras, and results for various inductive limits (particularly, those satisfying the ideal property; see \cite{GJL:TLMS} for example), there are no abstract results where tensorial absorption is the fundamental classification hypothesis.

\begin{challenge}\label{q:NonsimpleClass}
    Classify those (unital) separable nuclear $\mathcal Z$-stable {$C^*$}-algebras (potentially satisfying suitable UCT hypotheses) which have exactly one proper non-zero ideal using a suitable combination of ideal-related $KK$-theory and tracial data.
\end{challenge}

It may be natural to start with the case when both the ideal and the quotient are stably finite before tackling mixed cases where one of the ideal or quotient are stably finite and the other is purely infinite.  We hope that a solution to Problem~\ref{q:NonsimpleClass} could be the beginning of a greater challenge of classifying $\mathcal Z$-stable {$C^*$}-algebras with finitely many ideals and beyond.

For non-simple infinite {$C^*$}-algebras, Kirchberg's $\mathcal O_2$-stable classification -- while a huge achievement in its own right -- is an order of magnitude simpler than his final $\mathcal O_\infty$-stable classification as $\mathcal O_2$-stability kills $KK$-theory (in all quotients of all ideals).  In the stably finite setting, the Razak--Jacelon algebra $\mathcal W$, explored in \cite{Razak:CJM,Dean:CJM,Jacelon:JLMS}, is a natural analogue of $\mathcal O_2$. It is $KK$-contractible, simple, separable, nuclear,  $\Z$-stable with unique tracial state, and so tensoring by $\mathcal W$ annihilates $K$-theory, but leaves tracial data unchanged. Precisely, one works with the lower semicontinuous tracial weights (which come with a natural cancellative cone structure and topology making it a suitable potential classification invariant;\footnote{For a closed ideal $I\lhd A$, one gets a lower semicontinuous tracial weight on $A$ by \begin{equation} \tau_I(a)\coloneqq \begin{cases} 0,\quad &a \in I_+; \\ \infty,\quad &a\in A_+\setminus I_+.\end{cases}
\end{equation}  In this way, the cone of lower semicontinuous tracial weights recovers the ideal lattice of $A$.\label{FootnoteIdealTraces}} see \cite{ERS:AJM}). This sets up the following problem, posed by Leonel Robert. 


\begin{question}\label{q:Robert}
    For a separable nuclear {$C^*$}-algebra $A$, does the isomorphism class of $A\otimes\mathcal W\otimes \mathcal K$ depend only on the cone of lower semicontinuous tracial weights on $A$?
\end{question}

In the simple setting, Problem~\ref{q:Robert} has a positive answer: in the traceless case it is the uniqueness of $\mathcal O_2\otimes\mathcal K$, while for stably finite algebras it is the $KK$-contractible classification result of \cite{EGLN:JGP}, combined with \cite{CE:APDE}. For non-simple algebras, one should view Problem \ref{q:Robert} as proposing a tracial analogue of $\mathcal O_2$-stable classification.  In the traceless setting (in the sense of Definition \ref{Def:TracialWeight}), every lower semicontinuous tracial weight on $A$ comes from a closed ideal as in footnote \ref{FootnoteIdealTraces}.  Also, using the following result of R\o{}rdam, if $A$ is separable nuclear and traceless, then $A\otimes\mathcal W\otimes\mathcal K$ is $\mathcal O_2$-stable (since $\mathcal W$ is $\Z$-stable, and $\mathcal W\otimes\mathcal O_\infty\cong \mathcal O_2\otimes\mathcal K$).  In this way, Kirchberg's $\mathcal O_2$-stable classification theorem given in Theorem \ref{O2StableClass} gives a positive answer to Problem \ref{q:Robert} for traceless algebras. 

\begin{theorem}[R\o{}rdam, {\cite[Theorem 5.2]{Rordam:IJM}}]\label{Thm:RordamZInfinite}
Let $A$ be a separable nuclear $\Z$-stable {$C^*$}-algebra.  Then $A$ is $\mathcal O_\infty$-stable if and only if $A$ is traceless.
\end{theorem}

Note that the combination of Problems \ref{q:NonsimpleClass} and \ref{q:Robert} is open: for a separable nuclear {$C^*$}-algebra $A$ with one proper non-zero ideal, does the isomorphism class of $A\otimes\mathcal W\otimes \mathcal K$ depend only on the cone of lower semicontinuous extended traces on $A$?

Szab\'o has asked the following, predicting a strong structural result for the {$C^*$}-algebras in Problem~\ref{q:Robert} under an additional strong null-homotopy hypothesis.

\begin{question}\label{q:NullHomotopy1NCCW}
Let $A$ be a separable nuclear {$C^*$}-algebra, which is \emph{homotopic to zero in an ideal-preserving fashion}, i.e.\  there is a point-norm continuous family $(\theta_t)_{t\in[0,1]}$ of endomorphisms of $A$ connecting $\theta_0=0$ to $\theta_1=\mathrm{id}_{A}$, such that $\theta_t(I)\subseteq I$ for all ideals $I\lhd A$.  Must $A\otimes\mathcal W\otimes\mathcal K$ be an inductive limit of $1$-dimensional NCCW-complexes?
\end{question}

While this question arose in the context of nuclear dimension bounds for $\Z$-stable {$C^*$}-algebras (see a discussion towards the end of Section~\ref{sec:NonsimpleTW}), it is really a classification-type question. The basis for reasonably expecting a positive answer comes from the following striking result of Kirchberg and R\o{}rdam, which shows Problem~\ref{q:NullHomotopy1NCCW} has a positive answer in the $\mathcal O_\infty$-stable case.
\begin{theorem}[{\cite[Theorem~5.12]{KirchbergRordam:GAFA}}]\label{Thm:KRNullHomotopic}
Any separable nuclear $\mathcal O_\infty$-stable {$C^*$}-algebra which is homotopic to zero in an ideal-preserving fashion has an inductive limit description where the building blocks are homogeneous with 1-dimensional spectra.
\end{theorem}

We end this section by returning to classification results for algebras with non-trivial $K$-theory. Graph {$C^*$}-algebras give a particularly prominent class which mixes stably finite and purely infinite behaviour. Following early work by R\o rdam (\cite{Rordam:MA}) and Restorff (\cite{Restorff:Crelle}), $K$-theoretic invariants (e.g.\ `filtered'  $K$-theory, which is sometimes called `filtrated' $K$-theory)  which carefully account for the $K$-theory of ideal-quotients have been identified and studied. Decades of work in this direction culminated in the remarkable complete classification of graph algebras by Eilers, Restorff, Ruiz and S\o{}rensen (\cite{ERRS:Duke}).

\begin{theorem}[{\cite[Theorem 3.1]{ERRS:Duke}}]\label{GraphClassification}
Let $A$ and $B$ each be the stabilized graph {$C^*$}-algebra associated to a directed graph with finitely many vertices and countably many edges.
Then $A \cong B$ if and only if their ordered reduced filtered $K$-theories, relative to gauge-invariant prime ideals, are isomorphic. 
\end{theorem}

The following focused classification question arises directly from graph {$C^*$}-al\-ge\-bras; a more precise question (predicting what the characterisation should be) is found in \cite[Question~14.1]{ERRS:Duke}.

\begin{question}
    Give an abstract characterisation of {$C^*$}-algebras that are stably isomorphic to graph {$C^*$}-algebras of graphs with finitely many vertices.
\end{question}

\section{The primitive ideal space}

Whenever one has a classification by some invariant, it is natural to want to fully understand this associated data. For example, Elliott's classification of AF algebras (\cite{Elliott:JAlg}) is complemented by the Effros--Handelman--Shen theorem, which abstractly describes all dimension groups (\cite{EHS:AJM}). Thus, motivated by his result that the ideal lattice classifies unital separable nuclear $\mathcal O_2$-stable {$C^*$}-algebras (Theorem~\ref{O2StableClass}), a long-term endeavour of Kirchberg was to understand the range of this invariant.  This amounts to asking precisely which topological spaces can arise as $\mathrm{Prim}(A)$ when $A$ is a separable nuclear {$C^*$}-algebra.  Every primitive ideal space of a separable {$C^*$}-algebra is second countable, locally compact,\footnote{Kirchberg and his coauthors preferred to include Hausdorff in the definition of compactness and so his papers on this subject refer to the condition of `local quasicompactness'.} $T_0$, and point-complete (also known as sober).\footnote{A $T_0$ space $X$ is \emph{point-complete} or \emph{sober} if every prime closed subset $Y$ of $X$ is the closure of a point (where $Y$ is \emph{prime} if $Y$ cannot be written as the union of two closed proper subsets of $Y$).  Fleshing out the second sentence of the second paragraph of \cite{HarnischKirchberg}, the primitive ideal space $\mathrm{Prim}(A)$ of a separable {$C^*$}-algebra $A$ is point-complete as follows (which as the referee points out is similar in spirit to the proof that prime ideals in separable {$C^*$}-algebras are primitive). The primitive ideal space is a Baire space (as $\mathrm{Prim}(A)$ is continuous open image of the pure state space, which is Polish), and the topology has a countable basis. The same holds for all closed subsets of $\mathrm{Prim}(A)$.  Let $Y\subset\mathrm{Prim}(A)$ be prime and closed, and suppose $Y$ is not the closure of a point. Let $(U_n)_{n=1}^\infty$ be a countable basis for the topology on $Y$. The primeness condition shows that each $U_n$ is dense, so the Baire property gives $\bigcap_nU_n$ is dense in $Y$. But that intersection can not contain two distinct points as $\mathrm{Prim}(A)$, and hence $Y$, is $T_0$.} Harnisch and Kirchberg call spaces satisfying these conditions \emph{Dini spaces} as they are determined by their Dini functions (\cite[Section 6]{HarnischKirchberg}).  The question is whether these conditions fully characterise primitive ideal spaces of separable {$C^*$}-algebras.

\begin{question}[{\cite[Question~1.1]{KirchbergRordam:GAFA}}]\label{q:PrimSpaces}
    Is every Dini topological space the primitive ideal space of:
    \begin{enumerate}[(1)]
        \item\label{it:PrimSpaces.1}
        a separable nuclear {$C^*$}-algebra?
        \item\label{it:PrimSpaces.2}
        a separable {$C^*$}-algebra?
    \end{enumerate}
\end{question}

In \cite{KirchbergRordam:GAFA}, Kirchberg and R\o rdam established an important potential obstruction for primitive ideal spaces of exact {$C^*$}-algebras. This was subsequently described in purely topological terms by Harnisch and Kirchberg, and their main theorem states that it is the only obstruction even in the nuclear case (\cite{HarnischKirchberg}, written in 2005 but unpublished and only posted on the arXiv after Kirchberg's death).
More precisely, their result is that for a Dini topological space $X$, the following are equivalent.
\begin{enumerate}[(a)]
    \item There exists a separable nuclear {$C^*$}-algebra $A$ such that $X$ is homeomorphic to $\mathrm{Prim}(A)$.
    \item There exists a separable exact {$C^*$}-algebra $A$ such that $X$ is homeomorphic to $\mathrm{Prim}(A)$.
    \item\label{cond:PrimSpaces.3} There exists a locally compact separable completely metrisable space $P$ together with a continuous $\pi\colon P \to X$ which is pseudo-open and pseudo-epimorphic (as defined in \cite[Definition~1.3]{HarnischKirchberg}).
\end{enumerate}

The definitions of pseudo-open and pseudo-epimorphic are somewhat technical, but the combination of these is equivalent to natural lattice properties of the induced map between the collections of open sets; see \cite[Proposition~A.11]{HarnischKirchberg}.

This turns Problem~\ref{q:PrimSpaces}(\ref{it:PrimSpaces.1}) into a purely topological problem: does every 
Dini space automatically satisfy condition~\eqref{cond:PrimSpaces.3}?
If the answer is yes, then this of course solves Problem~\ref{q:PrimSpaces}(\ref{it:PrimSpaces.2}) as well.
If the answer is no, then it is interesting even to ask the following subquestion: is there a separable {$C^*$}-algebra $A$ such that $\mathrm{Prim}(A)$ is not the primitive ideal space of a nuclear (or equivalently, exact) {$C^*$}-algebra?

The recent public availability of the preprint \cite{HarnischKirchberg} should enable a fresh look at these problems by other researchers.

\section{Pure infiniteness and global Glimm halving}\label{S19}

The success of Cuntz's concept of pure infiniteness for simple {$C^*$}-algebras -- through the use of Kirchberg's Theorem \ref{Thm:Oinftystable} to access classification and the wealth of examples for which simple pure infiniteness can be verified -- together with Kirchberg's non-simple $\mathcal O_\infty$-stable classification results (Theorem \ref{OInftyStableClass}), led Kirchberg and R\o{}rdam to examine pure infiniteness outside the simple setting in \cite{KirchbergRordam:AJM,KirchbergRordam:AIM,KirchbergRordam:GAFA}. They sought a condition which holds for the class of $\mathcal O_\infty$-stable {$C^*$}-algebras and which implies tracelessness (see Definition \ref{Def:TracialWeight}). Ideally one would have an accessible characterisation of $\mathcal O_\infty$-stability for separable nuclear {$C^*$}-algebras, generalising Theorem \ref{Thm:Oinftystable}.  Since $\mathcal O_\infty$-stable {$C^*$}-algebras need not have many (or any) projections -- for example $C_0((0,1])\otimes\mathcal O_\infty$ -- this precludes the suggested definitions in the late '90s in terms of infinite projections in hereditary subalgebras.  Kirchberg and R\o{}rdam instead developed their theory at the level of positive elements using Cuntz comparison with the following definitions.

\begin{itemize}
\item A {$C^*$}-algebra $A$ is \emph{purely infinite} if it has no characters and for any pair of positive elements $a$ and $b$, if $a$ is in the ideal generated by $b$, then $a$ is Cuntz below $b$.
(This second condition is the appropriate formulation of strict comparison for traceless non-simple {$C^*$}-algebras, i.e. it characterises almost unperforation of the Cuntz semigroup in this setting).
\item A {$C^*$}-algebra $A$ is \emph{weakly purely infinite} if there exists $n\in\mathbb N$ such that for every non-zero positive $a\in A$, the element $a^{\oplus n}$ is properly infinite (in $M_n(A)$).\footnote{As with projections, a positive element $a$ in a {$C^*$}-algebra $B$ is \emph{properly infinite} if $a\oplus a\precsim a\oplus 0$ (in $M_2(B)$).
Weak pure infiniteness is equivalent to quasitracelessness and $n$-comparison for some $n$: that is, if $a,b\in (A\otimes \mathcal K)_+$ are such that $a$ is in the ideal generated by $b$ then $a \precsim b^{\oplus n}$.
The forward implication is \cite[Lemma~4.7]{KirchbergRordam:AIM}.
Conversely, given this condition, note that for any $a$, we have $a^{\oplus 2n}$ is in the ideal generated by $a \oplus 0^{\oplus(2n -1)}$ and therefore $a^{\oplus 2n} \precsim a^{\oplus n} \oplus 0^{\oplus n}$.}
\item \emph{Strong pure infiniteness} for a {$C^*$}-algebra $A$ has a more technical appearance. It requires that for any positive matrix $\begin{bmatrix} a&x \\ x^*&b\end{bmatrix} \in M_2(A)_+$, there is a sequence of diagonal matrices $\begin{bmatrix} d_n&0 \\ 0&e_n \end{bmatrix} \in M_2(A)$ such that
\begin{equation}
\begin{bmatrix} d_n&0 \\ 0&e_n \end{bmatrix}\begin{bmatrix} a&x\\x^*&b \end{bmatrix}\begin{bmatrix} d_n&0 \\ 0&e_n \end{bmatrix}^* \to \begin{bmatrix} a&0 \\ 0&b \end{bmatrix}.
\end{equation}
\end{itemize}

The terminology is justified; strongly purely infinite algebras are purely infinite (\cite[Proposition 5.4]{KirchbergRordam:AIM}), and purely infinite algebras are weakly purely infinite; in fact pure infiniteness is exactly weak pure infiniteness with $n=1$ (\cite[Theorem~4.16]{KirchbergRordam:AJM}). For simple {$C^*$}-algebras all these concepts agree, recapturing Cuntz's original notion, and in general, Kirchberg and R\o{}rdam develop pleasing permanence properties.  Weak and strong pure infiniteness relate to tracelessness and $\mathcal O_\infty$-stability respectively; in particular, strong pure infiniteness characterises $\mathcal O_\infty$-stability for separable nuclear {$C^*$}-algebras.

\begin{theorem}[{Kirchberg--R\o{}rdam; \cite{KirchbergRordam:AIM}}]\label{Thm:KR-PI}\hfill
\begin{enumerate}
    \item A separable {$C^*$}-algebra $A$ is weakly purely infinite if and only if it is quasitraceless. Weakly purely infinite {$C^*$}-algebras are quasitraceless.
    \item Every $\mathcal O_\infty$-stable {$C^*$}-algebra is strongly purely infinite.  The converse holds for separable nuclear {$C^*$}-algebras.\footnote{Kirchberg and R\o{}rdam obtained this for stable {$C^*$}-algebras as, at the time, it was not known that $\mathcal O_\infty$-stability is preserved under stable isomorphism for separable {$C^*$}-algebras. Now, using \cite[Section 3]{TomsWinter:TAMS} or \cite[Proposition 4.4 and Theorem 4.5]{Kirchberg:Abel}, it holds generally.}
\end{enumerate}

\end{theorem}
As a consequence, pure infiniteness meets Kirchberg and R\o{}rdam's two primary requirements, and whether it characterises $\mathcal O_\infty$-stability for separable nuclear {$C^*$}-algebras comes down to the (potential) difference between pure infiniteness and strong pure infiniteness. 

\begin{question}[{\cite[Question~9.5]{KirchbergRordam:AIM}}]\label{q:PIConditions}
    Do the following properties coincide for all (nuclear) {$C^*$}-algebras:
    weakly purely infinite, purely infinite, and strongly purely infinite?
\end{question}

Relating this back to our central theme of $\Z$-stability, in \cite{Rordam:IJM} R\o{}rdam showed that traceless exact $\Z$-stable {$C^*$}-algebras are strongly purely infinite to prove Theorem \ref{Thm:RordamZInfinite} (as an application of Theorem \ref{Thm:KR-PI}). The proof works outside the exact setting (for quasitraceless $\Z$-stable {$C^*$}-algebras).  Combining Theorems \ref{Thm:RordamZInfinite} and \ref{Thm:KR-PI}, for separable nuclear {$C^*$}-algebras the potential equivalence of pure infiniteness and strong pure infiniteness in Problem \ref{q:PIConditions}, is equivalent to whether traceless algebras with strict comparison are $\Z$-stable.

Kirchberg and R\o{}rdam gave a positive answer to Problem \ref{q:PIConditions} in the real rank zero case (\cite[Corollary~9.4]{KirchbergRordam:AIM}, though recall the caution that $\mathcal O_\infty$-stable {$C^*$}-algebras can be projectionless). With Blanchard, Kirchberg studied
the case of {$C^*$}-algebras with Hausdorff primitive ideal space (which can be viewed as continuous {$C^*$}-bundles with simple fibres).
They established that in this case, pure infiniteness and strong pure infiniteness coincide (\cite[Theorem~5.8]{BlanchardKirchberg:JFA}), along with weak pure infiniteness if the primitive ideal space is additionally finite-dimensional (\cite[Proposition~4.11 and Corollary~5.3]{BlanchardKirchberg:JFA}).

Intriguingly, Kirchberg and R{\o}rdam reduced the problem of whether weakly purely infinite {$C^*$}-algebras are purely infinite to a problem of Glimm halving. 
Based on the fact that any non-abelian von Neumann algebra admits an embedding of $M_2$ (this embedding need not be unital), Glimm's halving lemma says that every non-abelian {$C^*$}-algebra contains a non-zero element $x$ such that $x^2=0$.\footnote{This can be extracted from the proof of Glimm's \cite[Lemma 4]{Glimm:Ann}, or obtained from Kadison's transitivity theorem (via a short argument due to Kaplansky set out in \cite[2.12.21]{Dixmier:Book}). 
Here is a quick proof using modern technology (though arguably more involved than using transitivity).
Take an embedding of $M_2$ in $A^{**}$ and apply the Kaplansky density theorem for order zero maps (\cite[Lemma 1.1]{HKW:Adv}) to it to obtain a net of non-zero order zero maps $M_2\to\overline{aAa}$, and take $x$ to be any non-zero image of $e_{12}$ under one of these maps.} 
This result is especially useful in the simple case, where one can leverage the fact that $x$ must be full.  
To replicate this power in the non-simple case, a {$C^*$}-algebra $A$ is said to have the \emph{global Glimm halving property} (\cite[Definition 1.2]{BlanchardKirchberg:JOT}) if for any $a \in A_+$ and any $\e>0$ there exists $x \in \overline{aAa}$ such that $x^2=0$ and $(a-\e)_+$ is in the ideal generated by $x$.\footnote{Equivalent is the \emph{global Glimm property} of \cite[Definition 4.12]{KirchbergRordam:AIM}, which asks that for all non-zero $a\in A_+$, $\epsilon>0$, and all $n\geq 2$ (not just $n=2$), there is an order zero map $M_n\to\overline{aAa}$ such that the ideal generated by the image contains $(a-\e)_+$.} 

\begin{theorem}[Kirchberg--R\o{}rdam; {\cite[Proposition~4.15]{KirchbergRordam:AIM}}]
    A weakly purely infinite {$C^*$}-algebra is purely infinite if and only if it has the global Glimm halving property.
\end{theorem}

Of relevance to the broad theme of this paper, every $\Z$-stable {$C^*$}-algebra $A$ satisfies global Glimm-halving, as for every non-zero $a\in A_+$, the hereditary subalgebra $\overline{aAa}$ is $\Z$-stable by \cite[Corollary 3.1]{TomsWinter:TAMS}. Accordingly, $\overline{aAa}$ contains a full nilpotent element of order two (namely, under an identification $\overline{aAa}\cong \overline{aAa}\otimes\Z$, consider $a\otimes x$, where $x$ is a full nilpotent element of order $2$ in $\Z$). 

A {$C^*$}-algebra $A$ with the global Glimm halving property cannot have ideals $J\lhd I\lhd A$ so that  $I/J$ is non-zero and elementary; in particular, such an $A$ has no non-zero finite-dimensional representations. This `no elementary ideal-quotients' condition is the appropriate version for non-simple {$C^*$}-algebras of requiring a simple {$C^*$}-algebra to be non-elementary, and was coined \emph{nowhere scattered} by Thiel and Vilalta in \cite{ThielVilalta:JNCG} in their detailed examination of this condition. We view being nowhere scattered as the {$C^*$}-parallel to asking a von Neumann algebra to have no type I part. Indeed, a {$C^*$}-algebra is nowhere scattered if and only if every irreducible representation has image disjoint from the compacts (\cite[Theorem 3.1 (1)$\Leftrightarrow$(6)]{ThielVilalta:JNCG}).

It is open whether the global Glimm halving property always holds for all nowhere scattered {$C^*$}-algebras. 
The following is a reformulation of this question.\footnote{If Problem~\ref{q:gGlimm} has a positive answer, then the global Glimm halving property holds whenever $A$ is nowhere scattered by applying the positive solution to the hereditary subalgebra $\overline{aAa}$.
In the unital case, this was asked explicitly by Elliott and R\o rdam in 2004 (\cite[Question~2]{ElliottRordam:Abel}).}  The difficulty is that, under the hypotheses of the problem, while there will be a unital embedding $M_2\to A^{**}$ and one can approximate this strong$^*$ by order zero maps into $A$, there is no reason at all why these maps should be full (or even approximately full, in the sense of the following question).

\begin{question}\label{q:gGlimm}
    Let $A$ be a {$C^*$}-algebra with no non-zero finite-dimensional representations.
    Given $a \in A_+$ and $\epsilon>0$, does there exist an element $x\in A$ such that $x^2=0$ and $(a-\epsilon)_+$ is in the ideal generated by $x$?\footnote{In the case that $A$ is unital, this question can be simplified to: does $A$ contain a full element $x$ such that $x^2=0$?}
\end{question}

In \cite[Theorem~4.3]{BlanchardKirchberg:JOT}, Blanchard and Kirchberg answer this question affirmatively in the case of {$C^*$}-algebras with finite-dimensional Hausdorff primitive ideal space. Elliott and R\o{}rdam's \cite[Corollary 7]{ElliottRordam:Abel} answers the problem positively in the unital real rank zero case, and much more recently, Antoine, Perera, Robert, and Thiel also confirmed it for the class of {$C^*$}-algebras with stable rank one (\cite[Theorem~9.1]{APRT:Duke}) using Cuntz semigroup techniques. Related to this,  a formulation of the global Glimm property in terms of the Cuntz semigroup is undertaken in \cite{ThielVilalta:TAMS}.

Returning to the setting of simple {$C^*$}-algebras, R\o{}rdam's examples of simple algebras with infinite and finite projections (\cite{Rordam:Acta}) showed that simple nuclear infinite {$C^*$}-algebras need not be purely infinite.  It is open whether this can happen without a finite projection. This question seems to be due to Elliott, potentially sparked by Problem~\ref{q:RR0dichotomy}. Any such example could not have real rank zero. 

\begin{question}[{\cite[Question~7.7]{Rordam:Acta}}]
    Let $A$ be a unital simple (nuclear) {$C^*$}-algebra in which all non-zero projections in $A$ are infinite. Must $A$ be purely infinite?
\end{question}

\section{Regularity for non-simple {$C^*$}-algebras}\label{sec:NonsimpleTW}
Just as structure and classification theory have driven each other in the study of simple {$C^*$}-algebras, it is natural to turn here to structure -- and specifically, regularity -- for non-simple {$C^*$}-algebras.  How generally do the structural theorems (and conjectures) in the simple setting extend to the non-simple framework? Just as a main motivation of the problems in the previous section is to give a more elementary characterisation of the algebras covered by Kirchberg's $\mathcal O_\infty$-stable classification (Theorem~\ref{Thm:Oinftystable}), one hopes that any future classification theorems (such as those proposed in Section~\ref{sec:nonsimpleclass}) will also have structural counterparts.  Just as the structure theorem for simple {$C^*$}-algebras (Theorem \ref{Thm:Structure}) needs to exclude elementary algebras, the right framework for generalisations is the class of nowhere scattered {$C^*$}-algebras.  These ideas have been actively considered for some time, being floated at least as far back as Winter's 2012 CBMS lectures, and seeing the first major progress  in \cite{RobertTikuisis:TAMS} in 2014.  However, until fairly recently, it appears that a formal statement of a non-simple Toms--Winter regularity conjecture was missing from the literature; the first we are aware of is in the recent reprint \cite{APTV24}.\footnote{As we discuss further below, \cite{APTV24} uses pureness as condition \ref{it:NonsimpleTW.3} as compared to our formulation in terms of Cuntz semigroup regularity.  We hope these are equivalent.}

\begin{question}[{cf.\ \cite[Question B]{APTV24}}]\label{q:NonsimpleTW}
    Let $A$ be a separable nuclear {$C^*$}-algebra which is nowhere scattered.
    Are the following equivalent?
    \begin{enumerate}
        \item\label{it:NonsimpleTW.1}
        $A$ has finite nuclear dimension.
        \item\label{it:NonsimpleTW.2}
        $A$ is $\mathcal Z$-stable.
        \item\label{it:NonsimpleTW.3} $A$ is Cuntz semigroup regular, i.e.\ the first factor map $x\mapsto x\otimes 1_\mathcal Z$ induces an isomorphism $\mathrm{Cu}(A) \cong \mathrm{Cu}(A\otimes \mathcal Z)$.
    \end{enumerate}
\end{question}

Other than the vacuous implication \ref{it:NonsimpleTW.2}$\implies$\ref{it:NonsimpleTW.3}, none of the implications in Problem~\ref{q:NonsimpleTW} are known to hold in full generality.   Of course, we would be delighted if the equivalence in  Problem \ref{q:NonsimpleTW} holds with the weaker condition that $\Cu(A)$ is almost unperforated in place of condition \ref{it:NonsimpleTW.3}.  In that case, it would represent a full generalisation of the Toms--Winter conjecture to the non-simple setting.

Recall that for simple {$C^*$}-algebras, pureness is defined to be the combination of strict comparison, or equivalently, almost unperforation of the Cuntz semigroup, and almost-divisibility.  As discussed briefly in footnote \ref{Footnote:DefPure} a non-simple {$C^*$}-algebra $A$ is defined to be pure if $\Cu(A)$ is almost unperforated and almost divisible.  Strict comparison is then defined in the non-simple setting so that it is equivalent to almost unperforation of the Cuntz semigroup.  Cuntz semigroup regularity implies pureness in general, but while simple pure  {$C^*$}-algebras are Cuntz semigroup regular (see Proposition \ref{prop:cu-regular}), the route we took used both Lin's Theorem \ref{thm:Lin} to obtain stable rank one, together with an explicit calculation of the Cuntz semigroup in terms of the Murray--von Neumann semigroup and traces (see condition \ref{cu-reg6} of Proposition~\ref{prop:cu-regular}).  Both of these tools are not available in the non-simple setting.  However the characterisation of pureness through the Cuntz semigroup tensor product $\otimes_\Cu$ developed in \cite{APT:MAMS} holds generally: $A$ is pure if and only if $\Cu(A)\cong \Cu(A)\otimes_{\Cu} \Cu(\Z)$ (\cite[Theorem 7.3.11]{APT:MAMS}). One issue is that it is not known whether $\Cu(A)\otimes_{\Cu}\Cu(\Z)$ is isomorphic to $\Cu(A\otimes\Z)$ in general. The following problem is very close to {\cite[Question 5.5]{APTV24}}.


\begin{question}\label{q:PureCUReg}
Is every pure {$C^*$}-algebra Cuntz semigroup regular?
\end{question}

In their recent preprint obtaining a dimension reduction from $(m,m')$-pureness to pureness,  Antoine, Perera, Thiel and Vilalta show that {$C^*$}-algebras with finite nuclear dimension and the global Glimm halving property are pure (\cite[Theorem~C]{APTV24}). 

For traceless separable nuclear {$C^*$}-algebras, $\Z$-stability and $\mathcal O_\infty$-stability coincide (\cite[Theorem 5.2]{Rordam:IJM}). Accordingly, in this case, the implication \ref{it:NonsimpleTW.3}$\implies$\ref{it:NonsimpleTW.2} is the task of going from pure infiniteness to strong pure infiniteness in Problem~\ref{q:PIConditions}. Also, for traceless algebras, \ref{it:NonsimpleTW.2}$\implies$\ref{it:NonsimpleTW.1} is known in full generality;  this was first proved in \cite{Szabo:Adv}, with the exact value of the nuclear dimension -- it is $1$ -- being later obtained in  \cite{BGSW:AIM}. This latter paper goes further and works at the level of maps: nuclear $\mathcal O_\infty$-stable maps with exact domains have nuclear dimension at most $1$ (\cite[Theorem C]{BGSW:AIM}). So all the maps classified by Theorem \ref{OInftyStableMaps} have dimension $1$. But can we characterise these classifiable maps in terms of behaviour of positive elements?
One necessary condition for a map to be $\mathcal O_\infty$-stable is that its image should consist of properly infinite elements -- this forces the map to be traceless in a sufficiently strong sense.  

\begin{question}\label{q:NDim=>StableMaps}
Let $A$ and $B$ be unital {$C^*$}-algebras with $A$ separable and exact. Let $\phi\colon A \to B$ be a unital $^*$-homomorphism with finite nuclear dimension such that every non-zero positive element in $\phi(A)$ is properly infinite in $B$.  Must $\phi$ be $\mathcal O_\infty$-stable?
\end{question}

Note that if we take $A=B$ and $\phi$ to be the identity map in Problem \ref{q:NDim=>StableMaps} then this becomes the pure infiniteness implies strong pure infiniteness problem, with an additional finite nuclear dimension assumption (to our knowledge, such an assumption has not yet been successfully used to make progress on Problem~\ref{q:PIConditions}).  It is possible to pose numerous other problems asking for map versions of existing results. Here is one, suggested to us by Jamie Gabe, asking to upgrade a $\Z$-stable map to $\mathcal O_\infty$-stability in the absence of traces (akin to R\o{}rdam's \cite[Theorem~5.2]{Rordam:IJM}).
\begin{question}
Let $A$ be unital simple separable nuclear and $\Z$-stable and let $B$ be unital and have no normalized quasitraces. Is every unital $^*$-homomorphism $A\to B$ automatically $\mathcal O_\infty$-stable?
\end{question}
As Gabe observed, a positive answer would positively resolve the real rank zero dichotomy problem (Problem \ref{q:RR0dichotomy}). Indeed, given a unital simple real rank zero algebra $B$, Elliott and R\o{}rdam provide a unital simple non-elementary AF subalgebra $A\subset B$ (\cite{ElliottRordam:Abel}) which will certainly be $\Z$-stable.  If the inclusion $A\subset B$ was $\mathcal O_\infty$-stable then $B$ would be properly infinite.  Repeating this argument in corners shows that every projection in $B$ is infinite, and so $B$ is purely infinite (see \cite[Proposition 4.1.1]{Rordam:Book}, for example).

In the rest of this section, we will focus on potential generalisations of the known equivalence between (i) and (ii) from the structure theorem for simple nuclear {$C^*$}-algebras.   Firstly, in contrast to Section \ref{sec:nonsimpleclass} -- where classification problems are open even in the presence of a unique non-trivial ideal -- since both conditions \ref{it:NonsimpleTW.1} and \ref{it:NonsimpleTW.2} are closed under extensions, induction using the structure theorem shows they are equivalent whenever $A$ has finitely many ideals.
However finite decomposition rank is not preserved by general extensions, so it is natural to ask for the version of Theorem \ref{Thm:Structure}(i$''$) in this case.
\begin{question}
    Let $A$ be a nowhere scattered {$C^*$}-algebra with finitely many ideals and finite nuclear dimension.  Characterise, in terms of conditions relating to quasidiagonality, when $A$ has finite decomposition rank.
\end{question}

It is also natural to consider when $\Z$-stable separable nuclear {$C^*$}-algebras have finite decomposition rank more generally.  In the $\mathcal O_\infty$-stable case the answer is given by \cite[Theorem B]{BGSW:AIM}: when all their quotients are quasidiagonal.\footnote{This happens precisely when the primitive ideal space has no locally closed singleton sets.}

With Robert, the second-named author showed that Winter's argument from \cite{Winter:IM12} can be generalised to prove the implication \ref{it:NonsimpleTW.1}$\implies$\ref{it:NonsimpleTW.2} in Problem~\ref{q:NonsimpleTW} provided there are two orthogonal full positive elements in Kirchberg's central sequence algebra $F(A)$ (which is $A_\omega \cap A'$ in the unital case). Indeed, a separable {$C^*$}-algebra with finite nuclear dimension is $\Z$-stable precisely when this condition on the central sequence algebra holds (\cite[Theorem 1.2]{RobertTikuisis:TAMS}). Such a condition in $F(A)$ is in the spirit of (but a priori weaker than) global Glimm halving,\footnote{Kirchberg and R{\o}rdam in fact show that this condition is equivalent to Glimm halving for $F(A)$, \cite[Proposition~6.1]{KirchbergRordam:IJM}} so it is natural to ask when $F(A)$ is nowhere scattered.  Ando and Kirchberg showed that when $A$ is not of type I, then $F(A)$ is never abelian, so in contrast with Murray and von Neumann's foundational result that the von Neumann algebra $L(\mathbb F_r)$ associated to a free group has a trivial central sequence algebra, $F(C^*_r(\mathbb F_2))$ is non-abelian (\cite{AndoKirchberg:JLMS}).\footnote{Subsequently, Enders and Schulman characterised the type I {$C^*$}-algebras with an abelian central sequence algebra as those satisfying Fell's condition in \cite{EndersShulman:AIM}. This answered a question of Ando and Kirchberg.}
  But, despite this, even for simple $A$, some condition is needed for its central sequence algebra to be nowhere scattered as it is possible for it to have characters (\cite{KirchbergRordam:IJM,Christensen:MathScand}).\footnote{In \cite{KirchbergRordam:IJM}, Kirchberg and R\o{}rdam examined characters on central sequence algebras, showing their absence gives rise to the corona factorisation property.}
In fact, Kirchberg and R\o{}rdam asked the following, noting its equivalence to an older question raised by Dadarlat and Toms in \cite{DadarlatToms:Adv}.\footnote{Dadarlat and Toms asked if $\Z$ embeds unitally into the minimal tensor product $D^{\otimes\infty}$ whenever $D$ is unital, separable, and has no characters. This problem is related to the questions in Section~\ref{Sec:Zembed}.}

\begin{question}[{\cite[Question 3.1]{KirchbergRordam:IJM}}]
    Let $A$ be a unital separable {$C^*$}-algebra such that $A_\omega\cap A'$ has no characters. Is $A$ necessarily $\Z$-stable?
\end{question}

Returning to our main theme, Robert and the second-named author were able to verify their full orthogonal elements condition provided both the two conditions below hold.
\begin{enumerate}[(a)]
    \item $A$ has no purely infinite ideal-quotients.
    \item The primitive ideal space is either Hausdorff or has a basis of compact open sets.
\end{enumerate}
This gives the implication \ref{it:NonsimpleTW.1}$\implies$\ref{it:NonsimpleTW.2} in Problem \ref{q:NonsimpleTW} under such conditions (\cite[Theorem 1.1]{RobertTikuisis:TAMS}).  

In the other direction, the main result of \cite{ENST:FM} (building on \cite{TikuisisWinter:APDE}) is that $\mathcal Z$-stable ASH algebras have finite nuclear dimension (and, in fact, finite decomposition rank; though the known bound is currently $2$, and by now, we would expect the bound to be $1$).  Returning to the traceless case, for a general separable nuclear $\mathcal O_\infty$-stable {$C^*$}-algebra $A$, Szab\'o first obtained the bound $\dim_{\mathrm{nuc}}(A)\leq 3$ in \cite{Szabo:Adv}, crucially using a surprising  construction of R\o{}rdam: an $\mathcal O_\infty$-stable {$C^*$}-algebra $\mathcal A_{[0,1]}$ arising as an inductive limit of building blocks of the form $C_0((0,1],M_k)$ (\cite{Rordam:Israel}).  This $\mathcal A_{[0,1]}$ has primitive ideal space $(0,1]$ equipped with the (non-Hausdorff) right order topology, and is homotopic to zero in an ideal-preserving fashion.\footnote{See Problem~\ref{q:NullHomotopy1NCCW} for the definition.}
The construction of $\mathcal A_{[0,1]}$ provided the motivation for Kirchberg and R\o{}rdam's Theorem \ref{Thm:KRNullHomotopic}. Szabó's nuclear dimension calculation applies Theorem \ref{Thm:KRNullHomotopic} to $A\otimes \mathcal A_{[0,1]}$, which is homotopic to zero in an ideal-preserving fashion and is therefore a limit of homogenous algebras over one-dimensional spaces, and he then bounds the nuclear dimension of $A$ in terms of the nuclear dimension of $A\otimes \mathcal A_{[0,1]}$.
A positive answer to Problem~\ref{q:NullHomotopy1NCCW} (beyond the purely infinite case) might lead to a similar bound on the nuclear dimension of a general separable nuclear $\mathcal Z$-stable {$C^*$}-algebra.

Following Szabó's work (\cite{Szabo:Adv}), the precise value of the nuclear dimension for a separable nuclear $\mathcal O_\infty$-stable {$C^*$}-algebra $A$ was obtained by fully exploiting the $\mathcal O_\infty$-stable classification theorem (Theorem \ref{Thm:Oinftystable}). The idea, going back to \cite{MatuiSato:Duke,BBSTWW:MAMS}, is to split the identity map on $A$ into the sum of two maps which factor through cones, and use classification to compare these with $0$-dimensional models (whose construction heavily uses quasidiagonality).  Accordingly, we expect sufficiently strong classification theorems, together with a collection of $0$-dimensional models, for (non-full) nuclear maps out of cones into $\Z$-stable {$C^*$}-algebras to give rise to new abstract situations where $\Z$-stability implies nuclear dimension at most one, but perhaps it will take such classification theorems to obtain the optimal value of the nuclear dimension for all separable nuclear $\Z$-stable {$C^*$}-algebras.

Other than the situation of finitely many ideals, we are not aware of other results on the equivalence between \ref{it:NonsimpleTW.1} and \ref{it:NonsimpleTW.2} which are valid for {$C^*$}-algebras $A$ containing both stably finite and purely infinite ideal-quotients.

\section{Computing nuclear dimension}

Although the nuclear dimension is a non-commutative generalisation of covering dimension in the sense that $\mathrm{dim}_{\mathrm{nuc}}(C_0(X))=\dim X$ (for $X$ $\sigma$-compact), and hence all possible values of the nuclear dimension occur, for simple {$C^*$}-algebras, the structure theorem (Theorem~\ref{Thm:Structure}) shows that the only possible values of the nuclear dimension are $0,1,$ and $\infty$. 
It is reasonable to wonder to what extent the nuclear dimension really provides a notion of dimension in the non-commutative setting and how to go about computing it outside the simple setting (or more generally, the nowhere-scattered setting, where the the generalised Toms--Winter conjecture of Problem \ref{q:NonsimpleTW} predicts the behaviour of the nuclear dimension). Since bounds on the nuclear dimension descend both to quotients and hereditary subalgebras, one way to obtain a lower bound on the nuclear dimension is through high-dimensional commutative hereditary subalgebras of quotients.  The spirit behind the following problem is to ask whether other methods are even possible.

\begin{question}\label{q:dimnucValues}
    Let $A$ be a {$C^*$}-algebra with the following property: for any ideal $I$ of $A$ and any hereditary subalgebra $C$ of $A/I$, if $C$ is commutative, then its  primitive ideal space has dimension at most one.
    Is it true that $\mathrm{dim}_{\mathrm{nuc}}(A) \in \{0,1,\infty\}$?  
\end{question}

We believe Problem \ref{q:dimnucValues} is related to the computation of the nuclear dimension of an extension. Winter and Zacharias showed that finiteness of nuclear dimension is preserved by extensions,\footnote{This is in contrast to the earlier notion of the decomposition rank from \cite{KirchbergWinter:IJM}: the compacts and $C(\mathbb T)$ have finite decomposition rank, but the Toeplitz algebra does not.} but the nature of the upper bound attained -- and its disparity from the known behaviour in the  commutative setting -- raises the following question (classical results for dimensions of metric spaces give a positive answer when $E$ is separable and commutative\footnote{These ideas were recently pushed to give a positive answer to Problem \ref{q:NdimExtension} for separable extensions where both the ideal and quotient are subhomogeneous by an Huef and Williams in \cite[Lemma 7.2]{anHuefWilliams:arXiv}. In this case, the extension is also subhomogeneous, and its subhomogeneity degree is equal to the maximum of the subhomogeneity degrees of the ideal and quotient.}).  Note that if the answer is positive, then separable nuclear $\Z$-stable {$C^*$}-algebras with finitely many ideals have nuclear dimension at most $1$.

\begin{question}\label{q:NdimExtension}
    Let $0 \to I \to E \to D \to 0$ be an extension of {$C^*$}-algebras.
    Is it true that
    \begin{equation}\mathrm{dim}_{\mathrm{nuc}}(E) = \mathrm{max}\{\mathrm{dim}_{\mathrm{nuc}}(I),\mathrm{dim}_{\mathrm{nuc}}(D)\}?
    \end{equation}
\end{question}

A negative answer to this Problem \ref{q:NdimExtension} would suggest a negative answer to Problem~\ref{q:dimnucValues}, whereas a positive answer (combined with Theorem~\ref{Thm:Structure}) would resolve Problem~\ref{q:dimnucValues} in the case that $A$ has finitely many ideals.

The nuclear dimension zero {$C^*$}-algebras are precisely the AF algebras, and since the extension of AF algebras by AF algebras is again AF (\cite{Brown80}), Problem \ref{q:NdimExtension} has a positive answer when both ideal and quotient are zero-dimensional. The Toeplitz algebra was an important test case for the nuclear dimension of extensions; Winter and Zacharias' original estimates show that its nuclear dimension is either $1$ or $2$ (\cite[Proposition 2.9]{WinterZacharias:AIM}), and nearly 10 years later, Brake and Winter elegantly demonstrated that the value is $1$ (\cite{BrakeWinter:BLMS}). This provided the motivation for Problem \ref{q:NdimExtension} and sparked further progress for ideals and quotients of particular forms (\cite{15authors,Evington:JOT,GardnerTikuisis,ENSW:inPrep}).\footnote{For example, \cite{GardnerTikuisis} and \cite{15authors} cover extensions of commutative algebras and Kirchberg algebras, respectively, by the compacts.}  

As every finite graph algebra is built up from iterated extensions of {$C^*$}-algebras of nuclear dimension at most one,\footnote{To be more precise, finite graph algebras are built from iterated extensions of AF algebras, Kirchberg algebras and $C(\mathbb T)\otimes F$, where $F$ is elementary.} a positive answer to Problem \ref{q:NdimExtension} gives a positive answer to the following problem.

\begin{question}[{\cite[Question C]{ENSW:inPrep}}]\label{q:NDimGraph}
Is the nuclear dimension of any graph {$C^*$}-algebra at most one?
\end{question}

As pointed out to us by the referee, arguably the first nuclear dimension computation for graph algebras is found in Winter and Zacharias foundational paper \cite{WinterZacharias:AIM}, where they show $\dim_{\mathrm{nuc}}(\mathcal O_n)=1$ for $n\in \mathbb N$ (they also obtained the estimate $\dim_{\mathrm{nuc}}(\mathcal O_\infty)\leq 2$; see \cite[Theorem 7.4]{WinterZacharias:AIM}).  This calculation is performed using the standard graph model for $\mathcal O_n$, and not via another model accessed through classification.  However, focus on the nuclear dimension of more general graph {$C^*$}-algebras really begins in 
 \cite{RST:Adv}, which relies crucially on Enders' work \cite{Enders:JFA} showing that UCT Kirchberg algebras with torsion-free $K_1$-group have nuclear dimension $1$.  Enders' theorem covers all simple purely infinite graph {$C^*$}-algebras. There are now a number of examples of positive answers to Problem \ref{q:NDimGraph} and also situations where the slightly larger upper bound of $2$ has been established. For example, for graph {$C^*$}-algebras with real rank zero,\footnote{A graph {$C^*$}-algebra has real rank zero if and only if the graph satisfies a property known (for slightly opaque reasons; see  \cite[Section 6]{KPRR:JFA}) as `condition (K)' (\cite{JeongPark:JFA}). This condition (K) is a kind of freeness condition on the graph dynamics equivalent to all ideals being invariant under the gauge action.}  the nuclear dimension is at most $2$ (\cite[Section~5]{RST:Adv}, \cite{FaurotSchafhauser:PAMS}). Real rank zero has the effect of ruling out algebras stably isomorphic to $C(\mathbb T)$ from the composition series of a graph algebra. In the case of a finite graph, if the graph $C^*$-algebra has real rank zero, then it decomposes as an extension of an $\mathcal O_\infty$-stable quotient by an AF ideal (see the proof of \cite[Theorem A]{FaurotSchafhauser:PAMS}). Using this, and the fact that nuclear dimension computations reduce to finite graphs,  \cite[Theorem B]{FaurotSchafhauser:PAMS} obtains an optimal upper bound of one for graphs with condition (K) satisfying an additional condition on sources, by means of using the results on Problem \ref{q:NdimExtension} from \cite{Evington:JOT}. Some examples of positive answers to Problem \ref{q:NDimGraph} coming from graphs with stabilised circle algebra quotients can be found in \cite{ENSW:inPrep}.

In another direction, {$C^*$}-algebras associated to amenable groups give a particularly prominent class of examples which have some type I behaviour.  

\begin{question}\label{q:NDimGroup}
Determine for which countable discrete amenable groups the group {$C^*$}-algebra has finite nuclear dimension.  Even better, compute the exact nuclear dimension of such groups.
\end{question}

Finitely generated groups with polynomial growth form a particularly natural class to consider. These have been extensively examined by Eckhardt in collaboration with Gillaspy and McKenny: the {$C^*$}-algebras generated by these groups have finite nuclear dimension, and in fact, finite decomposition rank (\cite{EMc:Crelle,EGMc:TAMS}).\footnote{It is worth noting that \cite{EMc:Crelle}, showing that finitely generated nilpotent groups have finite nuclear dimension, went through the monotracial case of the structure theorem.} This was followed by work (\cite{EG:MJM,Eckhardt:Adv}) on the UCT for irreducible representations of such groups   culminating in the classifiability of {$C^*$}-algebras associated to infinite-dimensional irreducible representations of groups of polynomial growth (\cite[Theorem A]{Eckhardt:Adv}).\footnote{Infinite dimensionality is needed as our definition of classifiability excludes elementary {$C^*$}-algebras.} These have unique trace and real rank zero. While outside the scope of the main topic of this section, we definitely want to ask for more computations of the invariant, in the spirit of \cite{EKMc:JFA}, which handles the case of representations of the unitriangular group $UT(4,\mathbb Z)$; interestingly, the classification theorem reveals unexpected isomorphisms for this {$C^*$}-algebra (see \cite[Theorem~6.2]{EKMc:JFA}).

\begin{question}    
Compute $KT_u(C^*_\pi(G))$ for infinite-dimensional irreducible representations $\pi$ of groups $G$ of polynomial growth.
\end{question}

Work on the decomposition rank and nuclear dimension of group {$C^*$}-algebras has been connected to the regularity questions in Sections \ref{sec:crossed-product} and \ref{Sect:noncomdynamics} by way of wreath products. For example $C^*(\mathbb Z\wr \mathbb Z)$ has infinite nuclear dimension as a consequence of Giol and Kerr's construction of a non-classifiable simple crossed product $C(X)\rtimes \mathbb Z$ in \cite{GK:Crelle} (see \cite[Therorem 5.1]{EMc:Crelle}), and the {$C^*$}-algebra of the lamplighter group $(\mathbb Z/2\mathbb Z)\wr \mathbb Z$ has finite nuclear dimension as a consequence of Hirshberg and Wu's result that $C(X)\rtimes\mathbb Z$ has finite nuclear dimension for any action of $\mathbb Z$ on a finite-dimensional compact Hausdorff space $X$ (\cite{HW:Adv}).  The lamplighter group is not strongly quasidiagonal (\cite[Corollary 3.5]{CarrionDadarlatEckhard:JFA}), and so its {$C^*$}-algebra cannot have finite decomposition rank (\cite{KirchbergWinter:IJM}). At present, there are no known examples of finitely generated groups whose {$C^*$}-algebra has finite decomposition rank which do not have polynomial growth. 
In the elementary amenable setting, based on the analysis of certain wreath products, Eckhardt and Wu conjecture that polynomial growth is necessary.

\begin{question}[{\cite[Conjecture II]{EW:arXiv24}}]
Let $G$ be a finitely generated elementary amenable group. Is it true that $C^*(G)$ has finite decomposition rank if and only if $G$ has polynomial growth?
\end{question}

The examples of strongly quasidiagonal groups which fail to have polynomial growth from \cite{Eckhardt:JOT} may be worth examining. These are of the form $\mathbb Z^3\rtimes_\alpha \mathbb Z^2$ for suitable actions, and so by means of Hirshberg and Wu's long-thin-covering dimension (\cite{HW:arXiv}), the resulting {$C^*$}-algebras will have finite nuclear dimension.  Pushing this further, a recent preprint of Eckhardt and Wu (\cite{EW:arXiv24}) combines the strategies of \cite{EMc:Crelle,EGMc:TAMS} with the ideas from  \cite{HW:arXiv} to tackle virtually polycyclic groups.
\begin{theorem}[{\cite[Theorem A]{EW:arXiv24}}]
    Let $G$ be a finitely generated virtually polycyclic group. Then $C^*(G)$ has finite nuclear dimension, bounded in terms of the Hirsch length of $G$.
\end{theorem}
Based on this and the fact that all known examples of groups with {$C^*$}-algebras of finite nuclear dimension have finite Hirsch length (in the sense of Hilman in the case that the group is not virtually polycyclic), Eckhardt and Wu conjecture a partial answer to Problem 
\ref{q:NDimGroup}: for a finitely generated elementary amenable group $G$, $C^*(G)$ has finite nuclear dimension if and only if $G$ has finite Hirsch length (\cite[Conjecture I]{EW:arXiv24}). 

Roe algebras provide another case of interest.
Here, the asymptotic dimension of a discrete metric space $X$ provides an upper bound for the nuclear dimension of the corresponding uniform Roe algebra $C^*_u(X)$ (\cite[Theorem 8.5]{WinterZacharias:AIM}).
Willett and Winter (and others) have made attempts at a converse, which is the following problem.   Even the case $X \coloneqq \mathbb Z^2$ (which has asymptotic dimension two) is open, and at least one of the authors suspects, or perhaps hopes, that the answer might be negative.  The only situations where positive answers are known is when the asymptotic dimension is zero (which, by \cite[Corollary~1.5]{LiWillett:JLMS}, is precisely when the uniform Roe algebra is AF,\footnote{Normally for non-separable {$C^*$}-algebras, one needs to worry about whether AF should be interpreted as a local property, or as an inductive limit of finite-dimensional {$C^*$}-algebras, as these properties are not generally equivalent outside the separable setting. But for uniform Roe algebras \cite[Corollary~1.5]{LiWillett:JLMS} shows these are equivalent.} or equivalently, when it has nuclear dimension zero)  or one.

\begin{question}[{\cite[Question~9.5]{WinterZacharias:AIM}}] \label{q:dimnucRoe}
    Let $X$ be a countable discrete metric space with bounded geometry.
    Is $\mathrm{dim}_{\mathrm{nuc}}(C^*_u(X))$ equal to the asymptotic dimension of $X$?
\end{question}

In the preprint \cite{LiLiaoWinter}, Li, Liao, and Winter introduced a variant of nuclear dimension, called \emph{diagonal dimension}, that takes into account a diagonal subalgebra.
They showed that for a discrete metric space $X$ with bounded geometry, the diagonal dimension of $\ell_\infty(X) \subseteq C^*_u(X)$ does agree with the asymptotic dimension of $X$.
Roughly, this interprets Problem~\ref{q:dimnucRoe} as follows: does the nuclear dimension of $C^*_u(X)$ see the diagonal subalgebra $\ell_\infty(X)$?
Through rigidity results by the last-named author and Willett (\cite{WhiteWillett:GGD}) and Baudier, Braga, Farah, Khukhro, Vignati, and Willett (\cite{BBFKVW:IM}; for the case of interest to this problem -- when $C^*_u(X)$ is nuclear -- \v{S}pakula and Willett's original \cite{SpakulaWillett:AIM} suffices), we now know that $C^*_u(X)$ itself does remember the diagonal subalgebra $\ell_\infty(X)$.

\section{Semiprojectivity}

Semiprojectivity is a lifting property for {$C^*$}-algebras which is intimately connected to stability of relations (in the uniform norm).
It arose in an adaptation of shape theory to the non-commutative setting by Effros--Kaminker and Blackadar (\cite{EffrosKaminker:GMOA,Blackadar:MS}); whereas an initial definition was made by Effros and Kaminker, Blackadar formulated a stronger notion that is now in use: a {$C^*$}-algebra $A$ is \emph{semiprojective} if for any {$C^*$}-algebra $B$, any increasing sequence of ideals $I_1\lhd I_2 \lhd \cdots\lhd B$, and any $^*$-homomorphism $\phi:A \to B/\overline{\bigcup I_n}$, there exists some $n$ and a $^*$-homomorphism $\tilde\phi:A \to B/I_n$ giving a (partial) lift of $\phi$.  If we restrict the algebras $B$ and $I_n$ to lie in some subcategory of {$C^*$}-algebras, then $A$ is said to be \emph{semiprojective} in that subcategory.

The definition is designed to be a non-commutative version of absolute neighbourhood retracts (ANR), and indeed, a unital commutative {$C^*$}-algebra $C(X)$ is semiprojective in the category of of commutative {$C^*$}-algebras precisely when $X$ is an ANR (\cite[Proposition~2.11]{Blackadar:MS}).\footnote{A unital commutative {$C^*$}-algebra that is semiprojective in the commutative category need not be semiprojective (in the category of all {$C^*$}-algebras).  For example, $C(\mathbb T^2)$ is semiprojective in the category of commutative {$C^*$}-algebras, but Voiculescu's work on almost commuting unitaries show that this is not true in the category of all {$C^*$}-algebras (\cite{Voiculescu:ActaSci}).  As shown by S\o{}rensen and Thiel, semiprojectivity of $C(X)$ forces the dimension of $X$  to be at most one, and moreover, this is the only other obstruction: if $X$ is an ANR of dimension at most one, then $C(X)$ is semiprojective (\cite{SorensenThiel:PLMS}).  This crucially uses Loring's semiprojectivity of $C(X)$ for finite graphs $X$ (\cite[Theorem~5.1]{Loring:PJM2}).}
Some natural examples of {$C^*$}-algebras are semiprojective: finite-dimensional {$C^*$}-algebras, $C([0,1])$, $C(\mathbb T)$, $\mathcal O_n$, and $C^*(\mathbb F_n)$ (for $n$ finite) for example.
This is usually proven by expressing them as appropriate universal {$C^*$}-algebras, although other examples such as $\mathcal O_\infty$ (\cite{Blackadar:ASPM}) and general 1-dimensional non-commutative CW complexes (\cite{ELP:Crelle}) can require significantly more intricate arguments.
More generally, in his unpublished paper \cite{Enders15}, Enders shows that a UCT Kirchberg algebra is semiprojective if and only if it has finitely generated $K$-theory.
(On the other hand, no finite classifiable {$C^*$}-algebra is semiprojective, as a quasidiagonal semiprojective {$C^*$}-algebra must be residually finite.)

A fundamental result in topological shape theory is that every compact Hausdorff space is an inverse limit of absolute neighbourhood retracts.
There are a couple of reasonable formulations of potential non-commutative analogues of this result; the most straightforward would be that all separable {$C^*$}-algebras are inductive limits of semiprojective {$C^*$}-algebras, but another possibility is that all separable nuclear {$C^*$}-algebras are inductive limits of semiprojective nuclear {$C^*$}-algebras.
Blackadar raised both of these questions nearly 40 years ago in \cite{Blackadar:MS}, and they remain widely open.
In the absence of a positive answer to either question, Blackadar rescued non-commutative shape theory by noting that an arbitrary separable {$C^*$}-algebra can be written as an inductive limit of (highly non-nuclear) {$C^*$}-algebras with semiprojective connecting maps\footnote{Semiprojectivity of a $^*$-homomorphism is a relative version of the definition of semiprojectivity of a {$C^*$}-algebra; see \cite[Definition~2.10]{Blackadar:MS}.} (\cite[Proposition~4.2]{Blackadar:MS}).

\begin{question}[{Blackadar, {(\cite{Blackadar:MS})}}]\label{q:Semiprojective}
\begin{enumerate}[(1)]
    \item Is every separable {$C^*$}-algebra an inductive limit of a sequence of semiprojective {$C^*$}-algebras?\label{q:Semiprojective.1}
    \item Is every separable nuclear {$C^*$}-algebra an inductive limit of a sequence of semiprojective nuclear {$C^*$}-algebras? Or at least, of a sequence of nuclear {$C^*$}-algebras, with semiprojective connecting maps? \label{q:Semiprojective.2}
    \end{enumerate}
\end{question}


In \cite{Thiel:Adv}, Thiel shows that the set of {$C^*$}-algebras which have a positive answer to Problem \ref{q:Semiprojective}(\ref{q:Semiprojective.1}) (known as having a \emph{strong shape system}) is closed under shape domination, and in particular under homotopy.  This gives rise to a number of positive answers to the problem (see \cite[Theorem 5.4]{Thiel:Adv}).  However, there is still a long way to go: Blackadar noted that the answer to the second of these questions is ``not even clear'' for commutative {$C^*$}-algebras, and indeed, this case of both problems remains open! 

\section{Examples of non-nuclear simple pure {$C^*$}-algebras}
\label{sec:NonNuc}

Given the role of Cuntz semigroup regularity -- equivalently pureness -- in the study of simple nuclear stably finite {$C^*$}-algebras, it is logical to seek natural examples outside the nuclear setting.  A good starting place is mono(quasi)tracial simple {$C^*$}-algebras as then all ranks occur and this becomes a search for strict comparison. Based on ideas of Dykema and R\o rdam (\cite{DykemaRordam:MZ}), $C^*_r(\mathbb F_\infty)$, which is the reduced free product of infinitely many copies of $C(\mathbb T)$ with respect to integration against Haar measure, has strict comparison of positive elements (\cite[Proposition~6.3.2]{Robert:AIM}).\footnote{Robert attributes \cite[Proposition~6.3.2]{Robert:AIM} to R\o rdam.}  These arguments cover many other reduced free products of \emph{infinitely many} {$C^*$}-algebras; the infiniteness of the family is used to ensure that given elements $a,b$ in (matrices over) a finite stage of free product, one can find a unitary $u$ from a later component so that $uC^*(a)u^*$ and $C^*(b)$ are free.   In his paper, Robert pointed out that the following question is open.\footnote{It is no longer open; see Addendum~\ref{Addendum25.1}.}
It is also asked as \cite[Problem~16.4]{GardellaPerera}.\footnote{In \cite{GardellaPerera}, they also ask for a computation of $\mathrm{Cu}(C^*_r(\mathbb F_2))$. If $C^*_r(\mathbb F_2)$ has strict comparison -- which would be the expected outcome -- then using this, unique trace, and stable rank one, it would follow that $\mathrm{Cu}(C^*_r(\mathbb F_2))\cong \mathrm{Cu}(\mathcal Z)$.}

\begin{question}\label{q:C*F2comparison}
    Does $C^*_r(\mathbb F_2)$ have strict comparison?
\end{question}

More generally, following the breakthrough result \cite{BKKO:IHES} showing reduced group {$C^*$}-algebras have a unique trace whenever they are simple, we hope that this trace is good enough to see the order on positive elements.

\begin{question}\label{q:C*simplecomparison}
    Let $G$ be a countable discrete {$C^*$}-simple group (i.e.\ $C^*_r(G)$ is simple).  Does $C^*_r(G)$ have strict comparison with respect to its unique trace?\footnote{As pointed out to the authors by Hannes Thiel, it seems challenging even to determine whether $C^*_r(G)$ has a unique quasitrace in this case (i.e.\  whether Problem~\ref{q:QT} holds for $C^*_r(G)$ when $G$ is {$C^*$}-simple). Strict comparison with respect to the unique trace would show that there is a unique quasitrace. For exact groups, Haagerup's Theorem \ref{thm:HaagerupQT} skirts the quasitrace problem, and a solution to Problem \ref{q:C*simplecomparison} for exact groups would already be very significant progress.}
\end{question}

In the preprint \cite{Robert23}, Robert gives a unified proof of strict comparison for $C^*_r(\mathbb F_\infty)$ and $\mathcal Z$ (among other {$C^*$}-algebras); he defines a {$C^*$}-algebra $A$ with a state\footnote{In our context, $\tau$ will always be a trace, hence the choice of notation.} $\tau$ to be \emph{selfless} if there exists a $^*$-homomorphism $\sigma$ as in the following diagram, where $i_1:A \to (A,\tau)^{*\infty}$ denotes the first-factor embedding into the reduced free product:
\begin{equation}
\begin{tikzcd}
    A \ar[rr] \ar[dr,"i_1", swap] && A_\omega \\
    & (A,\tau)^{*\infty} \ar[ur,dotted,"\exists \sigma", swap].
\end{tikzcd}
\end{equation}
It is shown that both $\mathcal Z$ and $C^*_r(\mathbb F_\infty)$ are selfless in this sense and that selfless {$C^*$}-algebras are simple, with strict comparison, and are either purely infinite, or have a unique trace, which is also the unique quasitrace  (\cite[Theorem 5.2, Proposition 2.2 and Theorem 3.1, respectively]{Robert23}).
Refining Problems~\ref{q:C*F2comparison} and \ref{q:C*simplecomparison}, Robert asked the following.\footnote{This was Question 12 in the first version of \cite{Robert23} on the arXiv, but, with the first part of the question being solved (see Addendum \ref{Addendum25.1}), it no longer appears explicitly in the revised version.}

\begin{question}\label{q:Selfless}
Is $C^*_r(\mathbb F_2)$ selfless? If $G$ is a countable discrete {$C^*$}-simple group, is $C^*_r(G)$ selfless?
\end{question}

Another very natural class to consider is crossed products. In the setting of II$_1$ factors, taking a crossed product associated to an outer action by an amenable group preserves the McDuff property and property $\Gamma$ (\cite{Bedos:JFA,Bisch:TAMS}). One could ask similar questions at the {$C^*$}-level, regarding preservation of $\Z$-stability under crossed products (in the spirit of many of the questions in Section \ref{Sect:noncomdynamics}), but arguably even more fundamental is to determine when pureness is preserved under crossed products. For simple purely infinite {$C^*$}-algebras, crossed products by outer actions of amenable groups preserve pure infiniteness (and hence pureness); see \cite[Lemma~10]{KK:OAA}. In the stably finite setting, we are not aware of corresponding results, save those obtained as consequences of situations where one has preservation of $\Z$-stability (as per Problem \ref{q:z-stable-product}). The following question is wide open for simple pure {$C^*$}-algebras which are not $\Z$-stable.

\begin{question}\label{PureAmenableCrossedProd}
Suppose that $G\curvearrowright A$ is an outer action of a countable discrete amenable group on a unital simple separable pure {$C^*$}-algebra. Under what conditions is $A\rtimes G$ pure?
\end{question}

Central sequences have played an immense role in the theory of von Neumann algebras, both in the injective setting (e.g.\ in the proof of Connes' theorem) and for non-injective algebras where they were the source of large numbers of examples (\cite{McDuff:Ann}) and remain of interest through to the present day.  Accordingly, we should like to better understand central sequences of regular, even $\Z$-stable, simple {$C^*$}-algebras outside the nuclear setting. For a separable unital {$C^*$}-algebra $A$, it is a result of Kirchberg that if the central sequence algebra $A_\omega\cap A'$ is simple, then $A$ is either a matrix algebra or simple,  nuclear, and purely infinite (in which case, $A_\omega\cap A'$ is also purely infinite); see \cite[Proposition 2.10 and Theorem 2.12]{Kirchberg:Abel}. What restrictions are there on the central sequence algebra having some form of comparison without being simple -- in particular, can we find examples of non-nuclear unital simple $\Z$-stable stably finite $A$ with property (SI)? (Recall Matui and Sato's Theorem \ref{Thm:SI}, where property (SI) follows from strict comparison in the simple nuclear case).

An analogy can be drawn between property (SI) for {$C^*$}-algebras and the super McDuff property for a {\rm II}$_1$ factor $\mathcal M$,\footnote{The terminology is from \cite{GoldbringHart:IMRN}, but the property goes back to Dixmier and Lance's work \cite{DixmierLance:Invent}, which obtained the 6$^\text{th}$ and 7$^\text{th}$ distinct examples of separably acting II$_1$ factors (see \cite[Corollaire~26]{DixmierLance:Invent}).} which asks that $\mathcal M^\omega \cap \mathcal M'$ is a ${\rm II}_1$ factor. When $\mathcal M$ is McDuff, it follows that $\mathcal M^\omega\cap \mathcal M'$ is type II$_1$ but not necessarily a factor; being super McDuff is equivalent to $\mathcal M$ being McDuff and the trace on $\mathcal M^\omega$ -- which is the trace induced from $\mathcal M$ -- classifying projections in $\mathcal M^\omega \cap \mathcal M'$.  Property (SI) asks for small-to-large comparison in the central sequence algebra $A_\omega\cap A'$ by the limit traces coming from $A$. 

By \cite[Proposition~19]{DixmierLance:Invent}, the ${\rm II}_1$ factor $L(\mathbb F_2) \bar\otimes \mathcal R$ is super McDuff (and indeed, so too is the tensor product of any II$_1$ factor without property $\Gamma$ with the hyperfinite II$_1$ factor; this follows from \cite[Theorem 4.7]{FangGeLi:Taiwan}, see also \cite[Section 6]{AGKE:Adv}).  This suggests the following question.

\begin{question}
Does $C^*_r(\mathbb F_2)\otimes\mathcal Z$ have property (SI)?
\end{question}

As the central sequence algebra of $C^*_r(\mathbb F_2)\otimes \Z$ is pure (see the discussion in Section \ref{Sec:CuReg}, before Problem \ref{Q7b}) it has strict comparison. The problem is about the (quasi)traces on this algebra: $C^*_r(\mathbb F_2)\otimes\mathcal Z$ has property (SI) if and only if its central sequence algebra has a unique quasitrace (in the forward direction, this is obtained by following the arguments for pulling back traces and strict comparison along the trace-kernel extension using property (SI) from \cite{MatuiSato:Duke} and extended further in \cite{BBSTWW:MAMS}). 

Another prominent example of a simple pure {$C^*$}-algebra is the hyperfinite II$_1$ factor $\mathcal R$ (or indeed, any II$_1$ factor). Moreover, $\mathcal R$ has uniform property $\Gamma$ as it has unique trace and property $\Gamma$ (and indeed, $\mathcal R$ has the uniform McDuff property since it is McDuff), however it is not $\Z$-stable as by \cite{Ghasemi:GMJ} all II$_1$ factors are tensorially prime as {$C^*$}-algebras (i.e.\ they do not admit a tensor decomposition $A\otimes B$ with both $A$ and $B$ infinite-dimensional).  But tensorial absorption of $\Z$ is not the right notion for non-separable {$C^*$}-algebras, and instead, one should ask for separable $\Z$-stability. Whereas, as discussed in Section \ref{sec:UCT}, $\mathcal R^\omega$ is separably $\Z$-stable, this is not known for $\mathcal R$.

\begin{question}\label{Q:RSepZStable}
    Is the hyperfinite II$_1$ factor separably $\Z$-stable?
\end{question}

A positive answer to this question (which would happen if $\mathcal R$ has property (SI) -- suitably adapted to work with relative commutants in ultrapowers of sufficiently large cardinality in place of central sequence algebras) would allow $\mathcal R$ to be taken as a codomain in the forthcoming classification theorems for morphisms in \cite{CGSTW:draft}.\footnote{In fact, by forthcoming work of Hua and the last-named author (\cite{HuaWhite}), one can use II$_1$ factors as co-domains of classification of morphisms without requiring separable $\Z$-stability, but we would still very much like to know the answer to Problem \ref{Q:RSepZStable}.}

\subsection{Addendum: June 2025 -- April 2026}\label{Addendum25.1}

In the time since the first version of this article, there has been a breakthrough in obtaining selflessness and hence strict comparison for reduced group {$C^*$}-algebras.  In 2022, around a year before Robert's selflessness paper, Louder and Magee generalised the celebrated strong convergence result of Haagerup and Thorbj\o{}rnsen (\cite{HT:Annals}) for free groups to show that all limit groups\footnote{Limit groups were introduced by Sela in \cite[Section~1]{Sela} and were later shown in \cite[Theorem~1.1]{ChampetierGuirardel} to coincide with the class of \emph{fully residually free} groups: those groups $G$ satisfying that for all finite sets $S \subset G$, there is a free group $F$ and a group homomorphism $\phi \colon G \rightarrow F$ such that $\phi|_S$ is injective.  Note in particular that free groups are limit groups.}
 $G$ have a sequence of finite-dimensional representations which strongly converge to the regular representation of $G$ (\cite{LouderMagee:JFA}). In their recent paper \cite{AGKP:preprint}, Amrutam, Gao, Kunnawalkam Elayavalli, and Patchell showed how to use Louder and Magee's techniques to obtain selflessness for $C^*_r(\mathbb F_n)$ for $n\geq 2$ -- answering Problem~\ref{q:C*F2comparison} and the first part of Problem~\ref{q:Selfless} -- and go a lot further. They prove that $C^*_r(G)$ is selfless, whenever $G$ is acylindrically hyperbolic with no non-trivial finite normal subgroups and $G$ has rapid decay (\cite[Theorem B]{AGKP:preprint}; see the discussion after this theorem for a host of examples covered by their theorem).

Louder and Magee quantify results of Gilbert Baumslag and Benjamin Baumslag to obtain a quantitative version of full residual freeness for limit groups (see \cite[Lemma 1.8]{LouderMagee:JFA} in particular).  This, combined with Haagerup's rapid decay property for free groups (\cite{Haagerup:Invent}), allows them to promote the approximate embeddings $\phi_n \colon G \rightarrow\mathbb  F$ of a limit group $G$ into a free group $\mathbb F$ to an embedding $C^*_r(G) \rightarrow C^*_r(\mathbb F)_\omega$, and then Haagerup and Thorbj\o{}rnsen's result gives that $C^*_r(G)$ is MF.

The results in \cite{AGKP:preprint} follow a similar idea of combining rapid decay with a quantitative version of the mixed identity freeness condition.  Briefly, a group $G$ is \emph{mixed identity free (MIF)} if one can find an infinite order element in the discrete group ultrapower $G^\omega$ that is in free position from the diagonal copy of $G$. Alternatively, assuming $G$ has a finite generating set giving a length function $|\cdot|$, being MIF means that for every integer $n \geq 1$, there is a group homomorphism $\phi_n \colon G \ast \mathbb Z \rightarrow G$ that is injective on the ball of radius $n$.  A group is \emph{selfless} (\cite[Definition~3.1]{AGKP:preprint}) if one can further find a subexponential function $f \colon \mathbb N \rightarrow \mathbb N$ such that for every $g \in G$, $|\phi_n(g)| \leq f(|g|)$.  A similar use of rapid decay shows these embeddings extend to a $^*$-homomorphism $C^*_r(G \ast \mathbb Z) \rightarrow C^*_r(G)_\omega$.  Then the image of the generator of $\mathbb Z$ provides the Haar unity needed to show $C^*_r(G)$ is selfless.  This verifies Robert's selflessness condition (and hence also strict comparison) for $C^*_r(G)$ whenever $G$ is selfless and has rapid decay.  While it later turned out that it is possible to deduce selflessness for free groups from Louder and Magee's work (see \cite[Section 3.4]{AGKP:preprint}),  Amrutam, Gao, Kunnawalkam Elayavalli, and Patchell required  quantitative versions of results for acylindrical actions of groups on hyperbolic space to reach the full force of their \cite[Theorem 3.4]{AGKP:preprint}.

Like the decay used in these selflessness arguments, subsequent progress has been rapid.  The fact that reduced free group {$C^*$}-algebras have strict comparison has been used by Kunnawalkam Elayavalli and the first-named author (\cite{KunnawalkamElayavalli-Schafhauser}) to show that these algebras have non-isomorphic ultrapowers, or in model-theoretic language, $C^*_r(\mathbb F_n)$ and $C^*_r(\mathbb F_m)$ are not elementarily equivalent for $n \neq m$. The point is that strict comparison allows access to the abstract classification techniques of \cite{Schafhauser:Ann,CGSTW}, showing that the ultrapower of the $K_1$-group is the $K_1$-group of the ultrapower: $K_1(C^*(\mathbb F_n)_\omega)\cong K_1(C^*(\mathbb F_n))^\omega \cong ({\mathbb Z}^n)^\omega$. In another direction, Vigdorovich has established selflessness for the reduced group {$C^*$}-algebras of cocompact lattices in $PSL_n(\mathbb R)$ ($\mathbb R$ can be replaced by any local field of characteristic zero) (\cite{Vigdorovich:arXiv}).  In particular, he deduces that these group {$C^*$}-algebras have stable rank one. In yet another direction, \cite{HEKR:arXiv} establishes selflessness (and hence strict comparison) of the {$C^*$}-reduced free products of semicircular systems (which appear in Problem \ref{FreeIsoQn} below).

This area is moving very quickly, and in the edits made to this paper in September 2025 Ozawa posted a preprint (\cite{Ozawa:selfless}) giving a dynamical approach to obtaining groups with selfless reduced {$C^*$}-algebras avoiding the use of rapid decay: infinite groups which admit a topologically free and minimal action on a compact Hausdorff space which is extremely proximal have selfless reduced group {$C^*$}-algebras. Such groups are all mixed identity free.  Getting away from the MIF condition, Ozawa also shows that the class of selfless and exact {$C^*$}-probability spaces is closed under the minimal tensor product  (\cite[Theorem~2]{Ozawa:selfless}) so that $C^*_r(\mathbb F_m\times \mathbb F_n)$ is selfless, for example. The following problem was slated to appear in the version of the paper posted in June 2025, with the motivation of proving that $C^*_r(\mathbb F_r\times \mathbb F_s)$ is pure.  Although Ozawa's methods now give pureness of $C^*_r(\mathbb F_r\times \mathbb F_s)$ through selflessness, the question of tensor products of pure $C^*$-algebras are pure remains open and natural.

\begin{question}\label{XCV}
Let $A$ and $B$ be {$C^*$}-algebras with $\Cu(A)\cong \Cu(B)\cong\Cu(\Z)$. Do we have $\Cu(A\otimes B)\cong \Cu(\Z)$?  More generally, is the minimal tensor product of pure {$C^*$}-algebras pure?
\end{question}

Even more recently, in a February 2026 preprint \cite{Vigdorovich:linear}, Vigdorovich showed that non-trivial linear groups with trivial amenable radical have selfless reduced $C^*$-algebra, which provides a positive answer to Problem \ref{q:Selfless} in the case of linear groups.

%

\section{Generation and isomorphism problems}\label{Sect:Gen}

Two of the longest standing problems in von Neumann algebras are the free group factor problem (whether $L(\mathbb F_2)$ and $L(\mathbb F_3)$ are isomorphic), which implicitly goes back to Murray and von Neumann (\cite{MurrayVonNeumann4}), and the generation problem: whether every separably acting von Neumann algebra is generated by a single operator (or equivalently, by two self-adjoint operators). The generation problem goes back at least to Kadison's Baton Rouge problem list from 1967; it holds for properly infinite von Neumann algebras, and the remaining II$_1$ case has been reduced to the setting of factors (see \cite{Topping:Lectures} for early work on this problem, and \cite{Wogan:BAMS,Willig:TMath} for the reduction to II$_1$ factors).  Here, single generation can very often be obtained from structural decompositions. Examples of singly generated II$_1$ factors include separably acting factors with Cartan subalgebras, with property $\Gamma$, and those which are tensorially non-prime  (\cite{GePopa:Duke}). These two problems are hopefully connected: ideally the free group factor $L(\mathbb F_\infty)$ would not be finitely generated, resolving both the generator and free group factor problems.

While separably acting abelian von Neumann algebras are singly generated, this is certainly not the case for separable abelian {$C^*$}-algebras. Yet surprisingly, the generation problem is open for unital simple separable {$C^*$}-algebras. (In a similar fashion to properly infinite von Neumann algebras, the stabilisations of  unital separable {$C^*$}-algebras are singly generated; see \cite[Theorem 8]{OlsenZame:TAMS}).

\begin{question}
Is every unital simple separable {$C^*$}-algebra generated by a single operator? More generally, is every unital separable and nowhere scattered {$C^*$}-algebra generated by a single operator?
\end{question}

 Thiel and Winter gave a positive answer to the generator problem for unital separable $\Z$-stable {$C^*$}-algebras in \cite{ThielWinter:TAMS}.\footnote{This generalised much earlier work of Olsen and Zame for unital separable UHF-stable {$C^*$}-algebras (\cite{OlsenZame:TAMS}).  In both cases (and also in Olsen and Zame's work on stabilisations of separable unital {$C^*$}-algebras), the unital assumption is used in the functional calculus arguments to show generation. In the non-unital case, Thiel and Winter use the minimal $\Z$-stable unitisation to deduce that general separable $\Z$-stable {$C^*$}-algebras are generated by at most $3$ self-adjoint operators; to the best of our knowledge, it is not known whether this can be reduced to $2$.}  In fact, they obtain single generation for tensor products\footnote{Thiel and Winter work with the maximal tensor product, so their result holds for any separable tensor product.} $A\otimes B$ when $A$ is unital and has enough orthogonal full positive elements (e.g.\ if it is simple and non-elementary) and $B$ contains a unital copy of $\Z$.  This includes $C(X,\mathcal Z)$ for any compact metrisable $X$ -- another instance where tensoring by $\Z$ reduces dimension in some sense. Further work of Thiel shows that unital simple separable $\Z$-stable {$C^*$}-algebras have a dense set of generators, when they have real rank zero or are ASH (\cite{Thiel:CJM,Thiel:JNCG}). In fact, Thiel's study of when the set of generators is dense arose in \cite{Thiel:JFA} to show that this property holds for AF algebras -- and in particular, that (non-simple) AF algebras are always singly generated.
To the best of our knowledge, there are no general results for {$C^*$}-algebras with other decomposition properties analogous to von Neumann results: for example, under what circumstances is the presence of a Cartan subalgebra enough for single generation of a separable simple {$C^*$}-algebra? What about the combination of strict comparison and uniform property $\Gamma$ (these conditions are equivalent to $\Z$-stability in the nuclear case but perhaps offer a direction of travel for a {$C^*$}-version of single generation of II$_1$ factors with property $\Gamma$ outside it, with strict comparison providing some minimal amount of regularity)? Another potentially interesting collection of examples are the reduced group {$C^*$}-algebras $C^*_r(SL_n(\mathbb Z))$ for $n\geq 3$ as their von Neumann completions $L(SL_n(\mathbb Z))$ are known to be singly generated (\cite{GeShen:PNAS}).  

We also highlight that, unlike the von Neumann situation, where the hyperfinite II$_1$ factor is certainly singly generated, the generator problem is equally open in the nuclear case. Just as Villadsen found ways of building exotic simple nuclear {$C^*$}-algebras replicating certain commutative phenomena in the simple setting, perhaps one can construct counterexamples to the generator problem in a similar vein. 

\begin{question}
Is every unital simple separable nuclear {$C^*$}-algebra generated by a single operator?
\end{question}

Of course, one hopes that the reduced group {$C^*$}-algebras associated to free groups with infinitely many generators should not be finitely generated.  Formally, this should be easier to show than the infinite generation of $L(\mathbb F_\infty)$, but this difference may be a mere formality; such a result appears to be extremely challenging. 

Finally, whereas the free group factor problem (whether $L(\mathbb F_2)$ and $L(\mathbb F_3)$ are isomorphic) is open, the directly analogous problem for {$C^*$}-algebras is not: $C^*_r(\mathbb F_2)$ is distinguished from $C^*_r(\mathbb F_3)$ by their $K_1$-groups (\cite[Corollary~3.2]{PimsnerVoiculescu:JOT}).
This is no longer the case for reduced free products of the form $\mathcal Z^{*k}$ (with respect to the unique trace) and  $C([0,1])^{* k}$ (with respect to the trace $\tau_{\mathrm{Leb}}$ induced by the Lebesgue measure)\footnote{This {$C^*$}-algebra is isomorphic to the reduced free product of $k$ semicircular elements, and so by the very recent preprint \cite{HEKR:arXiv} is now known to have strict comparison.} by $K$-theory computations of Germain (\cite[Corollary~2.6]{Germain:Crelle} and \cite[Theorem~4.1]{Germain:Duke}; the case of $C([0,1],\tau_{\mathrm{Leb}})^{*k}$ is addressed directly in \cite[Corollary 6.1]{Germain:Duke}).
The following may therefore be regarded as a `correct' version of the free group factor problem for {$C^*$}-algebras.

\begin{question}\label{FreeIsoQn}
    Determine whether any of the  {$C^*$}-algebras
    \begin{equation}
  (\mathcal Z,\tau)^{*k}\text{ and }(C([0,1]),\tau_{\mathrm{Leb}})^{* \ell}\text{ for }k,\ell\geq 2,       
    \end{equation}are isomorphic.
\end{question}

\subsection{Addendum April 2026} In (\cite{HP:arXiv}), Hirshberg and Phliips have resolved Problem \ref{FreeIsoQn} in the case when $k=\ell=\infty$, showing that $(\mathcal Z,\tau)^{*\infty}\cong (C([0,1]),\tau_{\mathrm{Leb}})^{* \infty}$. These techniques crucially use the infinite free product, and whether such an isomorphism holds for finite $k=\ell$ looks like a very intriguing question.

\section{Algebras close to $\Z$-stable algebras}

We end with a problem that did not fit anywhere else but has irritated one of the authors since they learnt the definition of $\Z$-stability.  In \cite{KadisonKastler}, Kadison and Kastler examined the metric space of all operator algebras acting on a fixed Hilbert space equipped with the Hausdorff metric on their unit balls, suggesting that sufficiently close algebras should be (spatially) isomorphic.  Positive answers have been obtained when one algebra is an injective von Neumann algebra (\cite{Christensen:Acta}), when one algebra is a separable nuclear {$C^*$}-algebra (\cite{CSSWW:Acta}), and -- outside the amenable setting -- when one algebra is the $\mathcal R$-stabilisation of a von Neumann crossed product associated to an action of a group with suitable cohomological properties (\cite{CCSSWW:DJM}). (For the long history, we refer to the introductions to these papers).  In the last of these, a necessary step is to show that any II$_1$ factor sufficiently close to a separably acting McDuff factor is again McDuff.
The following asks for the {$C^*$}-analogue of this.

\begin{question}\label{Q:PerturbZ}
Does there exist $\epsilon>0$ such that if $A\subset\mathcal B(\mathcal H)$ is a separable $\Z$-stable {$C^*$}-algebra and $B\subset\mathcal B(\mathcal H)$ is another {$C^*$}-algebra with $d_{\mathrm{Hausdorff}}(A_1,B_1)<\epsilon$, then $B$ is necessarily $\Z$-stable?
\end{question}

In \cite{CCSSWW:DJM}, it is further shown that, for a II$_1$ factor $\mathcal N$ close to a separably acting McDuff factor $\mathcal M\coloneqq \mathcal M_0\bar\otimes\mathcal R$, one can perform a small\footnote{That is, $\|u-1\|$ is small.} spatial perturbation $u\mathcal Nu^*$ of $\mathcal N$ so that $u\mathcal Nu^*$ has a McDuff decomposition using the same copy of $\mathcal R$ as in that for $\mathcal M$.  We would not expect a {$C^*$}-analogue of this,\footnote{The examples from  \cite{Johnson:CMB} suggest that this is unlikely -- for separable nuclear {$C^*$}-algebras, such a spatial perturbation cannot be done with the compacts in place of $\mathcal Z$.} though it is conceivable if one asks for the spatial perturbation to be point-norm small.\footnote{That is, the commutators $[u,x]$ can be taken small for finite sets of operators $x$.}

Although we have focused on $\Z$ in Problem \ref{Q:PerturbZ}, nothing is known if $\Z$ is replaced by another strongly self-absorbing {$C^*$}-algebra such as $M_{2^\infty}$ or $\mathcal O_\infty$.  At the level of the Cuntz semigroup, there are positive results. Firstly, {$C^*$}-algebras sufficiently close to simple purely infinite {$C^*$}-algebras are again simple and purely infinite; see the proof of \cite[Theorem 6.4]{CSSW:GAFA}, which shows that algebras close to Kirchberg algebras are again Kirchberg.\footnote{The approach taken there is to show that {$C^*$}-algebras close to those with real rank zero again have real rank zero, and likewise with the property of every non-zero projection being infinite. The result is then obtained from Zhang's characterisation that simple {$C^*$}-algebras are purely infinite if and only if they have real rank zero and all non-zero projections are infinite (the fact that simplicity transfers to close subalgebras goes back to \cite{Phillips:IUMJ}).  It is possible to obtain better constants by a direct argument (such a possibility is hinted at in \cite{CSSW:GAFA}).}  More generally, provided two {$C^*$}-algebras $A$ and $B$ have sufficiently close stabilisations $A\otimes\mathcal K$ and $B\otimes\mathcal K$, then $\Cu(A)\cong\Cu(B)$ and so in particular $A$ is pure if and only if $B$ is pure (see \cite[Section 3]{PTWW:APDE}). A priori, stabilisations are needed to be able to see all of the Cuntz semigroup but in the presence of a solution to Kadison's similarity problem, one can get an estimate on the distance between $A\otimes\mathcal K$ and $B\otimes\mathcal K$ in terms of the original distance between $A$ and $B$ (see \cite{CSSW:GAFA}).  Of relevance to Problem \ref{Q:PerturbZ} is that if $A$ is $\mathcal Z$-stable then it will satisfy Kadison's similarity property (see \cite[Corollary 4.9]{PTWW:APDE}), and any {$C^*$}-algebra $B$ close enough to $A$ will have an isomorphic Cuntz semigroup and so be pure. These ideas were used in \cite{PTWW:APDE} to investigate stability (as stable {$C^*$}-algebras, and more generally those with no bounded traces, satisfy Kadison's similarity problem). In the presence of  weak cancellation, stability is seen in the Cuntz semigroup together with its natural scale via work of Hjelmborg and R\o{}rdam (\cite{HR:JFA}; see \cite[Lemma~4.6]{PTWW:APDE}).  Using stable rank one to obtain weak cancellation, it follows that if $A$ is close enough  to a stable {$C^*$}-algebra of stable rank one, then $A$ must be stable (\cite[Theorem 4.7]{PTWW:APDE}). We would be interested both in a general perturbation result for stability, without any hypotheses to ensure weak cancellation, or progress on Problem \ref{Q:PerturbZ} in the presence of stable rank one.

\end{document}